\documentclass[11pt]{amsart}
\pdfoutput=1

\usepackage{amssymb}
\usepackage{amsthm}
\usepackage{mathrsfs}
\usepackage{amsfonts}
\usepackage{amsmath,amssymb}
\usepackage{pmboxdraw}
\usepackage{verbatim} 
\usepackage{graphicx}
\usepackage{color}
\usepackage[colorlinks=true, citecolor=blue, filecolor=black, linkcolor=black, urlcolor=black]{hyperref}
\usepackage[normalem]{ulem}
\usepackage{enumitem}
\usepackage{graphicx}
\usepackage{subcaption}
\setlist{leftmargin=15pt,labelindent=15pt}
\setlist[enumerate]{align=left}

\usepackage{mathtools}
\makeatletter
\newcommand*\bigcdot{\mathpalette\bigcdot@{.5}}
\newcommand*\bigcdot@[2]{\mathbin{\vcenter{\hbox{\scalebox{#2}{$\m@th#1\bullet$}}}}}
\makeatother

\newcommand{\RN}[1]{%
	\textup{\uppercase\expandafter{\romannumeral#1}}%
}

\def\pa{\partial}

\def\ve{\varepsilon}

\def\Re{ \mathrm{Re}}
\def\Im{ \mathrm{Im}}

\def\C{\mathbb{C}}

\def\R{\mathbb{R}}

\def\erf{\textup{erf}}

\newcommand{\bfE}{\mathbf{E}}

\newcommand{\SC}{\sigma_{\mathrm{sc}}}
\newcommand{\MP}{\sigma_{\mathrm{MP}}}
\newcommand{\MPa}{\sigma_{\mathrm{MP}(\alpha)}}

\newcommand{\bfR}{\mathbf{R}}

\newcommand{\calW}{{\mathcal W}}
\newcommand{\bfB}{{\mathbf B}}
\newcommand{\bfP}{{\mathbf P}}

\newcommand{\bfK}{{\mathbf K}}

\newcommand{\Ai}{\operatorname{Ai}}

\newcommand{\erfc}{\operatorname{erfc}}

\newcommand{\fii}{{\varphi}}

\newcommand{\eps}{{\varepsilon}}

\newcommand{\re}{\operatorname{Re}}
\newcommand{\im}{\operatorname{Im}}

\newcommand{\Ham}{H}

\newcommand{\Prob}{{\mathbf{P}}}

\renewcommand{\d}{{\partial}}
\newcommand{\dbar}{\bar{\partial}}
\newcommand{\1}{\mathbf{1}}

\newcommand{\dist}{\operatorname{dist}}
\newcommand{\supp}{\operatorname{supp}}

\newcommand{\Lap}{\Delta}

\theoremstyle{plain}
\newtheorem*{thm*}{Theorem}
\newtheorem{thm}{Theorem}[section]
\newtheorem{lem}[thm]{Lemma}

\theoremstyle{definition}

\newtheorem*{egs*}{Examples}
\newtheorem*{def*}{Definition}

\theoremstyle{remark}
\newtheorem*{eg*}{Example}
\newtheorem*{not*}{Notation}
\newtheorem*{rmk*}{Remark}
\newtheorem*{rmks*}{Remarks}

\numberwithin{equation}{section}

\makeatletter
\@namedef{subjclassname@2020}{\textup{2020} Mathematics Subject Classification}
\makeatother

\begin{document}
\title[The bandlimited Coulomb gas]{Almost-Hermitian random matrices and bandlimited point processes}

\author{Yacin Ameur}

\address{Yacin Ameur\\
Department of Mathematics\\
Faculty of Science\\
Lund University\\
P.O. BOX 118\\
221 00 Lund\\
Sweden}
\email{Yacin.Ameur@math.lu.se}

\author{Sung-Soo Byun}
  \address{Center for Mathematical Challenges, Korea Institute for Advanced Study, 85 Hoegiro, Dongdaemun-gu, Seoul 02455, Republic of Korea}


  \email{sungsoobyun@kias.re.kr}



\thanks{Sung-Soo Byun was supported by Samsung Science and Technology Foundation (SSTF-BA1401-51), by a KIAS Individual Grant (SP083201) via the Center for Mathematical Challenges at Korea Institute for Advanced Study, by the National Research Foundation of Korea (NRF-2019R1A5A1028324), and by the POSCO TJ Park Foundation (POSCO Science Fellowship).}

\begin{abstract} We study the distribution of eigenvalues of almost-Hermitian random matrices associated with the classical Gaussian and Laguerre unitary ensembles. In the almost-Hermitian setting, which was pioneered by Fyodorov, Khoruzhenko and Sommers in the case of GUE, the eigenvalues are not confined to the real axis, but instead have imaginary parts which vary within a narrow ``band'' about the real line, of height proportional to $\tfrac 1 N$, where $N$ denotes the size of the matrices.

We study vertical cross-sections of the 1-point density as well as microscopic scaling limits, and we compare with other results which have appeared in the literature in recent years. Our approach uses Ward's equation and a property which we call ``cross-section convergence'', which relates the large-$N$ limit of the cross-sections of the density of eigenvalues with the equilibrium density for the corresponding Hermitian ensemble: the semi-circle law for GUE and the Marchenko-Pastur law for LUE. As an application of our approach, we prove the bulk universality of the almost-circular ensembles.
\end{abstract}

\maketitle

\section{Introduction} \label{intro}

This note is the result of investigations of eigenvalue ensembles in the almost-Hermitian regime. In particular we study almost-Hermitian counterparts of the Gaussian Unitary Ensemble and of the Laguerre Unitary Ensemble (in the singular case), which we call
\emph{Almost-Hermitian GUE (AGUE)} and \emph{Almost-Hermitian LUE (ALUE)}, respectively.

Almost-Hermitian (or ``weakly non-Hermitian'') random matrices were introduced by Fyodorov, Khoruzhenko, and Sommers in the papers \cite{FKS,FKS2,FKS3}. Their work concerns eigenvalues of elliptic Ginibre matrices where the droplets collapse to the interval $[-2,2]$ at a suitable rate. We take this as our model
for almost-Hermitian GUE. The emergent
structure is physically interesting, and has been subject of several investigations in recent years, see in particular \cite{ACV} and the references there.

As is to be expected, the AGUE tends to follow Wigner's semi-circle law, and the ALUE tends to follow the Marchenko-Pastur law.
However, in the almost-Hermitian regime,
the particles are not confined to the real axis,
but instead vary randomly within a thin band about $\R$ of height proportional to $\tfrac 1 N$, where $N$ denotes the size of the matrices. See Figure \ref{fig0}.

In both cases, the droplets are similar, highly eccentric elliptic discs.
Yet, the particle distributions are quite different, with heavy clustering  going on
near the left edge for ALUE, in contrast with the much more evenly spread configurations for AGUE (compare Figure \ref{Fig_ANWB}). This reflects the singular behaviour of the critical Marchenko-Pastur density.

On the other hand, we shall prove that there is a universal microscopic behaviour in the bulk, given by a kernel found in \cite{FKS,FKS3}, which interpolates between
the sine-kernel and the Ginibre kernel.

We shall introduce a class of bandlimited ensembles in the interface between dimensions 1 and 2, which we use to formulate and prove certain theorems of a general character. However, ultimately, we shall resort to asymptotic properties of the respective orthogonal polynomials (Hermite and Laguerre) in order to deduce full universality
results.

Our microscopic analysis uses the theory for Ward's equation from the paper \cite{AKM}. In particular, we shall find a close connection between the scaling limits in \cite{ACV,FKS,FKS3} and the translation invariant solutions to Ward's equation, which are characterized in \cite{AKM}. (We also refer to a recent work \cite{ABK2} for an implementation of Ward's equation in the study of translation invariant scaling limits of planar symplectic ensembles.)

\begin{figure}[h]
		\begin{center}
			\includegraphics[width=4.5in,height=2.415in]{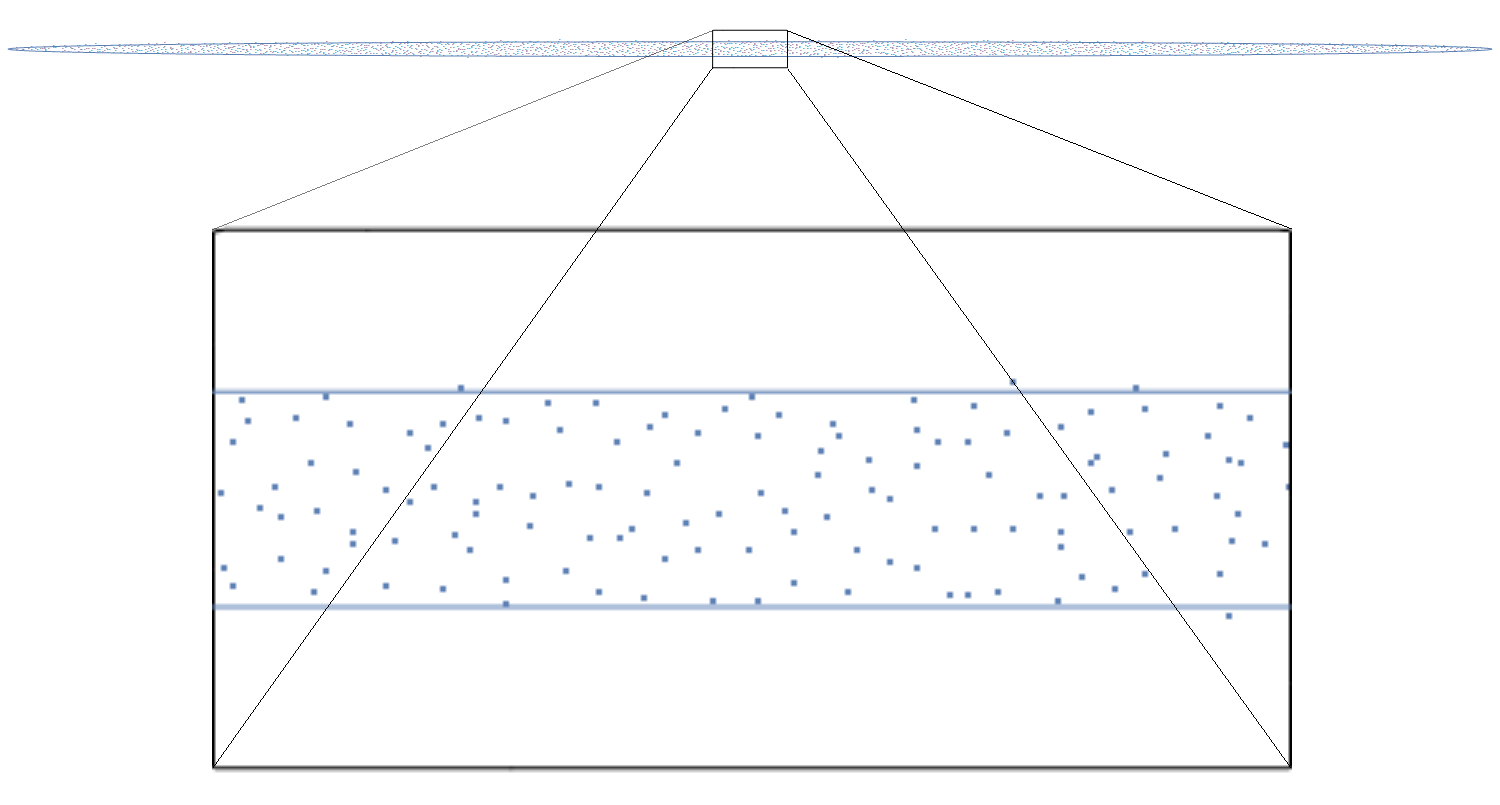}	
		\end{center}
		\caption{A sample from the Almost-Hermitian GUE.}
\label{fig0}
\end{figure}

\subsubsection*{Notational conventions.}
A typical point in the complex plane $\C$ is denoted $\zeta=\xi+i\eta$.

We write $\d=\tfrac 12(\tfrac \d {\d \xi}-i\tfrac \d {\d \eta})$ and $\dbar=\tfrac 1 2 (\tfrac \d
{\d\xi}+i\tfrac \d {\d\eta})$ for the complex derivatives and
$\Delta=\d\dbar=\tfrac 1 4(\tfrac {\d^2}{\d \xi^2}+\tfrac {\d^2}{\d \eta^2})$ for $\tfrac 1 4$ times the usual Laplacian.

We write $dA(\zeta)=\tfrac 1 \pi\, d^{\,2}\,\zeta=\tfrac 1 \pi\,d\xi\,d\eta$ for $\tfrac 1 \pi$ times Lebesgue measure on $\C$ and
$dA_N=(dA)^{\otimes N}$ for the corresponding volume measure in $\C^N$.
When $\mu$ is a measure and $f$ is a function, we write $\mu(f)$ for $\int f\, d\mu$.

\subsection{Purpose and aims} The theory of almost-Hermitian ensembles consists at this point largely of a number of important model cases, which are typically studied using properties of specific orthogonal polynomials.
One could say that the present work grew out of a desire to classify and provide a ``unifying structure'' to this theory.

The first question is how to suitably define
a general class of ``almost-Hermitian ensembles'' which has desired potential theoretic properties. This question does not depend on a determinantal structure, and thus is more natural to study
in a context of $\beta$-ensembles. We shall isolate a suitable class of potentials and establish several new results concerning convergence to the equilibrium. This convergence forms the starting
point for our further analysis of almost-Hermitian ensembles. (This approach is analogous with established practices in dimensions 1 and 2, and could in that context be viewed as a first step towards a Gaussian field theory in the almost-Hermitian case).

We also study (and survey) various types of microscopic limits in the almost-Hermitian regime, which typically interpolate between well-known limits in dimensions 1 and 2. In our approach using the
loop equation and compactness, questions about particular scaling limits are reduced to easier questions, such as establishing their apriori translation invariance.

That being said, with the exception of the so-called ``almost-circular ensembles'' below, establishing translation invariance is in general a subtle matter. While one might expect that translation invariance should hold quite generally in the bulk, we will not resolve that issue here. (The problem that arises concerns uniqueness of solution to Ward's equation given certain apriori conditions, and is closely related with the one discussed in \cite{AKM}.)

Instead of striving for largest possible generality, we take a detailed look at the model cases, where the analysis is already non-trivial and offers new angles. Apart from leading to new proofs of results which were rigorously established only fairly recently, our
analysis has some new ingredients that have led to spin-offs in for example \cite{ABK2}.

For comparison, the universality problem in the normal matrix setting was settled only recently in a fairly general setting in \cite{HW}; 
the techniques used there break down in the almost-Hermitian setting. (For example, the conformal mappings of the exterior of a thin droplet to the exterior disc have large derivatives at the end-points: the derivatives tend to infinity as the droplets collapse to an interval.)

\subsection{General setup} \label{Setup}

Given an $N$-point configuration  $\{\zeta_j\}_1^N\subset \C$ and a suitable function (external potential) $Q_N:\C\to\R\cup\{+\infty\}$, we associate the total energy
\begin{equation}
\label{hamn}\Ham_N=\sum_{j\ne k}\log\frac 1 {|\zeta_j-\zeta_k|}+N\sum_{j=1}^N Q_N(\zeta_j).
\end{equation}

We next fix a positive parameter $\beta$ and form
the canonical Gibbs measure
\begin{align}\label{prob}d\Prob_N^\beta=\tfrac 1 {Z_N^\beta}e^{-\beta \Ham_N}\, dA_N, \qquad (Z_N^\beta=\int_{\C^N}e^{-\beta\cdot \Ham_N}\, dA_N).
\end{align}

We assume that
the potentials $Q_N$ increase monotonically, as $N\to\infty$, to an external potential $V$ which obeys
$V(\zeta)=+\infty$ when $\zeta\not\in \R$.

Here and throughout, ``external potential'' means a lower semicontinuous function $W:\C\to\R\cup\{+\infty\}$ which is finite on some set of positive capacity and satisfies
\begin{equation}\label{uniform}\liminf_{\zeta\to\infty}\tfrac {W(\zeta)}{2\log|\zeta|}\ge k\end{equation}
where $k>1$ is a fixed constant. We shall assume that this holds, with the same $k$, for all $W=Q_N$ and for $W=V$.
Then, to a compactly supported Borel probability measure $\mu$ on $\C$, we associate the logarithmic $W$-energy
\begin{equation}\label{Ilog}I_W[\mu]=\iint_{\C^2}(\log\tfrac 1 {|\zeta-\eta|})\, d\mu(\zeta)\, d\mu(\eta)+\mu(W).\end{equation}
In addition to \eqref{uniform}, we shall sometimes assume the extra growth condition
\begin{equation} \label{extra growth}
	W(\zeta) \ge (1+\delta) \log (1+|\zeta|^2) -C \qquad \zeta \in \C
\end{equation}
for some small $\delta >0 $ and $C \in \R$.

By \cite[Theorem I.1.3]{ST}, there exists a unique equilibrium measure $\sigma_W$, i.e., a compactly supported Borel probability measure which minimizes the
energy \eqref{Ilog}. The support $S_W=\supp\sigma_W$ is called the \textit{droplet} in external potential $W$.

We shall always assume that each $Q_N$ is smooth in a neighbourhood of the droplet $S_{Q_N}$, except possibly for some singular points at which $\Delta Q_N=+\infty$. The singularities are
 assumed to be benign in the sense that the basic structure theorem for equilibrium measures in \cite[Theorem II.1.3]{ST} applies. Namely, the equilibrium measure $\sigma_{Q_N}$ is absolutely continuous with respect to $dA$ and
$d\sigma_{Q_N}=\Delta Q_N\cdot \1_{S_{Q_N}}\, dA.$

Finally, we assume that the equilibrium measure $\sigma_V$, which is supported in $\R$ (and equals to the weak limit of the measures $\sigma_{Q_N}$, see Subsection \ref{tld}), is
 absolutely continuous with respect to Lebesgue measure on $\R$. By slight of abuse of notation, we denote its linear density by $\sigma_V(\xi)$ (so $d\sigma_V(\xi)=\sigma_V(\xi)\, d\xi$).
The probability density $\sigma_V(\xi)$ is assumed to be continuous on $\R$, again with the possible exception of finitely many singular points.

\subsection{Limiting droplet and cross-sections}
Denote by $\{\zeta_j\}_1^N$ a random sample with respect to the Boltzmann-Gibbs measure \eqref{prob}, and write the 1-point function as
\begin{equation*}\bfR_N^{\,\beta}(\zeta)=\lim_{\eps\to 0}\frac {\bfE_N^{\,\beta}(\# (\{\zeta_j\}_1^N\cap D(\zeta,\eps)))}{\eps^2}.\end{equation*}
($D(\zeta,\eps)$ is the open disc with center $\zeta$ and radius $\eps$.)

Now fix an arbitrary zooming-point $p\in\R$ and define the blow-up map
\begin{equation}\label{blowup}\Gamma_{N,p}:\zeta\mapsto z,\qquad z=\Gamma_{N,p}(\zeta)={\scriptstyle \sqrt{N\Delta Q_N(p)}}\cdot (\zeta-p).\end{equation}

Let $\{z_j\}_1^N$ be the rescaled sample, $z_j=\Gamma_{N,p}(\zeta_j)$.

The potentials $Q_N$ used in this note have the property that the Laplacian $\Delta Q_N(p)$ is proportional to $N$.
More precisely, we shall assume throughout that the limit
\begin{equation}\label{rhop}\rho(p)=\lim_{N\to\infty}\sqrt{\tfrac N {\Delta Q_N(p)}}\end{equation}
is well-defined and finite for each $p\in S_V$.
We next form the function
\begin{equation}\label{ap}a(p)=\tfrac \pi 2 \cdot\rho(p)\cdot \sigma_V(p),\qquad p\in\R\end{equation}
with the understanding that $a(p)=0$ if $p\not\in S_V$.
The function \eqref{ap} has the following
geometric interpretation: let
\begin{equation}\label{nuber}
\gamma_N(p)=\Gamma_{N,p}(S_{Q_N})\cap (i\R)
\end{equation}
be the cross-section of the rescaled droplet with the imaginary axis. If we impose the symmetry $Q_N(\bar{\zeta})=Q_N(\zeta)$, and some other natural conditions (Subsection \ref{tld}), the number $a(p)$ represents the \emph{limiting height of the rescaled cross-section}, i.e.,
\begin{equation}\label{aplength}a(p)=\tfrac 1 2 \lim_{N\to\infty}|\gamma_N(p)|, \qquad p\in\R.\end{equation}

We shall also use the \textit{statistical} cross-sections $c_N^\beta$ of the ensemble,
\begin{equation}\label{decs}c_N^\beta(p)=\tfrac 1 N\int_\R\bfR_N^\beta(p+i\eta)\, d\eta,\qquad p\in\R.\end{equation}
A fairly general convergence result, asserting that (under some additional assumptions) $c_N^\beta\to \pi\sigma_V$ in the sense of measures on $\R$, is given in Theorem \ref{joh1} below.

The $1$-point function of the rescaled system $\{z_j\}_1^N$ is denoted by
\begin{equation}\label{rescaled}R_N^{\,\beta}(z)= \tfrac 1 {N\Delta Q_N(p)}\bfR_N^{\,\beta}(\zeta),\qquad z=\Gamma_N(\zeta).\end{equation}

We are interested in describing all possible limits $R^{\,\beta}=\lim_{N\to\infty}R_N^{\, \beta}$, and we shall here restrict to the
determinantal case when $\beta=1$. We shall therefore in the sequel drop the superscript, writing $R_N$ in place of $R_N^{\, 1}$, etc.

We recall the following fact from the theory of determinantal Coulomb gas processes (see e.g.~\cite[Lemma 1]{AKMW} and the references there).

\begin{lem}\label{Lem_infinite DPP}
 If $R_N\to R$ in $L^1_{\mathrm{loc}}$ (along some subsequence) then $\{z_j\}_1^N$ converges (along the same subsequence) in the sense of point fields to a unique infinite determinantal
point-field $\{z_j\}_1^\infty$ with $1$-point function $R$.
\end{lem}

With these preliminaries out of the way, we turn to our main objects of study.

\subsection{Almost-Hermitian GUE} \label{AGUEdef}

Fix a positive parameter $c$ and consider the sequence of potentials
\begin{equation}\label{ellipse}Q_N(\zeta)=\tfrac 1 2 \xi^{\,2}+\tfrac 1 2 \tfrac{N}{c^{2}}\eta^{\,2},\qquad (\zeta=\xi+i\eta).\end{equation}

The droplet in potential \eqref{ellipse} can be found by means of the following useful lemma from \cite{Girko,SCSS}. (The lemma has an alternative proof by solving
an obstacle problem, which may be left for the interested reader.)

\begin{lem} \label{dropl} The droplet $S_Q$ in potential $Q=a\xi^2+b\eta^2$ is the elliptic disc with equation
$\tfrac {a^2+ab}{2b}\,\xi^{\,2}+\tfrac {ab+b^2}{2a}\,\eta^2\le 1.$
\end{lem}

Using the lemma, we find that the droplet in potential \eqref{ellipse} is given by
\begin{equation}\label{potth}S_{Q_N}=\{\,\xi+i\eta\,;\,(1+\tfrac{c^2}{N})\xi^2+\tfrac{N}{c^2}(1+\tfrac {N} {c^2})\eta^{2}\le 4\,\}.\end{equation}

Note that $Q_N\nearrow V$ where $V$ is the Gaussian potential
\begin{equation}\label{GP}V(\xi)=\tfrac 1 2 \xi^{\,2},\qquad (\xi\in\R),\end{equation}
with the understanding that $V=+\infty$ on $\C\setminus \R$. The equilibrium measure in potential $V$ is Wigner's semi-circle law (\cite[Section 1.4]{F}),
\begin{equation}\label{SC}\SC(\xi)=\tfrac 1 {2\pi}\sqrt{4-\xi^{\,2}}\cdot \1_{[-2,2]}(\xi).\end{equation}

Let $c_N(\xi)=\tfrac 1 N\int_\R \bfR_N(\xi+i\eta)\, d\eta$ be the $N$:th cross-section (with $\beta=1$). The
following global result can be found in \cite[Theorem 3 (a)]{ACV} (cf.~\cite{FKS2,L}) and will here be reproved by different methods.

\begin{thm} \label{mth0} (``\emph{Pointwise cross-section convergence for AGUE}'') $c_N(\xi)\to\pi\cdot \SC(\xi)$ as $N\to\infty$ for each $\xi$ with $-2<\xi<2$.
\end{thm}

\begin{rmk*}
We shall use Theorem \ref{mth0} as a basic tool for our microscopic investigations. In fact the pointwise convergence $\tfrac 1 \pi c_N(\xi)\to\SC(\xi)$ for $-2<\xi<2$ is precisely what we need to
uniquely fix a translation invariant bulk scaling limit using our method below. This contrasts with the strategy in \cite{ACV}, where convergence of cross-sections is obtained as a consequence
of microscopic investigations.
\end{rmk*}

Following the terminology in \cite{AKM}, we put
$\gamma(z)=\tfrac 1 {\sqrt{2\pi}}e^{-\frac 1 2 z^2}$.
Given a Borel subset (or ``window'') $E\subset\R$ we then consider the  entire function
\begin{equation}\label{convolution}\gamma *\1_E(z)=\tfrac 1 {\sqrt{2\pi}}\int_Ee^{-\frac 1 2(z-t)^2}\, dt.\end{equation}

Now fix a point $p_*\in S_V$ in the bulk, i.e., $-2< p_*< 2$, and let $R_N$ be the $1$-point function of the rescaled process about $p_*$.

Note that for the potential \eqref{ellipse}, the numbers $\rho(p_*)$ and $a(p_*)$ in \eqref{rhop}, \eqref{aplength} reduce to
\begin{equation}\label{appl}\rho(p_*)=\sqrt{\tfrac N {\Delta Q_N}}=2c,\qquad a(p_*)=\pi\cdot c\cdot \SC(p_*).\end{equation}

The following theorem is equivalent with the bulk scaling limits found in \cite{ACV,FKS,L}.

\begin{thm} \label{mainth1} (``\emph{Bulk scaling limit for AGUE}'') Under the above assumptions, $R_N$ converges locally uniformly to the limit $R(z)=F(2\im z)$ where
$F(z)=\gamma*\1_{(-2a,2a)}(z)$ and $a=a(p_*)$ is given by \eqref{appl}.
\end{thm}

\begin{figure}[h!]
	\begin{subfigure}[h]{0.32\textwidth}
		\begin{center}
			\includegraphics[width=1.69in,height=1.36in]{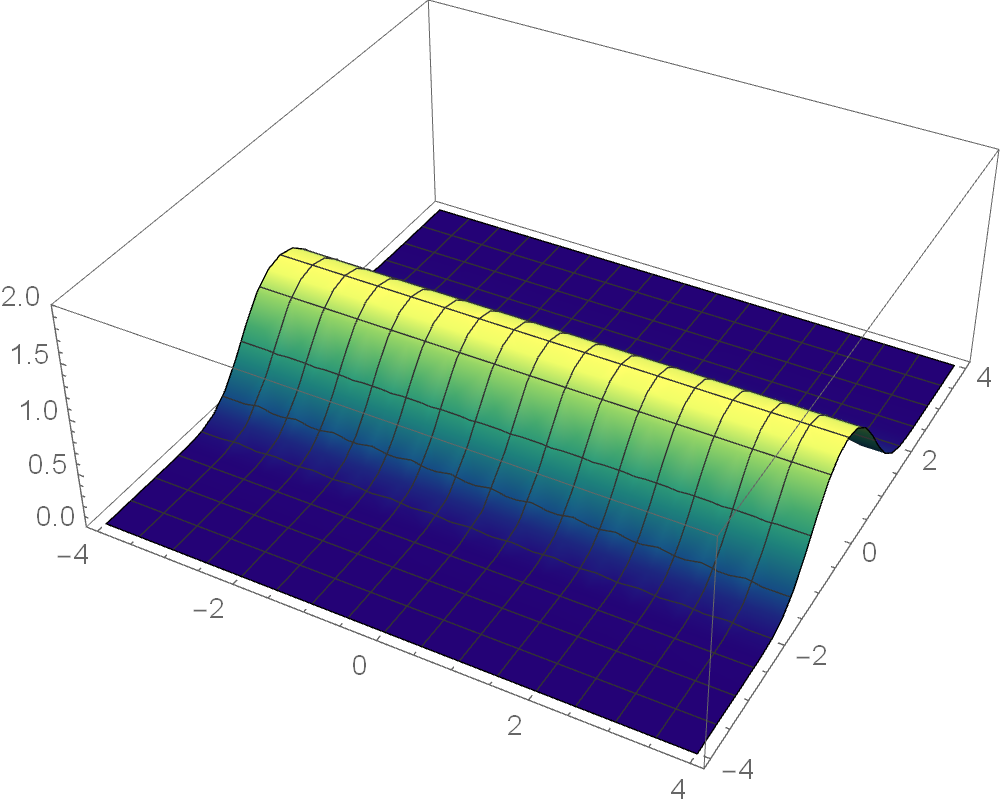}	
		\end{center}
		\caption{$c=1$}
	\end{subfigure}
	\begin{subfigure}[h]{0.32\textwidth}
		\begin{center}
			\includegraphics[width=1.69in,height=1.36in]{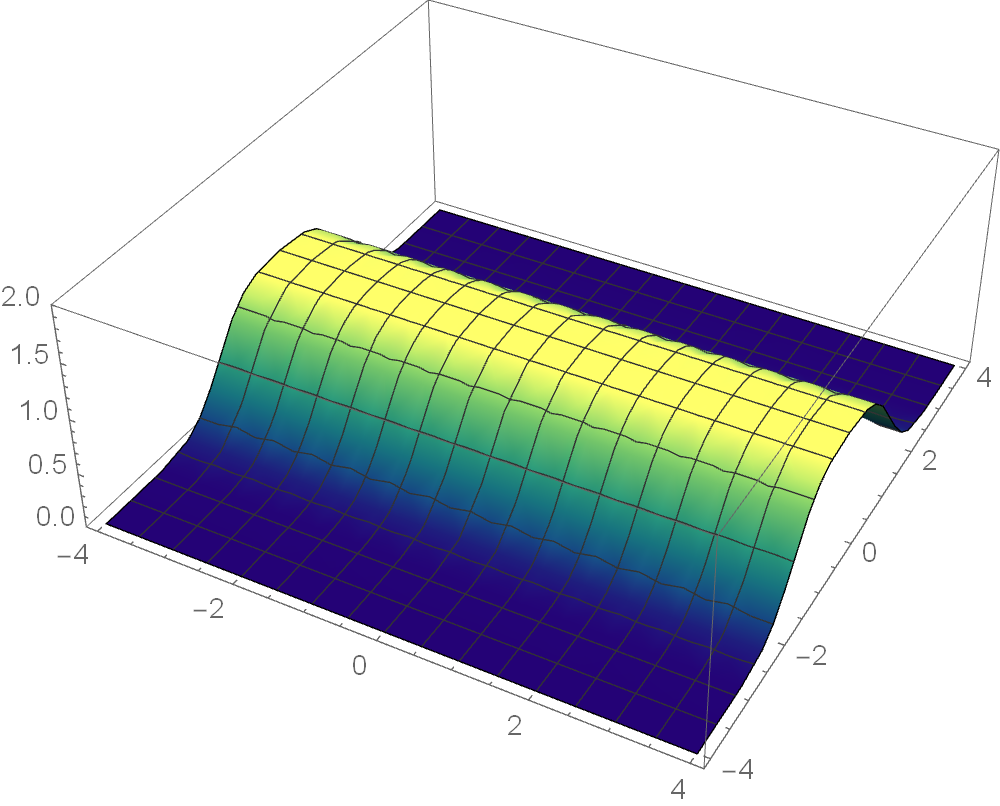}
		\end{center}
		\caption{$c=1.5$}
	\end{subfigure}	
	\begin{subfigure}[h]{0.32\textwidth}
		\begin{center}
			\includegraphics[width=1.69in,height=1.36in]{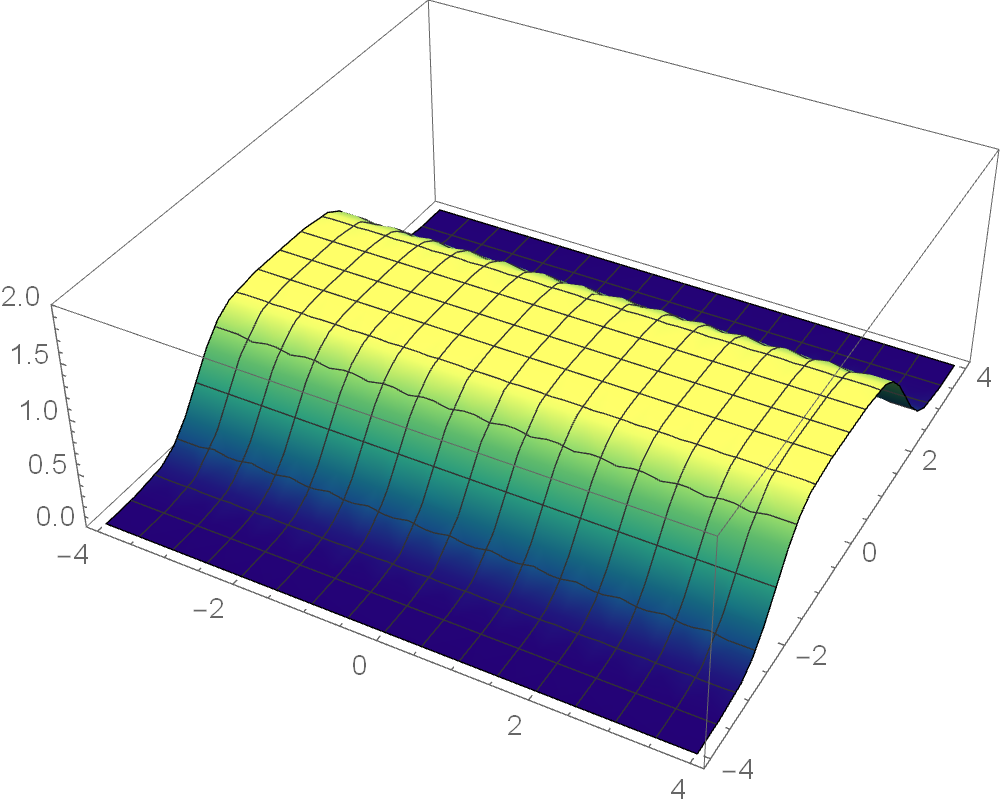}
		\end{center}
		\caption{$c=2$}
	\end{subfigure}	
	\caption{The limiting rescaled $1$-point density $R(z)$ about $p_*=0$. }
\label{Fig_FKS R}
\end{figure}

We now discuss how the scaling limits in Theorem \ref{mainth1} interpolate between the sine-kernel and the Ginibre kernel.

We recall (e.g.~\cite{AKM}) that the Ginibre kernel is the function
\begin{equation}\label{gin0}G(z,w)=e^{z\bar{w}-\frac 1 2|z|^2-\frac 1 2|w|^2},\end{equation}
and that its associated 1-point function $R(z)=G(z,z)=1$ is characteristic for the infinite Ginibre ensemble.

Now let $F=\gamma*\1_{(-2a,2a)}$ be the entire function appearing in Theorem \ref{mainth1}
and put
\begin{equation}\label{ulg2}K(z,w)=G(z,w)\cdot F(\tfrac {z-\bar{w}} i)=G(z,w)\cdot \tfrac 1 {\sqrt{2\pi}}\int_{-2a}^{\,2a}e^{\frac 1 2 (z-\bar{w}-it)^2}\, dt.\end{equation}

The 1-point function $R$ appearing in Theorem \ref{mainth1} is then $R(z)=F(2\im z)=K(z,z),$
and $K$ is a correlation kernel for the limiting point field $\{z_j\}_1^\infty$.

If in \eqref{ulg2} we pass to the limit $c\to\infty$, we find that the kernel $K$ converges to the Ginibre kernel.
Similarly, the limit $c\to 0$ gives the well-known sine-kernel process on $\R$, but to see this we need a further rescaling.
Given the limiting point-field $\{z_j\}_1^\infty$ we form a new one $\{\tilde{z}_j\}_1^\infty$ by
$\tilde{z}_j=\alpha\cdot z_j$, where $\alpha=\tfrac 2 \pi a(p_*)
.$ The correlation kernel of
$\{\tilde{z}_j\}_1^\infty$ is
\begin{align*}\nonumber \tilde{K}(\tilde{z},\tilde{w})&=\tfrac 1 {\alpha^2}K(z,w),\qquad\qquad [\tilde{z}=\alpha\cdot z,\, \tilde{w}=\alpha\cdot w]\\
&=\tfrac 1 {\sqrt{2\pi}\alpha}e^{-\frac {(\im\tilde{z})^2+(\im\tilde{w})^2}{\alpha^2}}\int_{-\pi}^\pi
e^{-\frac {\alpha^2u^2}2+iu(\tilde{z}-\bar{\tilde{w}})}\, du.
\end{align*}
Sending $c\to 0$ (i.e.~$\alpha\to0$) we see readily that $\tilde{K}$ converges to $\pi$ times the sine-kernel
$$K^\mathrm{sin}(x,y)=\tfrac 1 \pi \tfrac {\sin(\pi x-\pi y)}{x-y},\qquad (x,y\in\R).$$
This well-known convergence follows from the Gaussian approximation of the Dirac delta:
$$
\frac{1}{|a|\sqrt{\pi}} e^{ -(x/a)^2 } \to \delta(x), \qquad \text{as } a \to 0,
$$
see e.g. \cite[Remark 4.(c)]{ACV} and \cite[Section 2.3]{AP2}.

 It could be said that the form of the kernel $\tilde{K}$ is more natural from a one-dimensional perspective,
whereas the form of $K$ in \eqref{ulg2} is more natural from a two-dimensional one. The kernel $\tilde{K}$ is of the form appearing in \cite{ACV}, whereas $K$ appears in the
context of planar ensembles in the papers \cite{AKM,AKMW}.

\smallskip

We now turn to edge scaling limits for the almost-Hermitian GUE. A limiting point field was found by Bender \cite{B} and was further investigated in the paper \cite{AB}.
For our purposes, it is advantageous to modify the ellipse potential $Q_N$ in \eqref{ellipse} so that the droplet becomes thicker, of height roughly $N^{-\frac 1 3}$ rather than $N^{-1}$. The advantage is that the form of rescaling in \eqref{blowup} remains correct, which facilitates when applying our method with Ward's equation.

We are thus led to introduce the \emph{modified Almost-Hermitian GUE} by redefining $Q_N$ as
\begin{equation}\label{flat}Q_N(\zeta)=
\tfrac 1 2\xi^{\,2}+\tfrac 1 {2c^2}N^{\frac 1 3}\eta^{\,2}.\end{equation}

\begin{figure}[h!]
	\begin{center}
		\includegraphics[width=4.5in,height=2.415in]{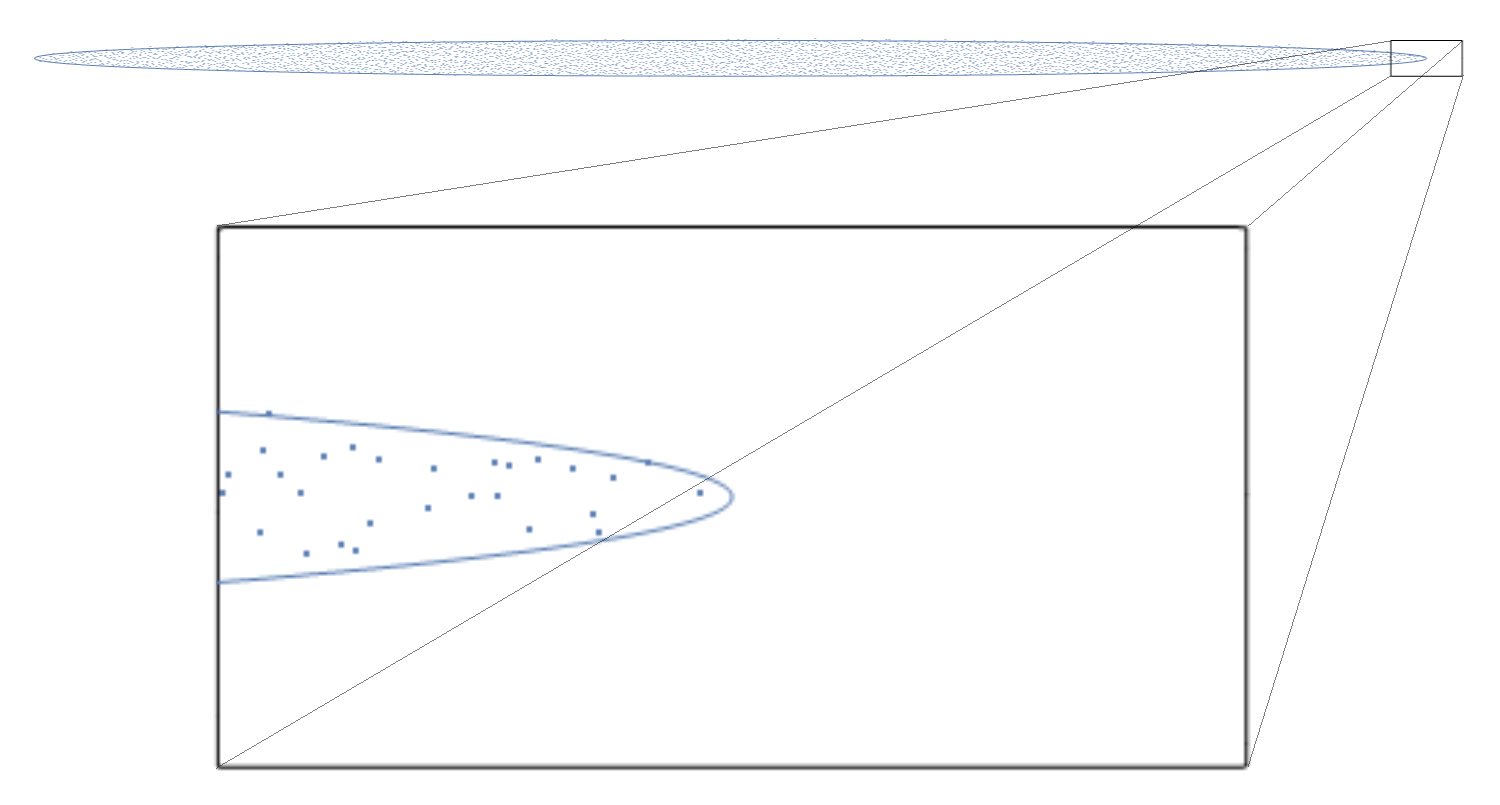}	
	\end{center}
	\caption{A sample from modified AGUE rescaled about the right endpoint $p_N$.} \label{Fig_AHEB}
\end{figure}

Using Lemma \ref{dropl}, we see that the right end-point of the droplet $S_{Q_N}$ is located
at
$p_N=2(1+c^2N^{-\frac 1 3})^{-\frac 1 2}.$ We rescale about this point using the magnification map in \eqref{blowup} with the modified potential \eqref{flat}, i.e., we put
$$\zeta\mapsto z,\qquad \Gamma_N(\zeta)={\scriptstyle \sqrt{N\Delta Q_N}}\cdot (\zeta-p_N)=\tfrac {N^{\frac 2 3}} {2c}(1+c^2N^{-\frac 1 3})^{\frac 1 2}\cdot (\zeta-p_N),$$
and we write $\{z_j\}_1^N$ for the rescaled ensemble, $z_j=\Gamma_N(\zeta_j)$, see Figure \ref{Fig_AHEB}.
As usual we denote by $R_N$ the 1-point function of $\{z_j\}_1^N$.

We can now restate the main result from \cite{B} in the following way.

\begin{thm} \label{Thm_Bender} (``\emph{Boundary scaling limit for modified AGUE}'') $R_N$ converges locally uniformly to the limit
\begin{align} \label{R bender}
\begin{split}
R(z)&=R_{(c)}(z)=\sqrt{2\pi} \, 4c^2 \,  e^{ \frac{4}{3}c^6-2(\im z)^2 }
\int_0^\infty e^{  4c^3(u+\re z) }  |\Ai( 2c(z+u)+c^4 )|^2   \, du.
\end{split}
\end{align}
\end{thm}

We now discuss how \eqref{R bender} interpolates between the linear Airy process and the planar $\erfc$-process.
We start with the planar case, which corresponds to letting $c\to\infty$ in \eqref{R bender}. For this, we fix a point $z$ in $\C\setminus (-\infty,0]$
and apply the well-known asymptotic formula
\begin{equation} \label{Airy asym}
\Ai z=\tfrac{1}{2\sqrt{\pi}z^{1/4}}e^{ -\frac{2}{3}z^{3/2} }\cdot (1+O(\tfrac 1z)),\qquad (z\to\infty).
\end{equation}

Inserting \eqref{Airy asym} in \eqref{R bender} and letting $c\to\infty$, one finds
\begin{align*}R_{(c)}(z)&=\sqrt{\tfrac{2}{\pi}} c^2e^{ \frac{4}{3}c^6-2(\im z)^2 }\\
&\times\int_0^\infty |2c(z+u)+c^4|^{-\frac 12} e^{4c^3(u+\re z)}e^{-\frac 4 3 \re((2c(z+u)+c^4)^{\frac 3 2})}\, du\cdot (1+O(c^{-1}))\\
&\to\sqrt{\tfrac{2}{\pi}}\int_0^\infty e^{-2(u+\re z)^2}\, du\cdot (1+O(c^{-1}))=\int_{-\infty}^0 \gamma(z+\bar{z}-t)\, dt.
\end{align*}
Here, we have used
$$
\tfrac{4}{3}c^6-2(\im z)^2+ 4c^3(u+\re z) -\tfrac 4 3 \re((2c(z+u)+c^4)^{\frac 3 2}) = -2(u+\re z)^2 + O(c^{-1}).
$$
Hence as $c\to\infty$, $R_{(c)}$ converges to the limit
$$R_{(c=\infty)}(z)=\gamma*\1_{(-\infty,0)}(z+\bar{z})=\tfrac 1 2 \erfc(\tfrac {z+\bar{z}} {\sqrt{2}}).$$ This is the 1-point function which appears in random normal matrix theory after the process of rescaling about a regular boundary point of the droplet (see \cite{AKM} and references).

We next consider the limit of \eqref{R bender}  as $c\to 0$.
A standard polarization argument shows that the 1-point function $R_{(c)}$ in \eqref{R bender} corresponds to the correlation kernel
\begin{align*}
\begin{split}
K(z,w)&=\sqrt{2\pi} \, 4c^2 \, e^{ \frac{4}{3}c^6-(\im z)^2-(\im w)^2 }
\\
&\times \int_0^\infty e^{ 2c^3(2u+z+\bar{w}) }   \Ai( 2c\,(z+u)+c^4 ) \text{Ai}( 2c\,(\bar{w}+u)+c^4 )  \, du.
\end{split}
\end{align*}

Now put
$\tilde{z}=\alpha\cdot z$, $\tilde{w}=\alpha\cdot w$, and $\tilde{z}_j=\alpha\cdot z_j$, where  $\alpha=c\sqrt{2}.$
The rescaled point-field $\{\tilde{z}_j\}_1^\infty$, then has correlation kernel
\begin{align*}
\begin{split}
\tilde{K}(\tilde{z},\tilde{w})&= \alpha^2 K( \alpha z ,\alpha w )
\\
&= \tfrac{\sqrt{\pi}}{\alpha} e^{ \frac{\alpha^6}{6} -\frac{  (\Im\,\tilde{z})^2+(\Im\,\tilde{w})^2 }{2\alpha^2} }
\int_0^\infty e^{  \frac{\alpha^2}{2} (2u+\tilde{z}+\bar{\tilde{w}}) } \Ai ( \tilde{z}+u+\tfrac{\alpha^4}{4} )  \Ai( \bar{\tilde{w}}+u+\tfrac{\alpha^4}{4} ) \, du.
\end{split}
\end{align*}
This kernel is found in the paper \cite{AB}, and is in turn equivalent with a double integral formula found in \cite{B}. We shall
write $\tilde{R}_{(c)}(z)=\tilde{K}(z,z)$.

Letting $c \to 0$ (i.e. $\alpha\to 0$), it is now straightforward to check (cf.~\cite{AB}) that $\tilde{K}$ converges to $\pi K^{\text{Ai}}$ where $K^{\text{Ai}}$ is the Airy kernel
$$
K^{\Ai}(x,y)=\frac{ \Ai x\, \Ai' y-\Ai' x\, \Ai y  }{x-y}=\int_{0}^{\infty} \Ai(x+u)\Ai(y+u) \, du,\qquad (x,y\in\R).
$$

\begin{figure}[h!]
    	\begin{subfigure}[h]{0.32\textwidth}
    	\begin{center}
    		\includegraphics[width=1.69in,height=1.36in]{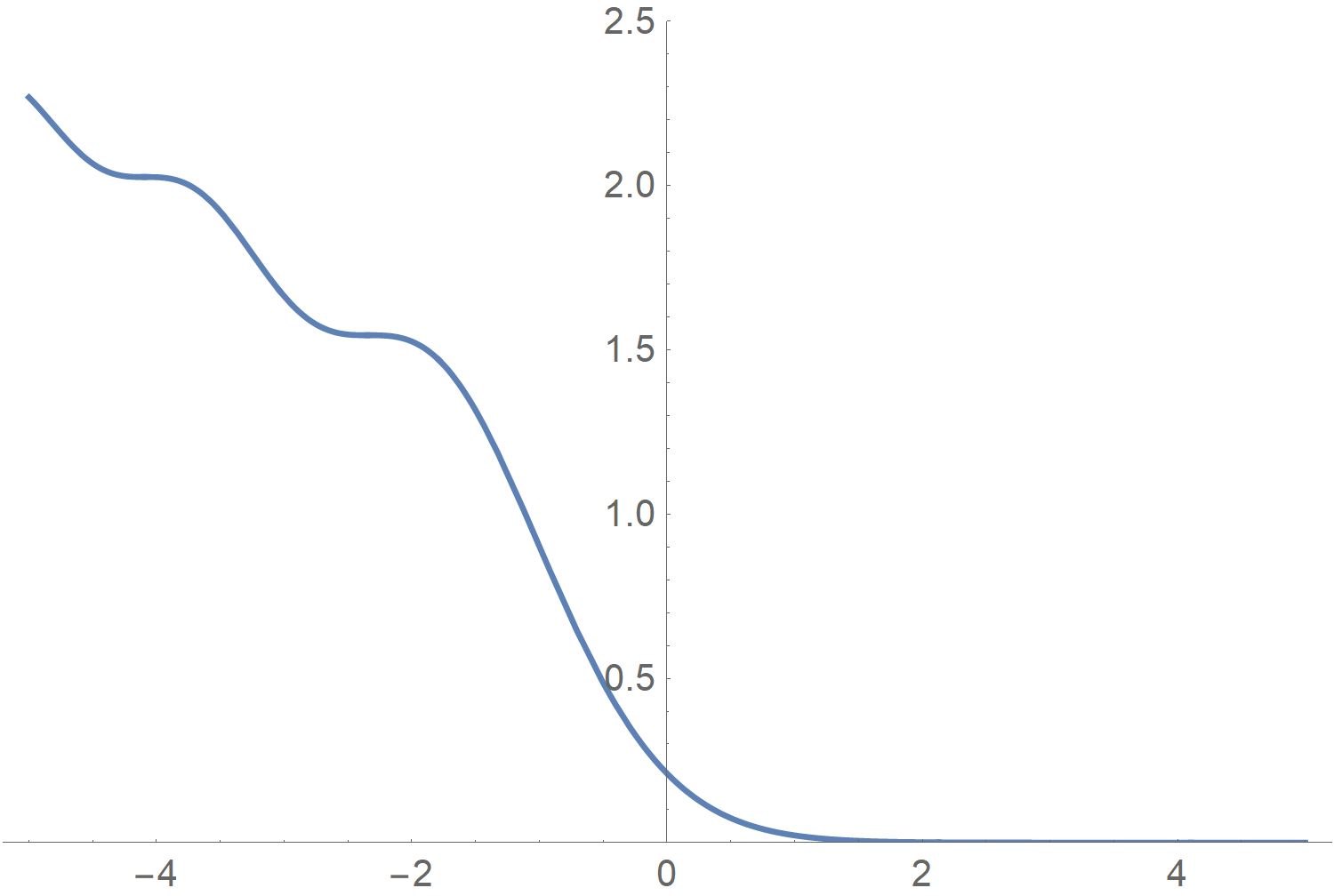}	
    	\end{center}
    	\caption{$c=0$}
    \end{subfigure}
    \begin{subfigure}[h]{0.32\textwidth}
    	\begin{center}
    		\includegraphics[width=1.69in,height=1.36in]{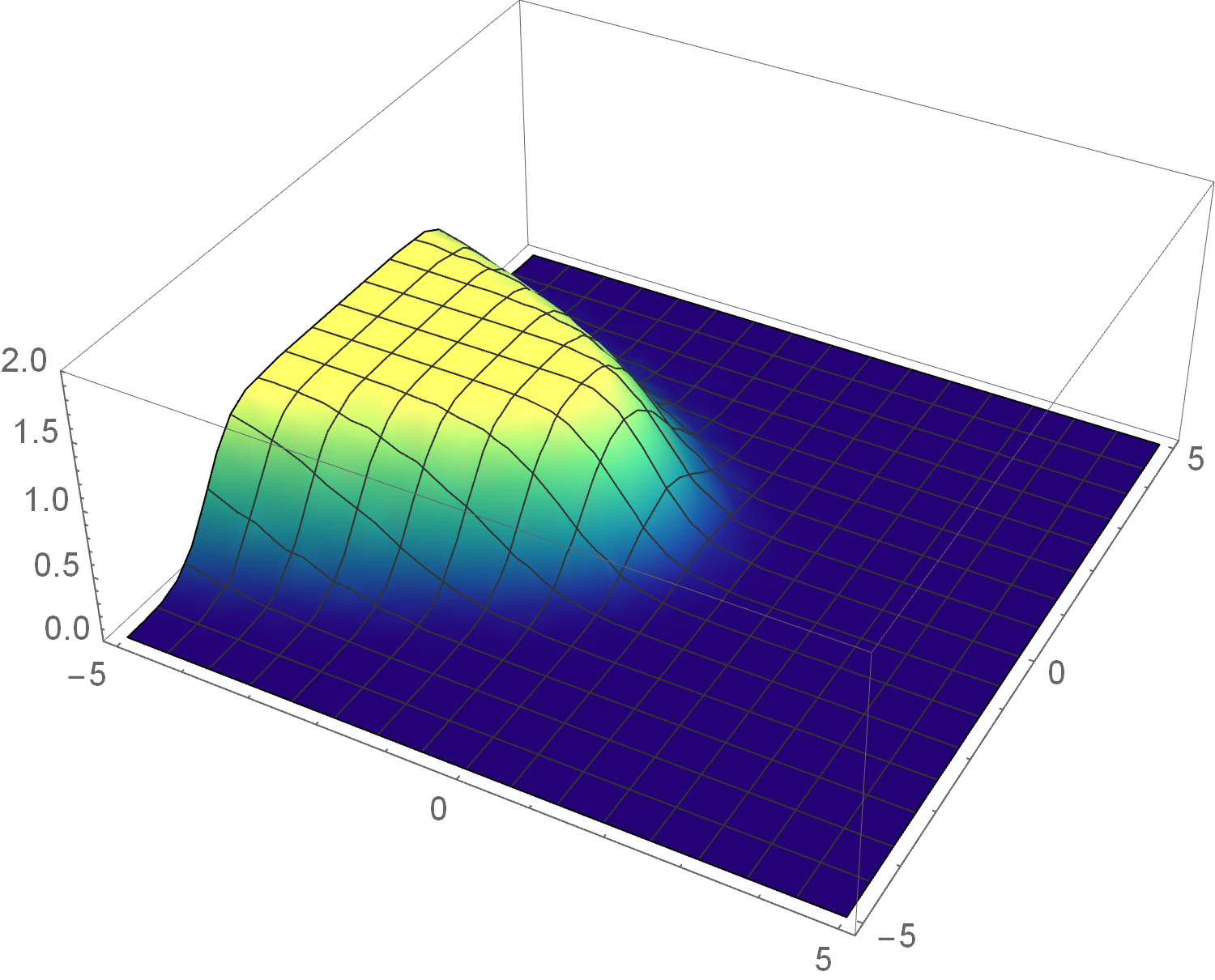}
    	\end{center}
    	\caption{$c=1$}
    \end{subfigure}	
    \begin{subfigure}[h]{0.32\textwidth}
    	\begin{center}
    		\includegraphics[width=1.69in,height=1.36in]{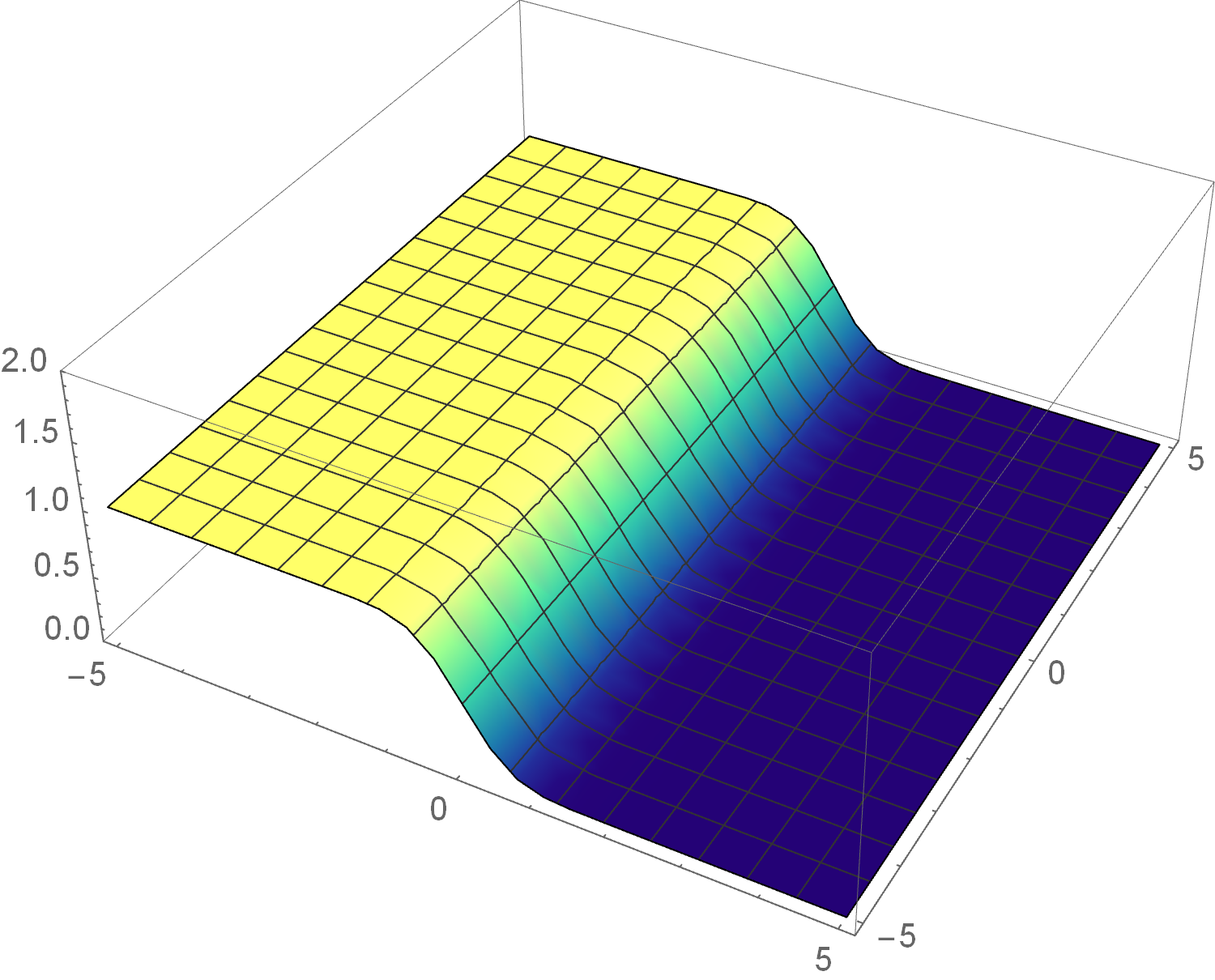}
    	\end{center}
    	\caption{$c=\infty$}
    \end{subfigure}	
	\caption{Graphs of the limiting one-point functions $\tilde{R}_{(0)}(x)$ for $x\in\R$ and of $R_{(1)}(z)$ and $R_{(\infty)}(z)$ for $z\in\C$.} \label{Fig_Bender R}
\end{figure}

Non-translation invariant solutions to Ward's equation are discussed in \cite[Section 8.3]{AKM}, but no example is given there. It is interesting to note that the
1-point functions in \eqref{R bender} give rise to such solutions.
See the case $c=1$ of Figure \ref{Fig_Bender R}.

\subsection{Almost-Hermitian LUE} \label{ssal} We now fix two parameters $c>0$ and $\nu>-1$. We shall denote by $K_\nu(z)$ the modified
Bessel function of the second kind (see \cite[p. 78]{W}).

By definition, the ALUE with parameters $c$ and $\nu$ (and $\beta=1$) is a random sample $\{\zeta_j\}_1^N$ picked from the Gibbs measure \eqref{prob}
associated with the potential
\begin{equation}\label{wish}
Q_N(\zeta)=-\tfrac 1 N\log \left[K_{\nu}(\tfrac {N^2|\zeta|}{c^2})\cdot |\zeta|^{\nu}\right]-\left(\tfrac N{c^{2}}-1\right)\cdot \re\zeta.
\end{equation}

We remark that $Q_N$ is continuous, but
$\Delta Q_N$ has a singularity at the origin.

\smallskip

When $\nu$ is an integer, a random sample $\{\zeta_j\}_1^N$ with respect to \eqref{wish} can be identified with eigenvalues of products of two rectangular random matrices with suitable Gaussian entries, see Subsection \ref{agen}.
Moreover (for general $\nu>-1$) we obtain a LUE-type analogue of a chiral ensemble studied by Akemann \cite{Ake} and Osborn \cite{O} in the context of quantum chromodynamics.
More about this will be said in Section \ref{QCD}.

It follows from computations in \cite[Theorem 1]{ABK} that the droplet in potential $Q_N$ is the elliptic disc (containing the origin)
\begin{align}\label{wishdrop}S_{Q_N}&=\{\xi+i\eta\, ;\,\tfrac {(\xi-t_N)^2}{a_N^2}+\tfrac {\eta^2}{b_N^2}\le 1\,\},\\
\nonumber \scriptstyle (a_N&\scriptstyle=2-c^2N^{-1}+\cdots,\, b_N=2c^2N^{-1}+\cdots, \, t_N=2-c^2N^{-1}+\cdots).\end{align}

We now recall the following asymptotic formula for the Bessel function $K_\nu$ (found in \cite[Section 7$\cdot$3]{W})
\begin{equation}\label{Bessel_asym}K_\nu(x)=\sqrt{\tfrac \pi {2x}}e^{-x}(1+O(x^{-1})),\qquad  (x\to+\infty),\end{equation}
and deduce that for each fixed $\zeta\ne 0$, as $N\to\infty$,
\begin{equation}\label{besselA1}
Q_N(\zeta) = \tfrac N{c^{2}}\cdot |\zeta|-(\tfrac N{c^{2}}-1)\cdot \re\zeta+\tfrac{1}{N}\log N+O(\tfrac 1N),\qquad (N\to\infty).
\end{equation}

Thus $Q_N$ converges pointwise to the limit
\begin{equation}\label{mppot}V(\zeta)={\scriptstyle{\begin{cases}\,\zeta & \text{if}\quad \zeta\in \R\quad \text{and}\quad \zeta\ge 0,\cr
+\infty & \text{otherwise}.\cr \end{cases}}}
\end{equation}

It is well-known that the equilibrium density in potential \eqref{mppot} is the Marchenko-Pastur law (\cite[Section 7.2]{F}),
\begin{equation}\label{MPlaw}\MP(\xi)=\tfrac 1 {2\pi\xi}\sqrt{(4-\xi)\xi}\cdot\1_{[0,4]}(\xi).\end{equation}

As usual, we write $\bfR_N$ for the 1-point function with respect to \eqref{wish} and $c_N$ for the corresponding cross-sections in \eqref{decs}.

\begin{thm}\label{mthMP} (``\emph{Pointwise cross-section convergence for ALUE}'') $c_N(\xi)\to\pi\cdot \MP(\xi)$ as $N\to\infty$ for each $\xi$
with $0<\xi<4$.
\end{thm}

We shall now consider scaling limits; the following remark will come in handy.

\begin{rmk*}
For fixed $\zeta\ne 0$ we have the following approximation for the Laplacian $\Delta Q_N(\zeta)$,
\begin{equation}\label{BA2}\Delta Q_N(\zeta)=\tfrac N {4c^2|\zeta|}+o(\tfrac 1N).\end{equation}
The formula \eqref{BA2} is easily deduced
using well-known asymptotics for the Bessel function $K_\nu$, see e.g.~\cite{ABK} (or \cite{W}) for details.
\end{rmk*}

Now fix a zooming-point $p_*\in S_V$ in the bulk, i.e.,  $0< p_*< 4$.
Using \eqref{BA2}, the quantities $\rho(p_*)$ and $a(p_*)$ in \eqref{rhop} and \eqref{ap} are readily computed as
\begin{equation}\label{apMP}\rho(p_*)=\lim_{N\to\infty}\sqrt{\tfrac N {\Delta Q_N(p_*)}}=2c\sqrt{p_*},\qquad a(p_*)=\pi c\sqrt{p_*}\cdot \MP(p_*).\end{equation}
We rescale a random sample $\{\zeta_j\}_1^N$ about $p_*$ using the map \eqref{blowup} and let $R_N$ denote the 1-point function of the rescaled system $\{z_j\}_1^N$.

\begin{thm}\label{mthWish} (``\emph{Bulk scaling limit for ALUE}'') If $\nu>-1$ is an integer then $R_N\to R$ locally uniformly where $R(z)=F(2\im z)$ and $F(z)=\gamma*\1_{(-2a,2a)}(z)$, where $a=a(p_*)$ is given by \eqref{apMP}.
\end{thm}

(The assumption that $\nu$ be an integer is made mainly for technical convenience; we do not think that it should be necessary.)

While Theorem \ref{mthWish} looks similar to Theorem \ref{mainth1}, the result is new; the shape of the droplet was calculated only recently in \cite{ABK}, and the singular behaviour
of the equilibrium measure renders our discussion more delicate in the present case.

The similarity with AGUE breaks down when we consider edge scaling limits,
and more precisely when we zoom appropriately about the singularity at the origin. We now turn to this interesting scaling limit.

In keeping with a two-dimensional perspective, we rescale $\{\zeta_j\}_1^N$ about the origin by blowing up with the factor
$(\tfrac N c)^2$, i.e., we put
	$$z_j=(\tfrac N c)^2\zeta_j,\qquad z=(\tfrac N c)^2\zeta, \qquad R_N(z)=(\tfrac c N)^4\, \bfR_N(\zeta).$$

\begin{figure}[h!]
	\begin{center}
		\includegraphics[width=4.5in,height=2.415in]{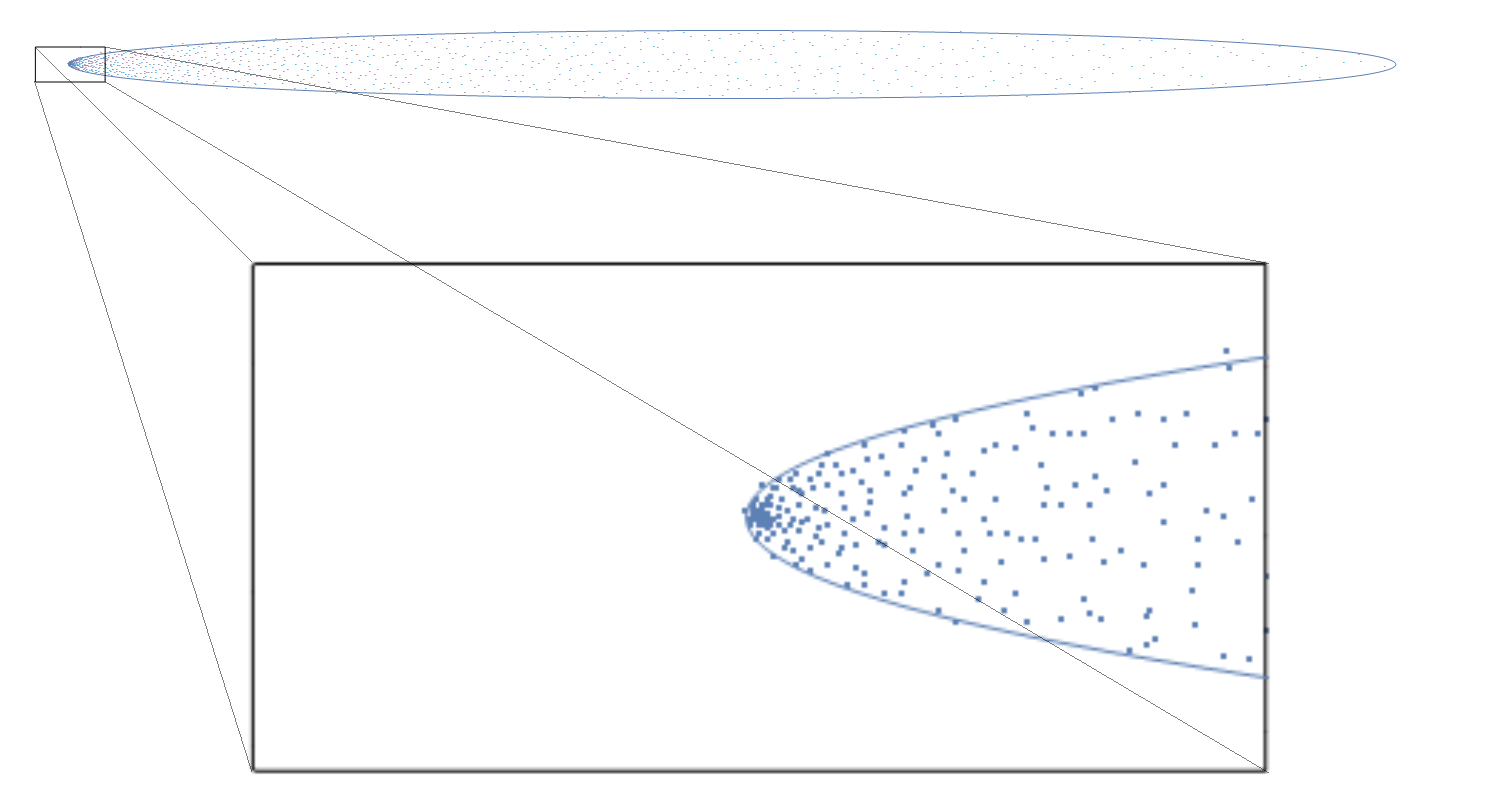}	
	\end{center}
	\caption{A sample from ALUE rescaled about the singular point $0$. Here $N=1000$ and $c$ is approximately equal to $7$.} \label{Fig_ANWB}
\end{figure}

Using a slight modification of computations due to Osborn in \cite{O}, we will prove the following result.

\begin{thm}\label{Thm_Osborn} (``\emph{Edge scaling limit for ALUE}'')
We have $R_N \to R$ uniformly on compact subsets of $\C\setminus (-\infty,0]$, where
\begin{equation}\label{RO}
R(z)=R_{(c)}(z)=\tfrac{1}{2} \, K_\nu(|z|)\, e^{\,\re z} \int_{0}^{2c} s\, e^{-\frac 12 s^2}  |J_\nu(s \, z^{\frac 12})|^2 \, ds.
\end{equation}
\end{thm}

(The square-root $z^{\frac 12}$ is defined as the principal branch, with a branch-cut along the negative real axis.)

Again the density \eqref{RO} has an interesting interpolating property, this time between a linear Bessel process and a relatively lesser known  planar Bessel-process (see e.g., \cite{APS1}).

We start by passing to the planar limit as $c\to\infty$ in \eqref{RO}, which is
\begin{equation}R_{(c=\infty)}(z)=\tfrac{1}{2} \, K_\nu(|z|)\,  e^{\,\re z} \int_{0}^{\infty} s\, e^{-\frac 1 2s^2}  |J_\nu(s \, z^{\frac 1 2})|^2 \, ds.
\end{equation}
The integral here can be evaluated by means of a formula in \cite[(10.22.67)]{OLBC}, giving
 \begin{equation}\label{volm}R_{(c=\infty)}(z)=\tfrac 1 2K_\nu(|z|)\,I_\nu(|z|),\qquad z\in\C,\end{equation}
where $I_\nu$ is the modified Bessel function of the first kind. There is a unique determinantal point-field $\{z_j\}_1^\infty$ having the one-point function \eqref{volm},
which we call a ``planar Bessel-process''.

We rescale once more in order to obtain a 1-dimensional perspective: given a random sample $\{z_j\}_1^\infty$ associated with \eqref{RO}, we form
 $\{\tilde{z}_j\}_1^\infty$ where
 $\tilde{z}_j=c^2z_j.$ The point field $\{\tilde{z}_j\}_1^\infty$ has $1$-point function $\tilde{R}_{(c)}(z)=c^{-4}R_{(c)}(c^{-2}z).$
Writing $\tilde{R}_{(0)}$ for the limit of $\tilde{R}_{(c)}$ as $c\to 0$, we readily see that
 $\tilde{R}_{(0)}=0$ outside of the ray $[0,\infty)$ while for $x>0$,
\begin{align*}
	\tilde{R}_{(0)}(x)&=\tfrac{\pi}{2} \int_{0}^{1} t  |J_\nu(t  \sqrt{x})|^2 \, dt=\tfrac{\pi}{4} (J_\nu(\sqrt{x})^2 -J_{\nu+1}(\sqrt{x})J_{\nu-1}(\sqrt{x}) ).
\end{align*}

We recognize the last expression as $\pi$ times the restriction to the diagonal of the well-known Bessel kernel, given for $x,y>0$ by
\begin{align}
\begin{split}
\label{besker}K^{\mathrm{Bes},\nu}(x,y)&=\tfrac12 \int_{0}^{1} t  J_\nu(t \sqrt{x}) J_\nu(t \sqrt{y})  \, dt
=\tfrac {J_\nu(\sqrt{x})\sqrt{y}J_\nu'(\sqrt{y})-\sqrt{x}J_\nu'(\sqrt{x})J_\nu(\sqrt{y})}{2(x-y)}.
\end{split}
\end{align}

\begin{figure}[h]
	\begin{subfigure}[h]{0.32\textwidth}
		\begin{center}
			\includegraphics[width=1.69in,height=1.36in]{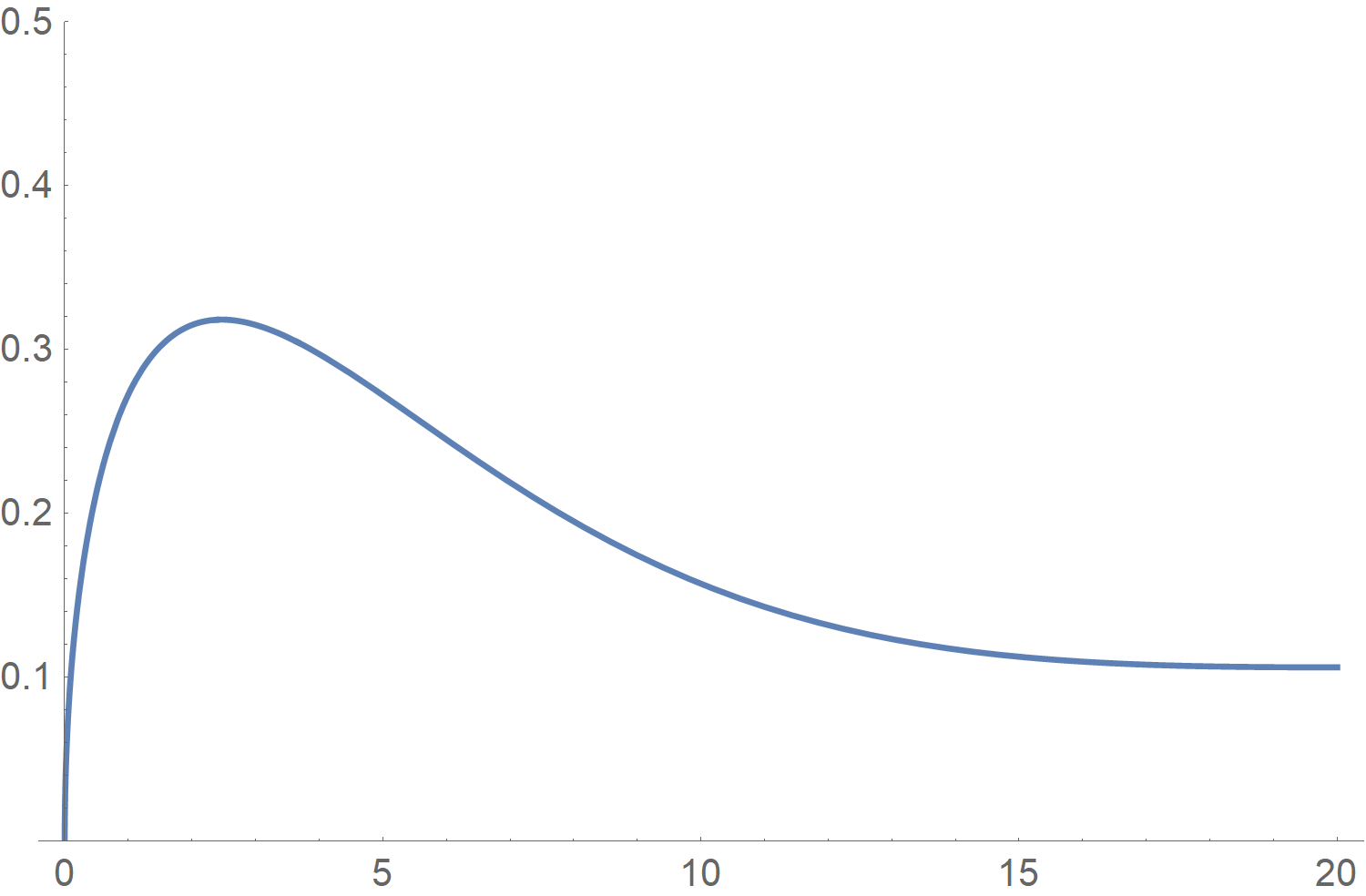}	
		\end{center}
		\caption{$c=0$, $\nu=\frac 1 2$}
	\end{subfigure}
	\begin{subfigure}[h]{0.32\textwidth}
		\begin{center}
			\includegraphics[width=1.69in,height=1.36in]{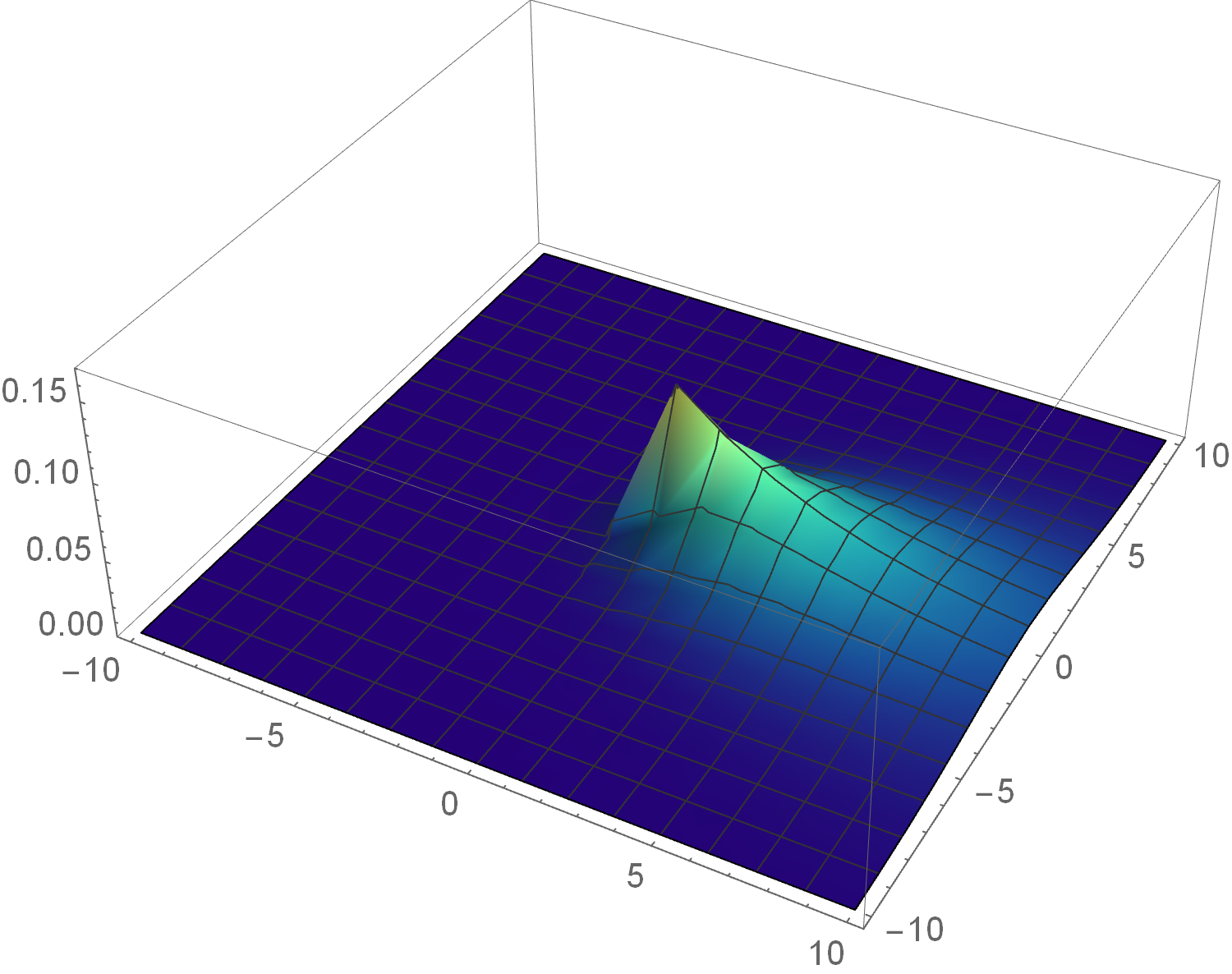}
		\end{center}
		\caption{$c=1$, $\nu=\frac 1 2$}
	\end{subfigure}	
	\begin{subfigure}[h]{0.32\textwidth}
		\begin{center}
			\includegraphics[width=1.69in,height=1.36in]{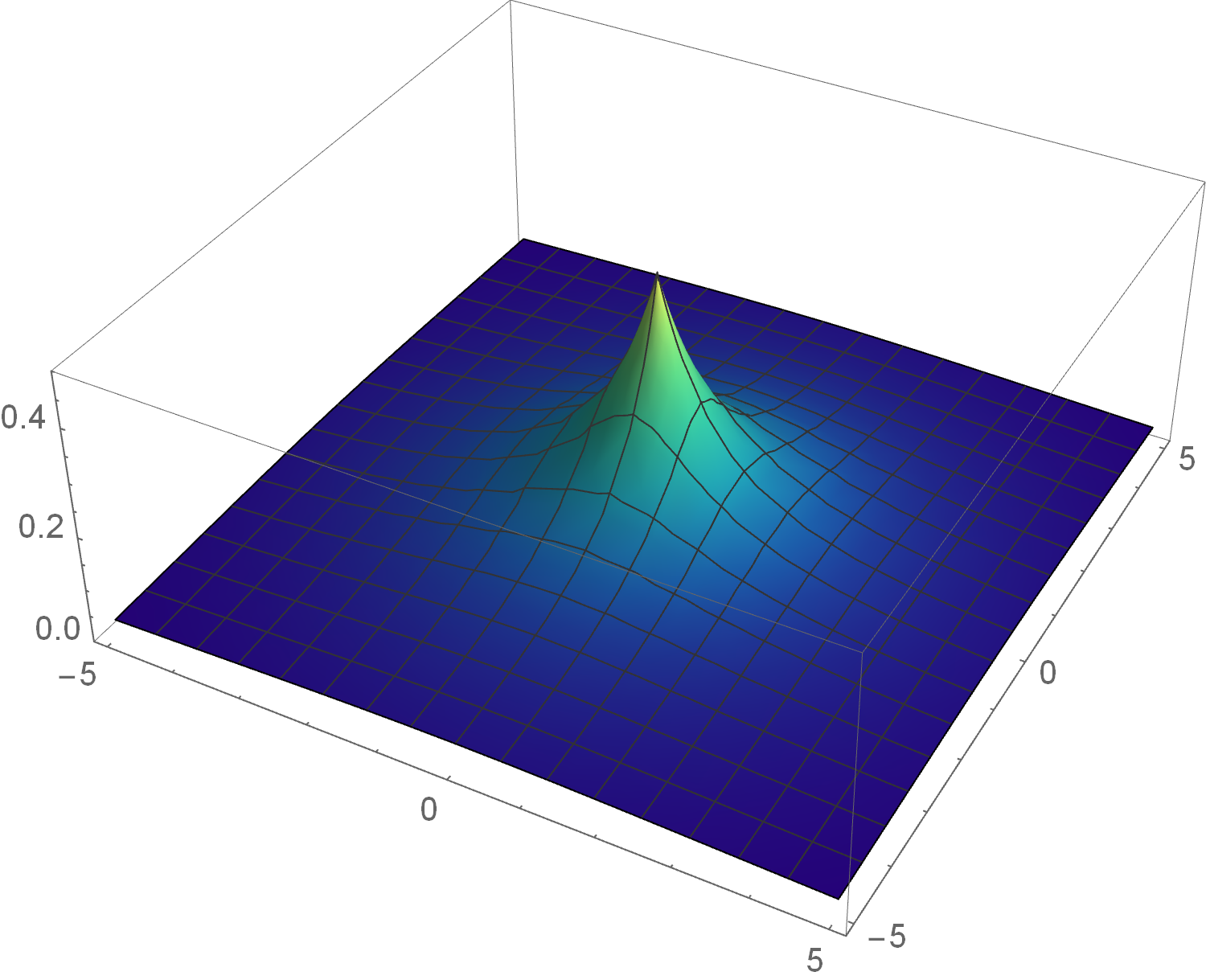}
		\end{center}
		\caption{$c=\infty$, $\nu=\frac 1 2$}
	\end{subfigure}	
	\caption{Graphs of $\tilde{R}_{(0)}(x)$ and $R_{(1)}(z)$, $R_{(\infty)}(z)$ for $\nu=\tfrac 1 2$. } \label{Fig_Osborn R}
\end{figure}

We conclude that the rescaled ensemble interpolates between the
linear Bessel-process with kernel \eqref{besker} (when $c=0$) and the planar Bessel process with 1-point function \eqref{RO} (when $c=\infty$).
See Figure \ref{Fig_Osborn R}.

\subsection{Almost-Circular ensembles}

We now discuss the almost-circular ensembles, a family of random normal matrices whose eigenvalues tend to be distributed within a narrow band around the unit circle of width proportional to $\frac1N$.
The terminology ``almost-circular'' comes from the fact that in the one-dimensional limit, this ensemble tends to the circular unitary ensemble distributed on the unit circle, cf. \cite[Chapter 2]{F}. (See also \cite{FW} for a recent work on circular $\beta$-ensemble.)

The almost-circular ensemble we consider is associated with a radially symmetric potential
\begin{equation}
	Q_N(\zeta)=g_N(r), \qquad r=|\zeta|, \qquad g_N: \R_+ \to \R,
\end{equation}
where $g_N$ is differentiable on $\R_+$ with absolutely continuous derivative.
Without loss of generality, we shall assume that
$g_N(1)=0$ and $g_N'(1)=2.$
We also assume that $Q_N$ is subharmonic in $\C$ and strictly subharmonic in a neighborhood of the unit circle.
Then the associated droplet $S_{Q_N}$ is given by an annulus
\begin{equation}
	S_{Q_N}:=\{ \zeta \in \C \, |  \, r_N \le |\zeta| \le 1 \},
\end{equation}
where $r_N$ is a unique constant satisfying $r_N g_N'(r_N)=0,$ see e.g. \cite{ST}.
A typical example of such a model is the \emph{induced Ginibre ensemble} \cite{FBKS}, an extension of the Ginibre ensemble to include zero eigenvalues

As in previous subsections, we shall assume that the limit
\begin{equation} \label{rho}
	\rho:=\lim_{N \to \infty} \sqrt{\tfrac{N}{\Lap Q_N(1)}} > 0.
\end{equation}
exists. This condition plays an important role in defining the almost-circular ensemble since
\begin{equation} \label{rN ACE}
	r_N=1-\tfrac{\rho^2}{2N}+o(\tfrac{1}{N}),\quad N\to \infty.
\end{equation}

Let $p_N:=(r_N+1)/2$. For an arbitrary zooming-point $ p_{N,\theta}:=p_N \cdot e^{i \theta} $ ($\theta \in [0,2\pi]$), we define the rescaled sample
\begin{equation}
	z_j:=-i e^{i\theta}  \sqrt{N\Delta Q_N(p_{N,\theta})} \cdot (\zeta-p_{N,\theta}),
\end{equation}
see Figure~\ref{fig_inducedG}.

\begin{figure}[h]
	\begin{center}
		\includegraphics[width=0.8\textwidth]{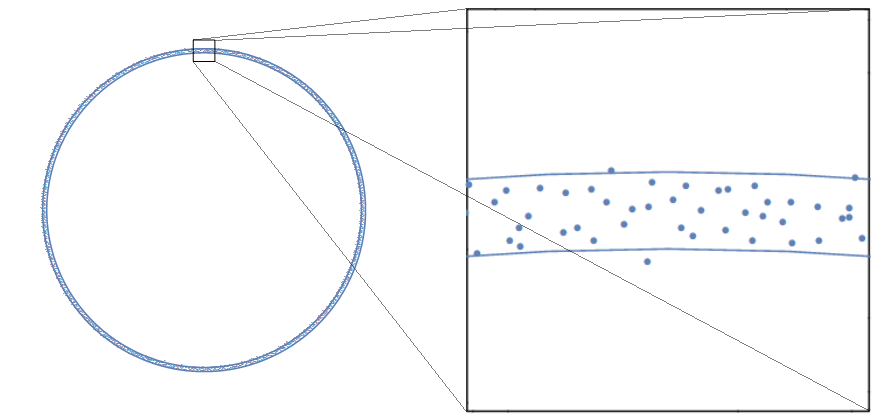}	
	\end{center}
	\caption{A sample from the almost-circular induced Ginibre ensemble.}
	\label{fig_inducedG}
\end{figure}	
	
As an immediate application of Theorem~\ref{tithm} we obtain the following result, which establishes the bulk universality for the almost-circular ensembles.

\begin{thm} \label{Thm_AUE}
	(``\emph{Bulk scaling limit for Almost-Circular ensembles}'') Under the above assumptions, $R_N$ converges locally uniformly to the limit $R(z)=F(2\im z)$ where
	$F(z)=\gamma*\1_{(-2a,2a)}(z)$ and $a=\rho/4$.
\end{thm}

We remark that the value $a$ again is of the form \eqref{appl} and \eqref{apMP} as it can reads as $a=\tfrac{\pi}{2}\cdot \rho \cdot \tfrac{1}{2\pi}=\tfrac{1}{4} \cdot \rho,$ where $\tfrac{1}{2\pi}$ corresponds to the uniform density on the unit circle.
We also refer to \cite{BS} for extensions of Theorem~\ref{Thm_AUE} under several boundary conditions.

\subsection{Plan of this paper and further results.} In Section \ref{BCG} we study a fairly general class of bandlimited Coulomb gas ensembles in the interface between dimensions 1 and 2.

In Section \ref{Sec3} we specialize to the determinantal case $\beta=1$. We adapt to the almost-Hermitian setting the basic results from \cite{AKM} concerning structure of limiting kernels, zero-one law, Ward's equation, and so on. As an application, we prove Theorem~\ref{Thm_AUE}.

In Section \ref{Sec4} we prove Theorem \ref{mth0} and Theorem \ref{mthMP} about cross-section convergence for almost-Hermitian GUE and LUE, respectively.

In Section \ref{BSL} we prove Theorem \ref{mainth1} and Theorem \ref{mthWish} about bulk scaling limits for AGUE and ALUE, respectively.

In Section \ref{ESL} we consider edge scaling limits for AGUE and ALUE. We include a new proof for
Theorem \ref{Thm_Bender}, as well as a proof of Theorem \ref{Thm_Osborn}.

In Section \ref{QCD} we compare our results for ALUE with earlier results about a chiral ensemble in \cite{Ake,O}. We also relate our results with some other point fields appearing in \cite{AKS,F,K,M}, for example.

In Section \ref{crem}, we provide concluding remarks and relate to other works. In particular, we state and prove generalizations of the cross-section convergence
and bulk scaling limits for the general ALUE with a rectangular parameter $\alpha\ge 0$. We also discuss results about a generalization of AGUE, which is obtained by inserting a point charge at the origin, and we discuss almost-Hermitian ensembles with a hard edge.

In the appendices, we collect some asymptotic estimates for Hermite and Laguerre polynomials which were adapted from the papers \cite{LR,Wa} and \cite{V}, respectively.

\section{Generalities about the bandlimited Coulomb gas} \label{BCG}

In this section, we introduce and study a fairly general setup for dealing with almost-Hermitian structures. We begin by studying limiting droplets in the bandlimited regime.
After that we study statistical cross-sections by using some basic large deviation estimates.
This prompts us to impose yet more restrictive conditions on the potentials $Q_N$.

\subsection{The limiting droplet} \label{tld} Let $Q_N$ and $V$ be potentials satisfying the conditions in Subsection \ref{Setup}. We now impose some standing assumptions.

First of all, we will assume that each $Q_N$ is continuous in some fixed neighbourhood of $S_V$. We assume also that $Q_N$ is real-analytic in this neighbourhood
with the possible exception of one (or several) singular points
at which $\Delta Q_N=+\infty$ for each $N$ (so that we include the ALUE in our setting, for example). With another very minor restriction, we assume that each droplet $S_{Q_N}$ equals to the corresponding coincidence set $S_{Q_N}^*$ for the obstacle problem (cf.~\cite{ST} or Subsection \ref{BCS} below).

In addition, we will assume  that
$Q_N(\zeta)=Q_N(\bar{\zeta})$ and $Q_N(\zeta)\ge Q_N(\re\zeta)$ for all $\zeta\in\C$,
and that $\Delta Q_N\ge mN$ on $S_{Q_N}$ for some constant $m>0$.

Under these assumptions, Sakai's regularity theorem (see \cite{AKMW} and references) implies that the boundaries $\d S_{Q_N}$ consist of finitely many Jordan curves, which are real analytic with the possible
exception of finitely many singular points. We shall assume that there are no singular points and that there are continuous functions $h_N:\R\to[0,\infty)$ such that
$$S_{Q_N}=\{\xi+i\eta\,;\, -h_N(\xi)\le\eta\le h_N(\xi)\}.$$

Since the $Q_N$ increase monotonically to $V$, we can apply a general convergence result in \cite[Theorem I.6.2]{ST} to conclude the weak convergence in the sense of measures
$\sigma_{Q_N}\to\sigma_V$ as $N\to\infty$.
Moreover, since $\sigma_{Q_N}=\Delta Q_N\cdot \1_{S_{Q_N}}\, dA$ we have,
for each bounded continuous function $f$, that
$$\tfrac 1 \pi \lim_{N\to\infty}\int_\R\, d\xi\int_{-h_N(\xi)}^{h_N(\xi)}f(\xi+i\eta)\Delta Q_N(\xi+i\eta)\, d\eta=\int_\R f(\xi)\sigma_V(\xi)\, d\xi.$$

We next impose the condition that the limit
$$\rho(\xi)=\lim_{N\to\infty}\sqrt{\tfrac N {\Delta Q_N(\xi)}}$$ exists as a finite
and strictly positive number whenever $\sigma_V(\xi)>0$.
Using this assumption and the mean-value theorem, we obtain for all $\xi$ with $\sigma_V(\xi)>0$ that
\begin{align}\label{star}\pi\cdot \sigma_V(\xi)&=\lim_{N\to\infty}\int_{-h_N(\xi)}^{h_N(\xi)}\Delta Q_N(\xi+i\eta)\, d\eta=\tfrac 2 {\rho(\xi)}
\lim_{N\to\infty}h_N(\xi){\scriptstyle \sqrt{N\Delta Q_N(\xi)}}.
\end{align}
Let us particularly mention that $h_N(\xi) \to 0$ as $N \to \infty$.

Let $\Gamma_{N,p}(\zeta)={\scriptstyle \sqrt{N\Delta Q_N(p)}}\cdot (\zeta-p)$ be the magnification map about a point $p\in\R$, and consider the vertical segment $\gamma_N(p)=\Gamma_{N,p}(S_{Q_N})\cap (i\R)$. Writing
$\gamma_N(p)=[-a_N(p)i,a_N(p)i]$ we then have
$a_N(p)={\scriptstyle \sqrt{N\Delta Q_N(p)}}\cdot h_N(p).$

In view of \eqref{star}, we have shown the following theorem.

\begin{thm}\label{aplem} (``\emph{Asymptotic shape of the droplet}'') If $Q_N$ satisfies the above conditions, and if $\gamma_N(p)=[-a_N(p)i,a_N(p)i]$ then
$a_N(p)\to a(p)$ as $N\to\infty$ where
$a(p)=\tfrac \pi 2\cdot\rho(p)\cdot \sigma_V(p).$
\end{thm}

In particular, the droplets $S_{Q_N}$ are \textit{bandlimited} in the sense that there is a constant $A$ such that $|\im\zeta|\le \tfrac A N$ for each $\zeta\in S_{Q_N}$.

\subsection{Statistical cross-sections} \label{stand} Recall from potential theory that the \emph{Robin's constant} $\gamma(V)$ in external potential $V$ is defined by
$\gamma(V)=I_V[\sigma_V]$ where $I_V$ is the energy-functional \eqref{Ilog}.
We say that a sequence $Q_N$ of potentials obeying the conditions in Subsection \ref{tld} is
an \textit{admissible sequence} if the following lower bound for the partition function holds:

\begin{equation}
\label{e1}\liminf_{N\to\infty}\tfrac 1 {N^2}\log Z_N^\beta\ge-\beta\cdot \gamma(V).\end{equation}

\begin{thm} \label{joh1} (``\emph{Weak cross-section convergence}'') Suppose that $(Q_N)$ is an admissible sequence and that the restriction $V|_\R$ is continuous in a
$\R$-neighbourhood of the $V$-droplet $S_V$.
Then $\tfrac 1 \pi c_{N}^\beta\to  \sigma_V$ in the weak sense of measures on $\R$.
\end{thm}

Note that the continuity assumption on $V|_\R$ is satisfied for AGUE but not for ALUE.

Our proof of Theorem \ref{joh1} in the following section uses a modification
of arguments from \cite{J} (as well as the planar version in \cite[Appendix A]{A}).

\subsection{Proof of Theorem \ref{joh1}}  Given an external potential $W$ we write $L_W$ for the kernel
$$L_W(\zeta,\eta)=\log\tfrac 1 {|\zeta-\eta|}+\tfrac 1 2(W(\zeta)+W(\eta)).$$

To a random configuration $\{\zeta_j\}_1^N$ we associate  the corresponding empirical measure
$$\mu_N=\tfrac 1 N\sum_1^N\delta_{\zeta_j}.$$ The \textit{discrete $W$-energy}
of this measure is defined to be
\begin{equation}I_W^\sharp[\mu_N]=\tfrac 1 {N(N-1)}\sum_{j\ne k}L_W(\zeta_j,\zeta_k)=\tfrac 1 {N(N-1)}\sum_{j\ne k}\log\tfrac 1 {|\zeta_j-\zeta_k|}+\mu_N(W).\end{equation}
This energy is closely related to the Hamiltonian $H_N$ in external potential $W$, since
$$H_N=\sum_{j\ne k}\log\tfrac 1 {|\zeta_j-\zeta_k|}+N\sum_{j=1}^N W(\zeta_j)=N(N-1)I_W^\sharp[\mu_N]+\sum_{j=1}^N W(\zeta_j).$$

We now come to a simple but useful observation. Write $Q_N^*(\zeta)=Q_N(\re\zeta)$. Recalling that $Q_N^*\le Q_N$ on $\C$ we obtain
\begin{equation}\label{monotone}I_{Q_N^*}^\sharp[\mu_N]\le I_{Q_N}^\sharp[\mu_N].\end{equation}

Following \cite{A,J} we next fix a small $\epsilon>0$ and potentials $Q_N$ and $V$ obeying the conditions in Subsection \ref{tld}.
We then consider the event
$$A(N,\epsilon)=\{I_{Q_N}^\sharp(\mu_N)\le \gamma(V)+\epsilon\}.$$

\begin{lem} \label{ldp} Fix $a\ge 0$ and suppose that the lower bound condition \eqref{e1} is satisfied.
Then there is a positive integer $N_0$ depending on $\epsilon$ but not on $a$ such that if $N\ge N_0$,
$$\bfP_N^\beta(\{\mu_N\not\in A(N,\epsilon+a)\})\le e^{-\frac 1 2 \beta aN^2}.$$

\end{lem}

\begin{proof} If $\mu_N\not\in A(N,\epsilon+a)$, then
\begin{equation}\label{con1}I_{Q_N}^\sharp[\mu_N]\ge \gamma(V)+\epsilon+a.\end{equation}

By \eqref{extra growth} and the elementary inequality $|\zeta-\eta|^2\le (1+|\zeta|^2)(1+|\eta|^2)$
we find that
$L_{Q_N}(\zeta,\eta)\ge \tfrac {c_1} 2(Q_N(\zeta)+Q_N(\eta))-c_2,$
where $c_1=1-1/k>0$ and $c_2=C/k.$ This gives
\begin{equation}\label{con2}I_{Q_N}^\sharp[\mu_N]\ge\tfrac {c_1}N\sum_{j=1}^N Q_N(\zeta_j)-c_2.\end{equation}

Now fix a small number $\theta>0$ and take a convex combination of the inequalities \eqref{con1} and \eqref{con2}. We obtain
that if $\mu_N\not\in A(N,\epsilon+a)$,
\begin{equation}\label{above}
I_{Q_N}^\sharp[\mu_N]\ge (1-\theta) (\gamma(V)+\epsilon+a)+\theta(\tfrac {c_1}N\sum_{j=1}^N Q_N(\zeta_j)- c_2),
\end{equation}
which implies that the Hamiltonian \eqref{hamn} satisfies
\begin{align*}
\begin{split}
\Ham_N
 \label{key}
&\ge   N(N-1)(1-\theta)(\gamma(V)+\epsilon+a)\\
 &+k((N-1)\theta c_1+1)\sum_{j=1}^N \log(1+|\zeta_j|^2)-N(\theta c_3N+c_4).
\end{split}
\end{align*}
Consequently,
\begin{align*}
\int_{\C^N\setminus A(N,\epsilon+a)}e^{-\beta \Ham_N}\, dA_N&\le e^{-\beta N(N-1)(1-\theta)(\gamma(V)+\epsilon+a)+\beta N(\theta c_3N+c_4)}\\
&\qquad \times [\int_\C(1+|\zeta|^2)^{-k\beta((N-1)\theta c_1+1)}\, dA(\zeta)]^N.
\end{align*}
Since
$\int_\C(1+|\zeta|^2)^{-\alpha}\, dA(\zeta)=\tfrac 1 {\alpha-1}$ for  $\alpha>1$, the integral in brackets is no larger than $1$, so
we obtain that for all large enough $N$,
\begin{align*}\bfP_N^\beta(\{\mu_N\not\in A(N,\epsilon+a)\})&\le \tfrac 1 {Z_N}e^{-\beta N(N-1)(1-\theta)(\gamma(V)+\epsilon+a)+\beta N(\theta c_3N+c_4)}.
\end{align*}

Next we use the lower bound
$Z_N\ge \exp(-N^2(\beta\gamma(V)+o(1)))$ from assumption \eqref{e1}.
This implies that
$$
\bfP_N^\beta(\{\mu_N\not\in  A(N,\epsilon+a)\}) \le e^{  N(N-1)[ \beta(\theta(\gamma(V)+c_3+\epsilon)-(1-\theta)a)+o(1)  ]   }.
$$
Finally we fix $\theta$ with $0<\theta<\tfrac 18$ and $N_0$ such that for all $N\ge N_0$,
$\theta(\gamma(V)+c_3+\epsilon)+o(1)<\tfrac a8.$
Then for large enough $N$,
$$
\bfP_N^\beta(\{\mu_N\not\in  A(N,\epsilon+a)\}) \le e^{ N(N-1)[ \beta(\frac a 8-(1-\theta)a)  ]   }\le e^{-\frac 1 2 \beta aN^2},
$$
which completes the proof.
\end{proof}

Now fix a constant $C>0$ and write $S_{Q_N}^{\,+}=\{\zeta\in\C\,;\, \dist(\zeta,S_{Q_N})\le C\tfrac {\log N}N\}.$

\begin{lem} \label{ua}
Suppose that for each $N$ we have $\mu_N\in A(N,\epsilon)$ and $\supp \mu_N\subset S_{Q_N}^{\,+}$.
Suppose also
that $V|_\R$ is continuous in an $\R$-neighbourhood of $S_V$.
Then each subsequence of the measures $\mu_N$ has a further subsequence converging weak* to a probability measure $\mu$ supported on $\R$
such that $I_V[\mu]\le\gamma(V)+\epsilon$.
\end{lem}

\begin{proof} By hypothesis, the measures $\mu_N$ are all supported in some large compact set.
Pick a subsequence $\mu_{N_k}$ converging weak* to a probability measure $\mu$ which is necessarily supported on $\R$.
Also pick $\eps>0$ and write $Q_N^*(\zeta)=Q_N(\re\zeta)$. We also write $V^*(\zeta)=V(\re\zeta)$. By \eqref{monotone} we know that
$I_{Q_N^*}^\sharp[\mu_N]\le \gamma(V)+\epsilon$.

 Now recall that the $Q_N^*$ increase monotonically to $V^*$ and that $V^*$ is continuous on some compact set $K$ which contains the sets $S_{Q_N}^{\,+}$ for all large $N$.
 Hence the convergence is uniform on $K$,
by Dini's theorem. In particular
we can find an integer $n$
such that $\sup\{|Q_N^*(\zeta)-V^*(\zeta)|\,;\,\zeta\in K\}<\eps$ when $N\ge n$.
Then for $N_k\ge n$,
$|\mu_{N_k}(Q_{N_k}^*-V^*)|<\eps$.
Choosing $n$ larger if necessary, we can arrange that $|\mu_{N_k}(V^*)-\mu(V)|<\eps$, since $V^*$ is continuous on  $K$ and $\mu(V)=\mu(V^*)$.
Hence $|\mu_{N_k}(Q_{N_k}^*)-\mu(V)|<2\eps$ when $N_k\ge n$.
We have shown that
$\mu_{N_k}(Q_{N_k}^*)\to\mu(V)$ as $k\to\infty$.

From here on, we follow a routine argument which can be found in \cite[p. 146]{ST} for example. For a large real number $M$, we put
$L_M(\zeta,\eta)=\min\{M,\log\tfrac 1 {|\zeta-\eta|}\}.$

Since $L_M$ is continuous while $\mu_{N_k}\to\mu$ weak*,
\begin{align*}\iint\log \frac 1 {|\zeta-\eta|}\, d\mu(\zeta) \, d\mu(\eta)&=
\lim_{M\to\infty}\lim_{k\to\infty}\iint L_M(\zeta,\eta)\, d\mu_{N_k}(\zeta) \, d\mu_{N_k}(\eta)
\\
&\le\lim_{M\to\infty}\limsup_{k\to\infty}\{\tfrac 1 {N_k^2}\sum_{j\ne l}\log\tfrac 1 {|\zeta_j-\zeta_l|}+\tfrac M {N_k}\}
\\
&=\limsup_{k\to\infty}\tfrac 1 {N_k(N_k-1)}\sum_{j\ne l}\log\tfrac 1 {|\zeta_j-\zeta_l|}.
\end{align*}
Since $I_{Q_{N_k}^*}^\sharp[\mu_{N_k}]\le\gamma(V)+\epsilon$ it follows that
\begin{align*}I_V[\mu]&=\lim_{M\to\infty}\lim_{k\to\infty}(\iint L_M(\zeta,\eta)\, d\mu_{N_k}(\zeta)\, d\mu_{N_k}(\eta)+\mu_{N_k}(Q_{N_k}^*))
\\
&\le \limsup_{k\to\infty} I_{Q_{N_k}^*}^\sharp[\mu_{N_k}]\le \gamma(V)+\epsilon.
\end{align*}
The proof is complete.
\end{proof}

\begin{lem} \label{local} Let $d_N(\zeta)=\dist(\zeta,S_{Q_N})$ and put
$D_N=\max_{1\le j\le N}\{d_N(\zeta_j)\}.$
Then for each $q>0$ there exists $C>0$ and $N_0$ so that for $N\ge N_0$ we have
$$\bfP_N^\beta(\{D_N>C\tfrac {\log N} N\})\le \tfrac 1 {N^{q}}.$$
\end{lem}

\begin{proof} This follows from \cite[Theorem 2]{A}. More precisely, in the estimate \cite[(1.11)]{A} (with $Q=Q_N$),
we take $c=c_1N$ where $c_1>0$ is independent of $N$. Then taking $t=A\log N$ with a large enough $A$ and defining $C$ accordingly,
we finish the proof.
\end{proof}

We are now ready to finish our proof of Theorem \ref{joh1}.

For a fixed small $\epsilon>0$, we pick $N_0(\epsilon)$ so that if $N\ge N_0$ then, with large probability,
$\supp\mu_N\subset S_{Q_N}^+$ and $\mu_N\in A(N,2\epsilon)$. Indeed, by lemmas \ref{ldp} and \ref{local}, we can arrange that the probability of the complementary event is at most $N^{-q}+e^{-\frac 12\epsilon\beta N^2}$ where $q>0$ may be chosen as large as we please (by adjusting the value of $C$).

Now let $\epsilon=\epsilon_n\to 0$ through a suitable sequence, and
for each $n$, use Lemma \ref{ua} to pick a weak* subsequential limit $\mu^{\epsilon_n}$ of the measures $\mu_N$, with $I_V[\mu^{\epsilon_n}]\le \gamma(V)+2\epsilon_n$ and $\mu^{\epsilon_n}$ supported
in some fixed compact subset of $\R$.

Letting $n\to \infty$, the corresponding measures $\mu^{\epsilon_n}$ will converge
weak* along a subsequence to a measure $\mu$ with $I_V[\mu]\le \gamma(V)$. This implies that $\mu=\sigma_V$ by unicity of the equilibrium measure (cf.~\cite[Theorem I.1.3]{ST}).  We have shown that with probability $1-o(1)$, every subsequence of the measures $\mu_N$ has a further subsequence converging weakly to $\sigma_V$, which implies that the full sequence $\mu_N\to\sigma_V$ weakly as measures.

Now let $f$ be a bounded continuous function on $\R$ and define $f^*$ on $\C$ by $f^*(\zeta)=f(\re\zeta)$.
Then the uniformly bounded random variables $\mu_N(f^*)$ converge to $\sigma_V(f^*)=\sigma_V(f)$ in probability as $N\to\infty$.
Taking expectations we obtain $\bfE_N^\beta[\mu_N(f^*)]\to \sigma_V(f)$ as $N\to\infty$. However, by Fubini's theorem,
\begin{align*}\bfE_N^\beta[\mu_N(f^*)]&=\tfrac 1 N\bfE_N^\beta[f(\re\zeta_1)+\ldots+f(\re \zeta_N)]\\
&=\tfrac 1 \pi \tfrac 1 N\int_\R f(\xi)\, d\xi\int_\R \bfR_N^\beta(\xi+i\eta)\, d\eta\\
&=\tfrac 1 \pi \int_\R f(\xi)\, c_N^\beta(\xi)\, d\xi,\end{align*}
so we obtain the desired convergence $c_N^\beta\to \pi\sigma_V$ in the sense of measures on $\R$. $\qed$

\subsection{An entropy estimate} The lower bound on the partition function in \eqref{e1} can be checked directly for our model cases of AGUE and ALUE.
However, the stronger conditions in the following lemma are usually easier to check (e.g.~in the case of AGUE).

\begin{lem} \label{lem4} Suppose that the equilibrium measures $\sigma_{Q_N}$ obey the following energy and entropy limits:
\begin{equation}\label{eel}\lim_{N\to\infty}\tfrac 1 N\sigma_{Q_N}(Q_N)=0,\qquad \lim_{N\to\infty}\tfrac 1 N\sigma_{Q_N}(\,\log\Delta Q_N)=0.\end{equation}
Then the lower bound \eqref{e1} holds, i.e.,
$\liminf_{N\to\infty}\tfrac 1 {N^2}\log Z_N^\beta\ge -\beta\gamma(V)
$.
\end{lem}

\begin{proof} We start by fixing a continuous compactly supported function $\fii$ with
$\int_{\C} \varphi \, dA=1$ and noting that
\begin{align*}
Z_N=\int_{\C^N} \exp\{ -\beta  \sum_{j\ne k}L_{Q_N}(\zeta_j,\zeta_k)- \sum_{j=1}^N (\beta Q_N(\zeta_j)+\log \fii(\zeta_j)) \} \prod_{j=1}^N \, \varphi(\zeta_j)\, dA(\zeta_j).
\end{align*}

By Jensen's inequality,
\begin{align*}
\log Z_N &\ge \int_{\C^N} \{ -\beta  \sum_{j\ne k}L_{Q_N}(\zeta_j,\zeta_k)- \sum_{j=1}^N (\beta Q_N+\log \fii)(\zeta_j) \} \prod_{j=1}^N \, \varphi(\zeta_j)\, dA(\zeta_j)
\\
&= -\beta N(N-1)I_{Q_N}[\fii]- N \int_{\C} ( \beta \, Q_N+\log \varphi ) \, \varphi\, dA,
\end{align*}
with the understanding that $0\log 0=0$, and where we write $I_{Q_N}[\fii]$ in place of $I_{Q_N}[\fii\, dA]$.
This leads to
\begin{align}\label{babbel}
\tfrac 1 {N^2}\log Z_N
&\ge -\beta(1-\tfrac 1 N) I_{Q_N}[\fii]-\tfrac \beta N\int Q_N\varphi\, dA-\tfrac 1 N\int\varphi\log\varphi\, dA.
\end{align}

For small $\delta>0$ we let $\chi(\zeta)=\delta^{-2}\1_{D(0,\delta)}(\zeta)$ and define
$$\fii_{\delta,N}(\zeta)=\chi*\sigma_{Q_N}(\zeta)=\tfrac{\sigma_{Q_N}(D(\zeta,\delta))}{\delta^2}.
$$
As $\delta\to 0$ we have that $I_{Q_N}[\fii_{\delta,N}]\to I_{Q_N}[\sigma_{Q_N}]=\gamma(Q_N)$ (see remark in \cite[Appendix A]{A}).

Thus letting $\delta\to 0$ the right hand side in \eqref{babbel} converges to
\begin{align}\label{dough}-\beta(1-\tfrac 1 N) \gamma(Q_N)
-\tfrac 1 N\int_{S_{Q_N}} Q_N\Delta Q_N\, dA-\tfrac 1 N \int_{S_{Q_N}}\Delta Q_N\, \log \Delta Q_N\, dA.
\end{align}
Since $\gamma(Q_N)\uparrow\gamma(V)$ as $N\to\infty$ (cf. \cite[Theorem I.6.2]{ST}), we see that \eqref{dough} can be estimated from below by $-\beta \gamma(V)+o(1)$, due to the assumptions \eqref{eel}.
\end{proof}

\section{Generalities about the determinantal case} \label{Sec3}

Let $(Q_N)$ be any sequence of potentials satisfying the conditions in Subsection \ref{tld}, and set $\beta=1$.

We shall adapt to the present setting some basic techniques from the paper \cite{AKM} pertaining to weighted polynomials, Ward's equation, and so forth. The adaptations are straightforward, and we shall be correspondingly brief.

Recall that the $\beta=1$ process is determinantal, i.e., the intensity $k$-point functions of $\{\zeta_j\}_1^N$ can be represented as determinants
$\bfR_{N,k}(\eta_1,\ldots,\eta_k)=\det(\bfK_N(\eta_i,\eta_j))_{k\times k},$
where $\bfK_N$ is a correlation kernel.

Indeed, as is well-known, we can take $\bfK_N$ to be the reproducing kernel of the subspace $\calW_N$ of $L^2=L^2(\C,dA)$
consisting of all weighted polynomials $f=q\cdot e^{-NQ_N/2}$, where $q$ is a holomorphic polynomial of degree at most $N-1$. (See e.g. \cite[Section IV.7]{ST}.)
This \textit{canonical correlation kernel}
is used without exception below; note in particular that
$\bfR_N(\zeta)=\bfK_N(\zeta,\zeta).$

\subsection{Auxiliary estimates} \label{BCS} We now record a few basic estimates for weighted polynomials, which are proved using standard techniques, e.g.~in \cite[Section 3]{AKM}.
We begin with the following pointwise-$L^2$ estimate.

\begin{lem} \label{spike} Let $f=q\cdot e^{-NQ_N/2}\in\calW_N$ and fix a point $\zeta_0$ at which $\Delta Q_N(\zeta_0)>0$. Suppose that there exists a constant $C_0=C_0(\zeta_0)$ such that
\begin{equation}\label{jupp}\Delta Q_N(\zeta)\le C_0\Delta Q_N(\zeta_0)\qquad \text{when}\qquad |\zeta-\zeta_0|\le\tfrac 1 {\sqrt{N\Delta Q_N(\zeta_0)}}.\end{equation}
Then there exists a constant $C=C(C_0)$ such that
\begin{equation}\label{ching}|f(\zeta_0)|^2\le CN\Delta Q_N(\zeta_0)\int_{D(\zeta_0,\frac 1 {\sqrt{N\Delta Q_N(\zeta_0)}})}|f|^2\, dA.\end{equation}
Moreover (with the same $C$)
$\bfR_N(\zeta_0)\le CN\Delta Q_N(\zeta_0).$
\end{lem}

\begin{proof} Consider the function
$F(\zeta)=|f(\zeta)|^2e^{NC_0\Delta Q_N(\zeta_0)|\zeta-\zeta_0|^2}.$
By hypothesis, $F$ is logarithmically subharmonic in the disc $D(\zeta_0,\tfrac 1 {\sqrt{N\Delta Q_N(\zeta_0)}})$ since
$$
\Lap \log F(\zeta) = \Lap |f(\zeta)|^2+ NC_0\Delta Q_N(\zeta_0) = - N ( \Lap Q_N(\zeta) - C_0\Delta Q_N(\zeta_0) ) \ge 0 .
$$
An application of the sub-mean inequality gives
\begin{align*}|f(\zeta_0)|^2&\le N\Delta Q_N(\zeta_0)\int_{D(\zeta_0,\frac 1 {\sqrt{N\Delta Q_N(\zeta_0)}})}F\, dA\\
&\le e^{C_0}N\Delta Q_N(\zeta_0)\int_{D(\zeta_0,\frac 1 {\sqrt{N\Delta Q_N(\zeta_0)}})}|f|^2\, dA,
\end{align*}
which proves \eqref{ching} with $C=e^{C_0}$.

To verify the last statement, we consider the weighted polynomial
$f(\zeta)=\tfrac {\bfK_N(\zeta,\zeta_0)}{\sqrt{\bfK_N(\zeta_0,\zeta_0)}}$, which satisfies $|f(\zeta_0)|^2=\bfR_N(\zeta_0)$. Applying \eqref{ching} to this $f$, observing that
$\int_\C|f|^2=1$, we immediately
obtain the inequality $\bfR_N(\zeta_0)\le CN\Delta Q_N(\zeta_0)$.
\end{proof}

It is convenient to prove a stronger form of Lemma \ref{spike}, which incorporates the decay of $\bfR_N$ outside of the droplet $S_{Q_N}$. For this, we recall some standard potential theoretic
notions.

 Given an external potential $W$, we define the \textit{obstacle function} $\check{W}(\zeta)$ as the supremum of $s(\zeta)$ where $s$ runs through the set of subharmonic functions
on $\C$ which satisfy $s\le W$ everywhere and $s(\zeta)\le \log|\zeta|^2+O(1)$ as $\zeta\to\infty$. The coincidence set $S_W^*=\{W=\check{W}\}$ contains the droplet $S_W$, but may in general be slightly larger (see \cite[p.144]{ST}). The obstacle function $\check{W}$ is harmonic in the complement of the droplet $S_W$ and is $C^{1,1}$-smooth on $\C$, see \cite{ST}.

\begin{lem} \label{spike2} If $f\in\calW_N$ and $|f|\le C$ on $S_{Q_N}$, then $|f|\le Ce^{-N(Q_N-\check{Q}_N)/2}$ on $\C$.
\end{lem}

\begin{proof} This is a standard fact (e.g.~\cite{ST}) but it is easy enough to recall a proof. We can assume that $C=1$. Writing $f=q\cdot e^{-NQ_N/2}$ where $q$ has degree at most $N-1$,
we consider the subharmonic function $s=\tfrac 1 N \log|q(\zeta)|^2$ which satisfies $s(\zeta)\le \log|\zeta|^2+O(1)$ as $\zeta\to\infty$. Since $|f|\le 1$ on $\C$ we see that
$s(\zeta)=\tfrac 1 N\log|f(\zeta)|^2+Q_N(\zeta)\le Q_N(\zeta)$ on $\C$. A suitable version of the maximum principle shows that $s\le \check{Q}_N$ on $\C$, finishing the proof of the lemma.
\end{proof}

\begin{lem} \label{spike3} Suppose in addition to \eqref{jupp} that there is a constant $C_1$ such that
$\Delta Q_N\le C_1N$ on $S_{Q_N}$. Then there is a constant $C$ such that, for all $\zeta$ in a neighbourhood of the droplet,
$\bfR_N(\zeta)\le CN^2e^{-N(Q_N-\check{Q}_N)(\zeta)}.$
\end{lem}

\begin{proof} Fix a point $\zeta_1$ in a neighbourhood of the droplet, and take $f(\zeta)=\tfrac {\bfK_N(\zeta,\zeta_1)}{\sqrt{\bfK_N(\zeta_1,\zeta_1)}}$. Next fix $\zeta_0\in S_{Q_N}$ and use \eqref{ching} to conclude
that
$|f(\zeta_0)|^2\le CC_1N^2\int_\C |f|^2\, dA=CC_1N^2.$
Since $\zeta_0\in S_{Q_N}$ was arbitrary, Lemma \ref{spike2} shows that
$|f(\zeta)|^2\le CC_1N^2e^{-N(Q_N-\check{Q}_N)(\zeta)}$.

Taking $\zeta=\zeta_1$ we find that $\bfR_N(\zeta)=\bfK_N(\zeta,\zeta)\le CC_1N^2e^{-N(Q_N-\check{Q}_N)(\zeta)}$, as desired.
\end{proof}

\begin{rmk*} We stress that the conditions of Lemma \ref{spike3} holds for AGUE (since $\Delta Q_N$ is then essentially a constant multiple of $N$), but not for ALUE.
\end{rmk*}

\begin{lem} \label{bl00} Suppose that the conditions of Lemma \ref{spike3} are satisfied. Then the cross-sections
$c_N(p)=\tfrac 1 N\int_\R \bfR_N(p+i\eta)\, d\eta$ are uniformly bounded on $\R$.
\end{lem}

\begin{proof}Write $d_N(\zeta)=\dist(\zeta,S_{Q_N})$.
By Taylor's formula,
\begin{equation}\label{beckom}Q_N(\zeta)-\check{Q}_N(\zeta)\ge M_N\cdot d_N(\zeta)^2,\qquad M_N=\min\{ \Delta Q_N(\zeta)\,;\, \zeta\in S_{Q_N}\}\end{equation}
for all $\zeta$ close enough to the droplet, see e.g. \cite[Section 2]{A}. By assumption we have $M_N\ge mN$ for some $m>0$, so
there is a constant $C$ such that
\begin{equation}\label{beg}c_N(p)\le CN\int_{-\infty}^{+\infty}e^{-mN^2d_N(p+i\eta)^2}\, d\eta.\end{equation}
By Theorem \ref{aplem}, the droplet $S_{Q_N}$ is contained in a strip $|\eta|\le \tfrac AN$, so the right hand side in \eqref{beg} can be estimated by
$C(2A+2\int_0^\infty e^{-my^2}\, dy)<\infty.$
\end{proof}

\subsection{Rescaling in the bulk} \label{RIB} Now fix a point $p_*\in\R$ with $\sigma_V(p_*)>0$, where as always $V=\lim Q_N$.

We admit potentials such that $\Delta Q_N$ has singularities (such as for ALUE),
but we assume that each $\Delta Q_N$ is finite at the point $p_*$
and that there are positive constants $c_1$ and $c_2$ (with $c_2$ possibly depending on $p_*$)
such that $c_1N\le \Delta Q_N(\zeta)\le c_2 N$ for all $\zeta\in D(p_*,\tfrac 1 {\sqrt{N\Delta Q_N(p_*)}})$ and all large $N$.

We then rescale by the blow-up map $\Gamma_N$ in \eqref{blowup}, i.e.,
we rescale the system $\{\zeta_j\}_1^N$ by magnifying distances by the factor ${\scriptstyle \sqrt{N\Delta Q_N(p_*)}}$, about the origin $p_*$. The resulting process
is denoted $\{z_j\}_1^N$ and we write $R_N$ for its $1$-point intensity.

It is easy to see that the rescaled process $\{z_j\}_1^N$ is determinantal as well, with (canonical) correlation kernel
\begin{equation*}K_N(z,w)=\tfrac 1 {N\Delta Q_N(p_*)}\bfK_N(\zeta,\eta),\qquad z=\Gamma_N(\zeta),\quad w=\Gamma_N(\eta).\end{equation*}

The following lemma is a slight modification of \cite[Theorem 1.1]{AKM}. A proof can be accomplished by arguments from \cite[Section 3]{AKM} (with ``$Q_N$'' in lieu of ``$Q$''). We omit repeating those details here, but merely point out that the local boundedness in Lemma \ref{spike} is sufficient for the normal families argument in \cite[Section 3]{AKM} to carry through.

\begin{lem} \label{LPF} (``\emph{Structure of limiting kernels}'') Under the above conditions, there exists a sequence $c_N$ of cocycles such that each subsequence
of the kernels $c_NK_N$ has a further subsequence which converges locally uniformly to a Hermitian kernel $K$ of the form
$K(z,w)=G(z,w)\cdot L(z,w),$
where $L$ is a Hermitian-entire function and $G$ is the Ginibre kernel \eqref{gin0}.
\end{lem}

Here we have used the following notation (cf.~\cite{AKM}).
A continuous function $f(z,w)$ is called \textit{Hermitian} if $f(z,w)=\overline{f(w,z)}$. If $f$ is also analytic
(or entire) in $z$ and $\bar{w}$, then $f$ is called \textit{Hermitian-analytic} (or \textit{Hermitian-entire}). A Hermitian function of the form
$f(z,w)=g(z)\overline{g(w)}$, where $g$ is continuous and unimodular, is said to be a cocycle.

By the convergence in Lemma \ref{LPF} it follows that each limiting kernel $K$ appearing in Lemma \ref{LPF} is the correlation kernel of a unique determinantal point field $\{z_j\}_1^\infty$.
In order to determine $K$ it clearly suffices to determine the \textit{limiting 1-point function} $R(z)=L(z,z)$, for then $L$ is determined by $R$ via polarization, and so
$K=LG$ is also determined.

In what follows the subsequential limits $R=\lim R_{N_k}$ play a fundamental role. We refer to such limits as limiting 1-point functions about $p_*$.

Now recall our assumption that the limit $\rho(p_*)$ in \eqref{rhop} exists, i.e., we assume that
$\tfrac 1 {\sqrt{N\Delta Q_N(p_*)}}=\sqrt{\tfrac N {\Delta Q_N(p_*)}}\cdot \tfrac 1 N=\tfrac {\rho(p_*)} N\cdot (1+o(1))$ as $N\to\infty$.

We immediately obtain the following rescaled version of the pointwise cross-section convergence.

\begin{lem} \label{LPF2} If the cross-section $c_N(p_*)$ converges to $\pi\cdot \sigma_V(p_*)$, then
for each limiting $1$-point function $R$ at $p_*$ and each $x\in\R$,
\begin{equation*}\int_{\R}R(x+iy)\, dy=\pi\cdot \rho(p_*)\cdot \sigma_V(p_*).\end{equation*}
\end{lem}

Given a limiting kernel $K$, we define the \emph{Berezin kernel} $B$ via
$B(z,w)=\tfrac {|K(z,w)|^2}{R(z)}$ and we put
$C(z)=\int_\C\tfrac {B(z,w)}{z-w}\, dA(w).$

\begin{thm}\label{ecor} Suppose that $(Q_N)$ is an admissible sequence satisfying the cross-section convergence. Fix $p_*\in\R$ where $\sigma_V(p_*)>0$ and $\rho(p_*)>0$ and rescale as above. Then each limiting 1-point function $R=\lim R_{N_k}$ is nontrivial, i.e., $R>0$ everywhere on $\C$. Furthermore, $R$ gives rise to a solution to Ward's equation
$\dbar C(z)=R(z)-1-\Delta\log R(z).$
\end{thm}

\begin{proof} The proof again follows by appealing to general results from \cite{AKM}.
Indeed, it follows from Lemma \ref{LPF2} that a limiting 1-point function $R$ cannot vanish identically. Then in fact $R>0$ everywhere and Ward's equation
holds in view of the zero-one law in \cite[Theorem 1.3]{AKM}.
\end{proof}

\subsection{Characterization of translation invariant scaling limit} We now come to the important realization that the cross-section convergence uniquely
determines a limiting 1-point function $R$, provided that the latter can be shown to be translation invariant. Indeed, this now follows in a straightforward
way by appealing to the theory for Ward's equation in \cite{AKM}.

Similar as in \cite{AKM}, we shall say a limiting 1-point function $R$ in Lemma \ref{LPF} is \textit{(horizontal)
translation invariant} if $R(z+t)=R(z)$ for all $t\in\R$.

\begin{thm}\label{tithm} Keep the assumptions of Theorem \ref{ecor}. Then each translation invariant limiting 1-point function $R$ at $p_*$
is of the form
\begin{align}\label{rett}R(z)
=\gamma*\1_{(-2a,2a)}(2\im z)=
\tfrac 1 {\sqrt{2\pi}}\int_{-2a}^{\,2a}e^{-\frac 1 2 (2\im z-t)^2}\, dt,
\quad  a
 =\tfrac \pi 2\cdot \rho(p_*)\cdot\sigma_V(p_*).\end{align}
\end{thm}

\begin{proof} The formula \eqref{rett} (for some $a>0$) follows immediately from the characterization of translation invariant solutions to Ward's equation in \cite[Theorem 1.6]{AKM}.
Moreover, Lemma \ref{LPF2} fixes the value of $a$ in \eqref{rett}.
\end{proof}

\subsection{Universality of almost-circular ensembles} The theory mainly described in the almost-Hermitian setup so far applies to the almost-circular setup as well.

\begin{proof}[Proof of Theorem~\ref{Thm_AUE}]
Since $Q_N$ is radially symmetric, the limiting empirical distribution is uniform on the unit circle.
Furthermore, by \eqref{rho} and \eqref{rN ACE}, we have
	\begin{equation}
	\tfrac{1}{\pi}c_N = \tfrac{1}{\pi} (1-r_N) \Delta Q_N(1) \cdot (1+o(1)) =  \tfrac{1}{2\pi}  \cdot (1+o(1)), \qquad (N \to \infty).
\end{equation}
This in particular gives rises to the rescaled version of the pointwise cross-section convergence
	\begin{equation}
	\int_\R R(x+iy)\,dy=\tfrac{\rho}{2}.
\end{equation}
Moreover, since $Q_N$ is radially symmetric, it is easy to show the translation invariance, see e.g., \cite[Subsection 6.4]{AKM}.
Now Theorem~\ref{tithm} completes the proof.
\end{proof}

\section{Pointwise convergence of cross-sections} \label{Sec4}

In this section we prove Theorem \ref{mth0} and Theorem \ref{mthMP} about pointwise convergence of cross-sections to $\pi$ times the equilibrium measure $\sigma_V(p)$.

\subsection{Cross-sections for AGUE} \label{SCAGUE} Let
$$Q_N(\zeta)=\tfrac 1 2\xi^2+\tfrac 1 2\cdot \tfrac N {c^2}\eta^2$$
and fix numbers $\delta,\alpha$ with $0<\delta<\alpha<2$. Put $I_{\alpha,\delta}=\{\xi\in\R\,;\, \delta\le |\xi|\le \alpha\}$.

We shall prove in detail that $\tfrac 1 \pi c_N(p)\to \SC(p)$ uniformly for $p\in I_{\alpha,\delta}$. The pointwise convergence for $p=0$ may be
handled in a similar way, see a remark below.

We first recall that the 1-point function can be written
$\bfR_N(\zeta)=\sum_{j=0}^{N-1}|w_{j}(\zeta)|^2,$
where $w_{j}=q_{j}\cdot e^{-NQ_N/2}\in\calW_N$ is an orthonormal basis of the weighted polynomials subspace $\calW_N$ of $L^2$ (see Section \ref{Sec3}).

Consider now an arbitrary potential $Q$ of the form $Q(\xi+i\eta)=a\xi^2+b\eta^2$ where $a$ and $b$ are positive constants.

The following orthonormal polynomials in weight $e^{-NQ_N/2}$ are found in \cite{EM},
\begin{equation*}
q_j(\zeta)=(ab)^{\frac 1 4} \sqrt{\tfrac{N}{j!} } (\tfrac \tau 2)^{\frac j2} H_j( \sqrt{ \tfrac{Nab}{b-a} }\,\zeta), \quad \tau=\tfrac {b-a}{b+a},
\end{equation*}
where $H_j$ is the $j$:th Hermite polynomial,
$H_j(z)=(-1)^j e^{z^2} \tfrac{d^j}{dz^j}e^{-z^2}.$

In other words, the $1$-point function in potential $Q=a\xi^2+b\eta^2$ is given by
\begin{align}\label{Ge1pt} \begin{cases}\bfR_N(\zeta)=e^{-NQ(\zeta)}N\sqrt{ab}\cdot F_N(\sqrt{ \tfrac{Nab}{b-a} }\,\zeta),\cr
F_N(z)=\sum_{j=0}^{N-1}\tfrac {(\tau/2)^j}{j!}\left|H_j(z)\right|^{\,2}.\cr
\end{cases}
\end{align}

We have the following lemma. (See \cite[Proposition 2.3]{LR} for a related statement.)

\begin{lem}\label{Lem_Hermite}
With $x=\re z$ and $F_N$ as in \eqref{Ge1pt}, we have
\begin{align*}
\tfrac {\d F_N}{\d x}(z)=\tfrac{4\tau\,x}{1+\tau}F_N(z)-\tfrac{4(\tau/2)^N}{1+\tau} \tfrac{\Re[ H_{N-1}(z) H_N(\bar{z})]}{(N-1)!}.
\end{align*}
\end{lem}

\begin{proof}
We apply two standard facts for Hermite polynomials, namely the three-term recursion $H_{j+1}(z)=2z H_j(z)-H'_j(z)$ and the
identity $H'_j(z)=2jH_{j-1}(z)$.
These lead to
\begin{align}\label{lto}
\begin{split}
\tfrac {\d F_N}{\d z}
&=\sum_{j=0}^{N-1} \tfrac{(\tau/2)^j}{j!}  \cdot 2j H_{j-1}(z)\cdot \left( 2\bar{z}H_{j-1}(\bar{z})-H'_{j-1}(\bar{z}) \right)
\\
&=2\tau \bar{z} \sum_{j=0}^{N-2} \tfrac{(\tau/2)^j}{j!} | H_j(z) |^2-\tau \sum_{j=0}^{N-2} \tfrac{(\tau/2)^j}{j!} H_j(z)H'_j(\bar{z}).
\end{split}
\end{align}
Using once more the three-term recursion, the last sum is recognized as
\begin{align*}
\tfrac {\d F_N}{\d\bar{z}}(z)-2 \bar{z}  \tfrac{(\tau/2)^{N-1}}{(N-1)!} |H_{N-1}(z)|^2
+ \tfrac{(\tau/2)^{N-1}}{(N-1)!}  H_{N-1}(z) H_N(\bar{z}).
\end{align*}
Inserting this in \eqref{lto}, we obtain after some straightforward manipulations,
\begin{align} \label{th}
\begin{split}
\tfrac {\d F_N}{\d z}(z)=2\tau \bar{z} F_N(z)-\tau \tfrac {\d F_N}{\d\bar{z}}(z)-2\tfrac{(\tau/2)^N}{(N-1)!} H_{N-1}(z) H_N(\bar{z}).
\end{split}
\end{align}
The lemma follows from \eqref{th} by taking real parts and rearranging.
\end{proof}

Now consider the cross-section $c_N$, which we write in the form
\begin{equation}\label{kross}c_N(\xi)= \tfrac 1 {N^2}\int_{\R} \bfR_N(\xi+i\tfrac y N)\, dy.\end{equation}

Fix a point $\xi\in I_{\alpha,\delta}$.

By \eqref{Ge1pt} with $a=\tfrac 1 2$ and $b=\tfrac 1 2 \tfrac N {c^2}$ we have
$$
\tfrac 1 N \bfR_{N}(\xi+i\tfrac y N) =\tfrac{\sqrt{N}}{2c}\, e^{-\frac N 2 \xi^2-\frac {1}{2c^2} y^2} \,  F_N(  \tfrac{N}{\sqrt{2(N-c^2)}}  (\xi+i \tfrac y N) ).
$$
Applying Lemma~\ref{Lem_Hermite}, we obtain after some straightforward computations
\begin{equation}\label{juice}
\begin{split}
\tfrac 1 N \tfrac {\d\bfR_N} {\d\xi} (\xi+i\tfrac y N)=&-\tfrac{1}{c\sqrt{2}} \tfrac{N+c^2}{(N-1)!} \sqrt{\tfrac{N}{N-c^2}} e^{-\frac N 2 \xi^2-\frac 1 {2 c^{2}} y^2}  ( \tfrac{N-c^2}{2(N+c^2)} )^N
\cr
&\times \Re \{  H_{N-1}( \tfrac{N}{\sqrt{2(N-c^2)}} (\xi+i\tfrac y N)  ) H_{N}( \tfrac{N}{\sqrt{2(N-c^2)}} (\xi-i\tfrac y N)  )  \}.\cr
\end{split}
\end{equation}

Thus we have asymptotically, as $N\to\infty$,
\begin{align}\label{PL_as}
\nonumber \tfrac 1 {N^2}\tfrac {\d\bfR_N} {\d \xi}(\xi+i\tfrac y N)=&-(1+o(1))\cdot \tfrac{1}{c\sqrt{2}}  e^{-\frac N 2 \xi^2-\frac 1 {2 c^{2}} y^2-2 c^2} \tfrac{ 1 }{2^N(N-1)!}
\\
&\times \Re \{  H_{N-1}(  \sqrt{\tfrac N 2}\xi+i\tfrac{y }{\sqrt{2N}}  ) H_{N}(  \sqrt{\tfrac N 2}\xi-i\tfrac{y }{\sqrt{2N}}  )  \}.
\end{align}

We now apply the inequality \eqref{f1p} in \eqref{PL_as}, which gives that there are constants $C=C(\alpha,\delta)$ and $k=k(\alpha,\delta)$ such that whenever $\xi\in I_{\alpha,\delta}$,
\begin{equation*}\tfrac 1 {N^2}\tfrac {\d\bfR_N} {\d \xi}(\xi+i\tfrac y N)\le Ce^{ -\frac 1 {2c^2} y^2 }\max\{1,\tfrac {y^{2N}}{N^N}\} \cosh^2(ky).\end{equation*}

Using this and differentiating with respect to $\xi$ in \eqref{kross} we find that the derivative $c_N'$ satisfies
\begin{equation*}|c_N'(\xi)|\le C\int_\R e^{ -\frac 1 {2c^2} y^2 }\max\{1,y^{2N}N^{-N}\}\cosh^2(ky)\, dy\le C_1,\end{equation*}
with a new constant $C_1=C_1(\alpha,\delta,c)$. This proves equicontinuity of the functions $c_N$ on $I_{\alpha,\delta}$.

Since the functions $c_N$ are also uniformly bounded (Lemma \ref{bl00}) we can apply the Arzela-Ascoli theorem and conclude that each subsequence of the functions $c_N$ has a further subsequence which converges uniformly on $I_{\alpha,\delta}$
to a limit $c(\xi)$. In view of Theorem \ref{joh1}, we must then have $c=\pi\cdot\sigma_V$ on the interval $I_{\alpha,\delta}$
where $\sigma_V$ is the equilibrium measure in potential $V(\xi)=\tfrac 1 2 \xi^2$, i.e., $\sigma_V=\SC$ is Wigner's semi-circle law. We have thus shown
locally uniform convergence $c_N(\xi)\to\pi\sigma_V(\xi)$ for $\xi\in (-2,2)\setminus\{0\}$. It is not hard to obtain convergence also at $\xi=0$, by using the Mehler-Heine formulas in \eqref{herna15}. This detail may be left to the reader.
Our proof of Theorem \ref{mth0} is thereby complete. $\qed$

\subsection{Cross-sections for ALUE} \label{WEAK} We now prove Theorem \ref{mthMP} on the cross-sections of the ALUE.
To this end we fix a real number $c>0$ and a non-negative integer $\nu$ and take
\begin{equation}\label{wdef}
Q_N(\zeta)=-\tfrac 1 N\log \left\{ K_{\nu}(aN|\zeta|)\cdot |\zeta|^{\nu}\right\}-b\re\zeta,
\end{equation}
where we abbreviate
\begin{equation}\label{anbn}
a=a_N=\tfrac 1 {c^2}N, \quad  b=b_N=\tfrac 1 {c^2}N-1.
\end{equation}

Let us fix a small $\alpha>0$ and put $I_\alpha=[\alpha,4-\alpha]$.
We must prove that the cross-sections $c_N(\xi)$ converge uniformly on $I_\alpha$ to $\pi\cdot \MP(\xi)=\tfrac 1 {2\xi}\sqrt{\xi(4-\xi)}$.

For this purpose, we first recall from \cite{AB,O,Ake05} that
the $j$:th orthonormal polynomial $q_j$ in weight $e^{-NQ_N/2}$
can be expressed in terms of the Laguerre polynomial $L_j^\nu(z)=\tfrac{z^{-\nu} e^z }{j!} \tfrac{d^j}{dz^j} ( e^{-z} z^{j+\nu} )$ via
\begin{equation} \label{ONP W}
q_j(\zeta)=  C_N\cdot
\tau^{j}\sqrt{  \tfrac{j!}{(j+\nu)!}  }L_j^\nu(  \tfrac{a^2-b^2}{2b}N \zeta ),\qquad (\tau=\tau_N=\tfrac ba=1-\tfrac {c^2}N),
\end{equation}
where
$C_N=\sqrt{a N}   ( \tfrac{a^2-b^2}{2a}N )^{\frac {\nu+1}2}$.

Hence if we introduce the function
\begin{equation} \label{F tau Lag}
G_N(z)=\sum_{j=0}^{N-1} \tfrac{\tau^{2j}j! }{(j+\nu)!} |L_j^\nu(z)|^2,
\end{equation}
then the 1-point function in external potential $Q_N$ becomes
\begin{equation}\label{obt}\bfR_N(\zeta)=C_N^{\,2}\cdot G_N(\tfrac{a^2-b^2}{2b}N \zeta )\cdot e^{-NQ_N(\zeta)}.\end{equation}

Now pick $\xi$ in the interval $I_\alpha=[\alpha,4-\alpha]$, and put
$$\zeta=\xi+i\tfrac y N,\qquad (y\in\R).$$
Using \eqref{Bessel_asym} and \eqref{wish}, it follows from \eqref{obt} that
\begin{align*}
	\tfrac 1 {N^2} \bfR_N(\zeta)
	&= \sqrt{\tfrac{\pi}{2}} \tfrac{1}{c} N^{\nu}  \cdot  \xi^{\nu-\frac12} e^{-N\xi-\frac{y^2}{2c^2\xi}} \cdot  G_N( N\zeta ) \cdot (1+o(1)),
\end{align*}
where we have used
$$
N(\tfrac N{c^{2}}\cdot |\zeta|-(\tfrac N{c^{2}}-1)\cdot \re\zeta)= \tfrac{N^2}{c^{2}}\cdot \sqrt{ \xi^2+y^2/N^2 }-(\tfrac {N^2}{c^{2}}-N)\xi= N \xi+\tfrac{y^2}{2c^2 \xi}+O(\tfrac{1}{N}).
$$

We next use a summation identity found in \cite[(10.12.42)]{EMOT} to write
\begin{equation}\label{wrath}\tfrac{j!}{(j+\nu)!} |L_j^\nu(z)|^2=\sum_{k=0}^j \tfrac{|z|^{2k}}{k! (k+\nu)!} L_{j-k}^{\nu+2k}(z+\bar{z}).\end{equation}
As a consequence, we can write, in turn:
\begin{align*}
	G_N(z)&=\sum_{j=0}^{N-1} \tau^{2j} \sum_{k=0}^j \tfrac{|z|^{2k}}{k!(k+\nu)!} L_{j-k}^{\nu+2k} (z+\bar{z}),\\
\tfrac 1 {N^2} \bfR_N(\zeta)&= \sqrt{\tfrac{\pi}{2}} \tfrac{1}{c} N^{\nu}  \cdot \xi^{\nu-\frac12} e^{-N\xi-\frac{y^2}{2c^2\xi}}  \cdot  \sum_{j=0}^{N-1} \tau^{2j} \sum_{k=0}^j \tfrac{|N\zeta|^{2k}}{k!(k+\nu)!} L_{j-k}^{\nu+2k} (2N\xi) (1+o(1)) .
\end{align*}

Observing that
$|N\zeta|^{2k}=(N\xi)^{2k} (1+(\tfrac{y}{N\xi})^2)^k$ and integrating in $y$, we find for $\xi\in I_\alpha$
\begin{equation}\label{babyl}\begin{split}
	c_N(\xi)
	&=\sqrt{\tfrac{\pi}{2}} \tfrac{1}{c} N^{\nu}  \cdot \xi^{\nu-\frac12} e^{-N\xi}
	\sum_{j=0}^{N-1} \tau^{2j}
	\sum_{k=0}^j \tfrac{(N\xi)^{2k}}{k!(k+\nu)!} L_{j-k}^{\nu+2k} (2N\xi)\cr
&\times J(N,k)\cdot (1+o(1)),\quad \text{where}\quad
J(N,k)=\int_\R e^{-\frac{y^2}{2c^2\xi}}(1+(\tfrac{y}{N\xi})^2)^k\,dy.\cr
\end{split}
\end{equation}

We pause to evaluate the integral $J(N,k)$.

\begin{lem} \label{cs integral} $J(N,k)=c\cdot \sqrt{2\pi\,\xi}\cdot k! \, (\tfrac{-2c^2}{N^2\xi})^k \, L_k^{-k-1/2}( \tfrac{N^2\xi}{2c^2} ).$
\end{lem}

\begin{proof}
We shall use the following integral representation of the confluent Hypergeometric function $U(a,b,z)$ found in \cite[(13.4.4)]{OLBC}
$$
U(a,b,z)=\tfrac{1}{\Gamma(a)} \int_{0}^{\infty} e^{-zt} t^{a-1} (1+t)^{b-a-1}\,dt.
$$
By a change of variable and using $\Gamma(\tfrac12)=\sqrt{\pi}$, this leads to
$$
\int_{\R} e^{-ay^2}  (1+(by)^2)^k\,dy=\tfrac{\sqrt{\pi}}{b} U(\tfrac12,k+\tfrac32,\tfrac{a}{b^2})=\tfrac{\sqrt{\pi}}{b} (\tfrac{a}{b^2})^{-k-\tfrac12} U(-k,-k+\tfrac12,\tfrac{a}{b^2}).
$$
For the last identity, we also use the well-known (Kummer's) transformation
$$
U(a,b,z)=z^{1-b} \, U(a-b+1,2-b,z),
$$
see \cite[(13.2.29)]{OLBC}.
We finish the proof of the lemma by use of the functional relation $
U(-n,\nu+1,z)=(-1)^n \, n! \, L_n^\nu(z)
$ found in  \cite[(13.6.9)]{OLBC}.
\end{proof}

We next use the closed form of Laguerre polynomial in \cite[Chapter \RN{5}]{S},
$$
L_n^\nu(x)=\sum_{j=0}^n \tfrac{\Gamma(n+\nu+1)}{(n-j)!\Gamma(j+\nu+1)} \tfrac{(-x)^j}{j!}
$$
to infer that
$J(N,k)=c\cdot \sqrt{2\pi\,\xi}\cdot (1+O(\tfrac{1}{N^2})).$
Inserting this in \eqref{babyl}, we arrive at
\begin{equation}\label{csformula}\begin{split}
	c_N(\xi)&=\pi   (N\xi)^{\nu} \cdot e^{-N\xi} \cdot  \sum_{j=0}^{N-1} \tau^{2j} \sum_{k=0}^j \tfrac{(N\xi)^{2k}}{k!(k+\nu)!} L_{j-k}^{\nu+2k} (2N\xi) \cdot (1+o(1))\cr
&=\pi   (N\xi)^{\nu} \cdot e^{-N\xi} \cdot  \sum_{j=0}^{N-1} \tau^{2j}\tfrac {j!}{(j+\nu)!}[L_j^\nu(N\xi)]^2\cdot (1+o(1)),\cr
\end{split}
\end{equation}
where we used \eqref{wrath} with $z=N\xi$ to obtain the last equality.

On the other hand, we shall verify that
\begin{equation}\label{slarv}\lim_{N\to\infty}
(N\xi)^\nu e^{-N \xi} \sum_{j=0}^{N-1} \tfrac{j!}{(j+\nu)!} [L_j^\nu(N\xi)]^2 = \tfrac{1}{2\pi \xi} \sqrt{(4-\xi) \xi}.
\end{equation}

To prove \eqref{slarv}, we use the classical Christoffel-Darboux formula
$$
		\sum_{j=0}^{n-1} \tfrac{j! }{(j+\nu)!}  L_{j}^\nu (x) L_{j}^\nu (y) = \tfrac{ n! }{ (n-1+\nu)! } \tfrac{ 1 }{ x-y } (L_{n-1}^\nu(x)L_{n}^\nu(y)-L_{n-1}^\nu(y)L_{n}^\nu(x))
$$
together with differentiation rule $\tfrac d {dx} L_j^\nu(x)=-L_{j-1}^{\nu+1}(x)$ to obtain
\begin{align*}
(N\xi)^\nu e^{-N \xi} \sum_{j=0}^{N-1} \tfrac{j!}{(j+\nu)!} [L_j^\nu(N\xi)]^2&=(N\xi)^\nu e^{-N \xi} \tfrac{N!}{(N-1+\nu)!}
\\
&\times  ( L_{N-1}^\nu(N\xi)L_{N-1}^{\nu+1}(N\xi)-L_N^\nu(N\xi) L_{N-2}^{\nu+1}(N\xi) ).
\end{align*}

Next we apply the following Plancherel-Rotach type formula from \cite[Section \RN{3}]{FFG}:
for $x=4nX$ with $\frac{\ve}{n} \le X <1$ and fixed $m$,
\begin{equation}\label{PR Laguerre}
L_{n+m}^\nu(x)=x^{-\nu/2}e^{x/2} (-1)^{n+m} (2\pi \sqrt{X(1-X)})^{-1/2} n^{\nu/2-1/2} ( g_{n,m}^\nu(X)+O(\tfrac m n) ),
\end{equation}
where
$$
g_{n,m}^\nu(X)=\sin (2n(\sqrt{X(1-X)}-\arccos \sqrt{X})-(2m+\nu+1) \arccos\sqrt{X}+3\pi/4 ).
$$
Using \eqref{PR Laguerre}, we obtain that
\begin{align*}
&\quad L_{N-1}^\nu(N\xi)L_{N-1}^{\nu+1}(N\xi)-L_N^\nu(N\xi) L_{N-2}^{\nu+1}(N\xi)
\\
&=(N\xi)^{-\nu} e^{N\xi} N^{\nu-1} \cdot \tfrac{2}{\pi} \tfrac{1}{\xi \sqrt{4-\xi}} ( g_{N,-1}^\nu(X) g_{N,-1}^{\nu+1}(X)-g_{N,0}^\nu(X) g_{N,-2}^{\nu+1}(X) ) \cdot (1+o(1)),
\end{align*}
where $X=\frac{\xi}{4}$. Hence, using an elementary trigonometric reduction formula, we infer that
\begin{align*}
g_{N,-1}^\nu(X) g_{N,-1}^{\nu+1}(X)-g_{N,0}^\nu(X) g_{N,-2}^{\nu+1}(X)&=\tfrac12 ( \sqrt{X}-\cos( 3\arccos \sqrt{X} ) )
=\tfrac{1}{4} \sqrt{\xi} (4-\xi).
\end{align*}
This gives that the convergence \eqref{slarv} holds locally uniformly for $\xi \in (0,4)$.
(We refer to \cite[Propoistion 1]{GFF} or \cite[Subsection 7.2.3]{F} for further expansion of \eqref{slarv} up to order $O(1/N^2).$)

Combining \eqref{csformula} and \eqref{slarv}, using
that $\tau=a/b \to 1$ as $N \to \infty$, it follows that
$
c_N(\xi) \to \pi \cdot \MP(\xi)$ for all $\xi\in I_\alpha.$ This finishes our proof for Theorem \ref{mthMP}. $\qed$

\section{Bulk scaling limits} \label{BSL}
In this section we obtain bulk scaling limits for the almost-Hermitian GUE and LUE, i.e., we prove Theorem \ref{mainth1} and Theorem \ref{mthWish}, respectively.

\subsection{Bulk scaling limits for AGUE} \label{BAGUE} We now prove Theorem \ref{mainth1}.
To this end, recall that $Q_N(\zeta)=\tfrac 1 2 \xi^2+\tfrac 1 2\tfrac N {c^2}\eta^2$. We also fix a small parameter $\alpha$, $0<\alpha<2$ and
write $I_\alpha=[-2+\alpha,2-\alpha]$. Finally we fix a point $p_*\in I_\alpha$.

Let $R=\lim R_{N_k}$ be a limiting rescaled 1-point function about $p_*$ (guaranteed to exist by Lemma \ref{LPF}). By Theorem \ref{tithm} it suffices to prove that
$R$ is horizontal translation invariant, for then $R$ will be given by the explicit form in \eqref{rett}.

Next recall that $R_N(z)=\tfrac 1 {N\Delta Q_N}\bfR_N(\zeta)$, where $z={\scriptstyle \sqrt{N\Delta Q_N}}\cdot (\zeta-p_*)$, i.e.,
$$\zeta=p_*+ \tfrac {2c} N \, z\cdot (1+O(N^{-\frac 1 2})).$$
Thus using \eqref{Ge1pt} and Lemma~\ref{Lem_Hermite}
(with $a=\tfrac 1 2$, $b=\tfrac 1 2 \tfrac N {c^2}$, $\tau=\tfrac {b-a}{b+a}$)
we obtain after some elementary manipulations,
\begin{align}\label{RN pay}
\tfrac {\d R_N}{\d x}(z)&= -e^{-NQ_N(\zeta)} \tfrac{4a}{\sqrt{(b-a)\Delta Q_N}} \tfrac{( \tau/2 )^N} {(N-1)!}
\cdot  \re\{  H_{N-1}( \sqrt{ \tfrac{Nab}{b-a}  } \zeta )  H_N ( \sqrt{ \tfrac{Nab}{b-a}  } \bar{\zeta} ) \}\\
 &\label{ego}=-e^{-NQ_N(\zeta)}(1+o(1))\tfrac {4\sqrt{2}c^2 e^{-2c^2}}{2^NN!}\cdot \re\{  H_{N-1}( \sqrt{\tfrac N 2} \zeta_1 )  H_N ( \sqrt{\tfrac N 2 } \bar{\zeta_1} ) \},
\end{align}
where we write
$$\zeta_1= p_*+\tfrac {cz} {\sqrt{N/2}}\cdot (1+O(N^{-\frac 1 2}))$$
Here, the second line follows from $\tau^N=e^{-2c^2}+O(N^{-1})$.

Now fix a large number $M$.
Applying the Plancherel-Rotach asymptotic in Lemma \ref{herman}, we obtain
for all $w$ with $|w|\le M$ and all $p_*\in I_\alpha$ that
	\begin{align}
	H_N(\sqrt{\tfrac N 2}\cdot p_*+\tfrac w {\sqrt{N}})
=
e^{\frac N 4 p_*^2} 2^{\frac N 2} \sqrt{N!} \cdot N^{-\frac 1 4}\cdot O(1),\qquad (N\to\infty),
	\end{align}
where the $O(1)$ constant depends on $M$ and $\alpha$. Inserting this in \eqref{ego} we obtain, for $z$ in a given compact subset of $\C$, that
$|\tfrac {\d R_N}{\d x}(z)|= O(N^{-1})$ as $N\to\infty$.
To be more precise, it follows from
$$
e^{-NQ_N(\zeta)+ \frac{N}{2} p_*^2 }=O(1), \qquad O(\tfrac{1}{N!}\sqrt{N!(N-1)!} N^{-1/2} )=O(N^{-1}).
$$
Letting $N_k\to\infty$  we conclude the desired horizontal translation invariance.  $\qed$

\subsection{Bulk scaling limit for ALUE} \label{BALUE} We now prove Theorem \ref{mthWish}, and begin by
recalling that $$Q_N=-\tfrac 1 N\log(K_\nu(\tfrac {N^2|\zeta|}{c^2})|\zeta|^\nu)-(\tfrac N {c^2}-1)\re\zeta.$$

Again we fix a small $\alpha>0$ and consider zooming points $p_*$ in the interval $I_\alpha=[\alpha,4-\alpha]$. It suffices (by Theorem \ref{tithm})
to show that each limiting rescaled 1-point function $R$ about $p_*$ is horizontal translation invariant, i.e., $\tfrac {\d R}{\d x}=0$.

Recall that $R=\lim R_{N_k}$ where
\begin{equation}\label{ivm}R_N(z)=\tfrac 1 {N\Delta Q_N(p_*)}\bfR_N(\zeta),\qquad z={\scriptstyle\sqrt{N\Delta Q_N(p_*)}}\cdot (\zeta-p_*).\end{equation}

By the asymptotic formula in \eqref{BA2}, $z$ and $\zeta$ in \eqref{ivm} are related as
\begin{equation}\label{mappa}\zeta=\zeta(z)=p_*+\tfrac z {\sqrt{N\Delta Q_N(p_*)}}=p_*+\tfrac{2c\sqrt{p_*}}{N} z+o(\tfrac 1 N).\end{equation}

Moreover, when $z=x+iy$ remains in a fixed compact set, while $p_*\in [\alpha,4-\alpha]$ for some small $\alpha>0$, we infer from \eqref{besselA1} that
\begin{equation}\label{ett}
e^{-NQ_N(\zeta)} = \sqrt{\tfrac{\pi}{2}} \tfrac{c}{N} {p_*}^{\,\nu-\frac12}  e^{-Np_*-2c\sqrt{p_*}x-2y^2} \cdot (1+o(1)).
\end{equation}
Next recall the relation \eqref{obt},
\begin{equation}\label{tva}\bfR_N(\zeta)=C_N^2\cdot G_N(\tfrac {a^2-b^2}{2b}N\zeta)\cdot e^{-NQ_N(\zeta)},\quad \scriptstyle (a=\tfrac N {c^2}\,\quad b=\tfrac N {c^2}-1, \quad C_N=\sqrt{aN}(\tfrac {a^2-b^2}{2a} N)^{\frac {\nu+1}2}).\end{equation}

We will invoke an identity found in \cite[Lemma 2]{AB}, which asserts that if $\nu$ is an integer, then
\begin{align}\label{tre}G_N(z)&=\tfrac{\tau^{2N} e^{ \bar{z} }  }{4\pi^2 z^\nu} U_N(z),\qquad \scriptstyle (\tau=1-\tfrac{c^2} N),\\
\label{U_N}
U_N(z)&=\oint_{\gamma_1} du \oint_{\gamma_2}dv \left(  \frac{v(u-1)}{(v-1)u} \right)^\nu \left( \frac{v}{u} \right)^N \frac{e^{z \frac{u}{u-1}  -\bar{z} \frac{v}{v-1} }}{ (\tau^2v-u) (v-1)(u-1) },
\end{align}
where $\gamma_1$ is a simple closed contour encircling $u=0$ but not $u=1$, while $\gamma_2$ encircles both the point $v=1$ and the entire $\gamma_1$ in such a way that $\tau^2 v-u \not= 0$ on $\gamma_2$.

\begin{lem} \label{U_N pax} With $x=\re z$ and $U_N$ as in \eqref{U_N}, we have
\begin{equation*}
\tfrac  {\d U_N}{\d x}(z) = (2\pi i)^2 (-1)^\nu  e^{-\bar{z}} L_{ N+\nu-1 }^{ 1-\nu }(z) L_{N-1}^{\nu+1}(\bar{z})\cdot (1+o(1)),\qquad (N\to \infty).
\end{equation*}
\end{lem}

\begin{proof}
By a straightforward computation,
\begin{align*}
{\scriptstyle \frac \d {\d x} \Big[\frac{e^{z \frac{u}{u-1}  -\bar{z} \tfrac{v}{v-1} }}{ \tau^2v-u }\Big]}
= {\scriptstyle  \frac{e^{z \frac{u}{u-1}  -\bar{z} \tfrac{v}{v-1} }}{(u-1)(v-1)}}\cdot (1+o(1)).
\end{align*}

Differentiating under the integral sign, we obtain that
$$
\tfrac {\d U_N}{\d x}(z) = \oint_{\gamma_1}  \tfrac{(u-1)^{\nu-2} }{u^{N+\nu}}  e^{z \frac{u}{u-1}  }\, du \oint_{\gamma_2} \tfrac{v^{N+\nu}}{(v-1)^{\nu+2} }   e^{ -\bar{z} \frac{v}{v-1} }\, dv\cdot (1+o(1)).
$$

We now recall an integral representation for Laguerre polynomials, found for example in \cite[Section \RN{3}]{AB},
\begin{equation}\label{flax}
L_j^\nu(z)=\tfrac{1}{2\pi i} \oint_{\gamma} \tfrac{  e^{ -z\frac{s}{1-s}  } }{(1-s)^{\nu+1} s^{j+1} }\,ds,
\end{equation}
where the contour $\gamma$ encircles the origin $s = 0$ but not the essential singularity at $s=1$. The identity \eqref{flax} immediately implies
$$
\tfrac{1}{2\pi i} \oint_{\gamma_1}  \tfrac{(u-1)^{\nu-2} }{u^{N+\nu}}  e^{z \frac{u}{u-1}  }\, du =(-1)^\nu L_{ N+\nu-1 }^{ 1-\nu }(z).
$$
On the other hand the change of variable $v=1/w$ gives
\begin{align*}
\tfrac{1}{2\pi i}\oint_{\gamma_2} \tfrac{v^{N+\nu}}{(v-1)^{\nu+2} }   e^{ -\bar{z} \frac{v}{v-1} }\, dv&=e^{-\bar{z}} \tfrac{1}{2\pi i}\oint_{\gamma} \tfrac{ e^{ -\bar{z} \frac{w}{1-w} } }{(1-w)^{\nu+2} w^{N} }  \, dw =e^{-\bar{z}} L_{N-1}^{\nu+1}(\bar{z}).
\end{align*}
Combining all of the above, we obtain the lemma.
\end{proof}

Combining the identities \eqref{mappa} through \eqref{tre} and using \eqref{BA2}, we see that
\begin{align} \label{R AWishart}
\begin{split}
R_N(z)
&=\tfrac {a} {\Delta Q_N(p_*)} e^{-NQ_N(\zeta)}  ( \tfrac{a^2-b^2}{2a}N )^{\nu+1} G_N( \tfrac{a^2-b^2}{2b}N\,\zeta(z) )
\\
&=\tfrac {a} {\Delta Q_N(p_*)} e^{-NQ_N(\zeta)}  ( \tfrac{a^2-b^2}{2a}N )^{\nu+1} \tfrac{\tau^{2N} }{ 4\pi^2 } (\tfrac{a^2-b^2}{2b}N\,\zeta(z))^{-\nu}  e^{ \tfrac{a^2-b^2}{2b}N\,\bar{\zeta}(z) }  U_N( \tfrac{a^2-b^2}{2b}N\,\zeta(z) )
\\
&=\tfrac {aN} {\Delta Q_N(p_*)}  \tfrac{a^2-b^2}{2a}   \tfrac{\tau^{2N+\nu} }{ 4\pi^2 } \zeta(z)^{-\nu}  e^{-NQ_N(\zeta)} e^{ \tfrac{a^2-b^2}{2b}N\,\bar{\zeta}(z) }  U_N( \tfrac{a^2-b^2}{2b}N\,\zeta(z) ) \cdot (1+o(1))
\\
&= \tfrac{c\sqrt{p_*}}{ \sqrt{2} \pi^{3/2} }
e^{-2y^2-2c\sqrt{p_*} y i-2c^2 }  U_N( \tfrac{a^2-b^2}{2b}N\,\zeta(z) ) \cdot (1+o(1)).
\end{split}
\end{align}
For the last equality, we also used $\tau^{2N}=e^{-2c^2} (1+o(1))$ and \eqref{ett}.

Therefore it remains to show that
$\tfrac \d {\d x} U_N( \tfrac{a^2-b^2}{2b}N\,\zeta(x+iy) )\to 0$ as $N\to\infty$, where $\zeta(z)$ is the map \eqref{mappa}.

However, by Lemma \ref{U_N pax} (and since $\tfrac {a^2-b^2}{2b}=1+O(N^{-1})$), we have
\begin{equation}\label{spil}\begin{split}
\tfrac \d {\d x}  U_N&( \tfrac{a^2-b^2}{2b}N\,\zeta(z) ) =C
e^{-\frac{a^2-b^2}{2b}N \bar{\zeta}}  \cr
&\times L_{ N+\nu-1 }^{ 1-\nu }(\tfrac{a^2-b^2}{2b}N\,\zeta )\cdot L_{N-1}^{\nu+1}(\tfrac{a^2-b^2}{2b}N\,\bar{\zeta} )\cdot (1+o(1)),\cr
\end{split}
\end{equation}
where the constant $C=(2\pi)^2(-1)^{\nu+1}2c\sqrt{p_*}$ is independent of $N$. By \eqref{mappa}, we have for $p_*\in[\alpha,4-\alpha]$
$$
\tfrac{a^2-b^2}{2b}N \zeta =Np_*+\tfrac{c^2}{2}p_*+2c\sqrt{p_*}\,z+O(N^{-1}).
$$

When $p_*\in[\alpha,4-\alpha]$ and $z$ remains in a compact subset of $\C$, we now apply the Plancherel-Rotach asymptotic in Lemma \ref{BAAL}
to conclude that as $N\to\infty$
\begin{align*}
L_{ N+\nu-1 }^{ 1-\nu }(\tfrac{a^2-b^2}{2b}N\,\zeta(z) ) L_{N-1}^{\nu+1}(\tfrac{a^2-b^2}{2b}N\,\bar{\zeta}(z) ) = \tfrac{1}{N}\, e^{\frac{a^2-b^2}{2b}N \frac{\zeta(z)+ \bar{\zeta}(z)}{2}  }\cdot O(1),
\end{align*}
which together with \eqref{spil} shows that
\begin{align*}\tfrac \d {\d x} U_N( \tfrac{a^2-b^2}{2b}N\,\zeta(z) )&=C  \tfrac{1}{N}\, e^{\frac{a^2-b^2}{2b}N \frac{\zeta(z)- \bar{\zeta}(z)}{2}  }\cdot O(1)\\
&= C\tfrac{1}{N} e^{2c\sqrt{p_*} i\im z} \cdot O(1)=  O(\tfrac 1 N).
\end{align*}

We have shown that each subsequential limit $R=\lim R_{N_k}$ obeys $\tfrac {\d R}{\d x}=0$. $\qed$

\section{Edge scaling limits} \label{ESL} In this section we study edge scaling limits for the almost-Hermitian GUE and LUE, and we prove Theorem \ref{Thm_Bender} and Theorem \ref{Thm_Osborn}, respectively.

\subsection{Edge scaling limit for AGUE} \label{Subsection_Bender} We now prove Theorem \ref{Thm_Bender}.
We shall now give a somewhat simplified proof (with respect to earlier proofs in \cite{AB,B}) by using the summation formula in \eqref{RN pay}.

Consider the modified ellipse potential
$Q_N=\tfrac 1 2\xi^2+\tfrac 1 2\tfrac {N^{\frac 1 3}}{c^2}\eta^2.$ Recall
that the right end-point $p_N$ of the droplet $S_{Q_N}$ is
\begin{equation}\label{pna}p_N=2(1+\tfrac {c^2}{N^{\frac 1 3}})^{-\frac 1 2}=2-\tfrac {c^2}{N^{\frac 1 3}}+\tfrac 3 4\tfrac {c^4}{N^{\frac 2 3}}+O(\tfrac 1 N).\end{equation}

We rescale about $p_N$ by setting
\begin{equation*}R_N(z)=\tfrac 1 {N\Delta Q_N} \bfR_N(\zeta),\qquad z={\scriptstyle \sqrt{N\Delta Q_N}}\cdot(\zeta-p_N),\end{equation*}
i.e.,
\begin{equation}\label{resc2}\zeta=\zeta(z)=p_N+\tfrac z {\sqrt{N\Delta Q_N}}=p_N+\tfrac {2c}{N^{\frac 23}}z+o(N^{-1}).\end{equation}

Making use of the formula \eqref{Ge1pt} we see that
\begin{equation*}
R_N(z)=
e^{-NQ_N(\zeta)} \tfrac{2\sqrt{ab}}{a+b}\sum_{j=0}^{N-1} \tfrac{1}{j!} (\tfrac \tau 2)^{j} |  H_j( \sqrt{ \tfrac{Nab}{b-a} }\,\zeta ) |^2,
\end{equation*}
where
\begin{equation}\label{parameters edge}
	a=\tfrac 1 2,\quad b=\tfrac 1 2\tfrac {N^{\frac 1 3}}{c^2},\quad \tau=\tfrac {b-a}{b+a}=1-\tfrac {2c^2}{N^{\frac 1 3}}+O(N^{-\frac 2 3}).
\end{equation}

By \eqref{resc2} and \eqref{pna} and an elementary computation,
\begin{equation}\label{zeta edge}
\sqrt{ \tfrac{Nab}{b-a} }\,\zeta=\sqrt{2N}+\tfrac{c^4+2cz}{\sqrt{2} N^{1/6} }  +O(N^{-\frac56}).
\end{equation}
Write $z=x+iy$. Then again by \eqref{resc2} and \eqref{pna}, we have
\begin{align} \label{Q_N edge}
	\begin{split}
		NQ_N(\zeta) &=\tfrac{N}{2}(p_N+\tfrac{x}{\sqrt{N \Delta Q_N}})^2+\tfrac 1 2\tfrac {N^{\frac 1 3}}{c^2} \tfrac{y^2}{\Delta Q_N } +O(N^{-\frac 1 3})
		\\
		&= 2N-2c^2 N^{2/3} +2( c^4 + 2 c \,x) N^{\frac 1 3}-2( c^6  -  y^2+ 2 c^3 x)+O(N^{-\frac 1 3}).
	\end{split}
\end{align}

Inserting \eqref{parameters edge} into the identity \eqref{RN pay}, we have
\begin{equation} \label{RN pax}
	\tfrac {\d R_N}{\d x}(z)= - \tfrac{4\sqrt{2}\,e^{-NQ_N(\zeta)}}{\sqrt{N^{2/3}c^{-4}-1}}  \tfrac{( \tau/2 )^N} {(N-1)!}
	\cdot  \re\{  H_{N-1}( \sqrt{ \tfrac{Nab}{b-a}  } \zeta )  H_N ( \sqrt{ \tfrac{Nab}{b-a}  } \bar{\zeta} ) \}.
\end{equation}
Note that by \eqref{Q_N edge}, we have
\begin{align} \label{RN pax 1}
	\tfrac{e^{-NQ_N(\zeta)} \tau^N}{\sqrt{N^{2/3}c^{-4}-1}} =\frac{c^2} {N^{1/3}} e^{-2N -2 ( c^4 + 2 c \,x) N^{1/3}    } e^{  \frac43 c^6-2y^2+4c^3x   } \cdot (1+o(1)),
\end{align}
where we used that
$\log \tau^N = -2 c^2 N^{\frac23} - \tfrac23 c^6+O(N^{-\frac23}).$

We now invoke the critical Plancherel-Rotach estimate in \cite[Theorem 8.22.9 (c)]{S}:
\begin{align*}
	\begin{split}
		H_N(\sqrt{2N}+\tfrac{z}{\sqrt{2} N^{1/6} }  )&= (2\pi)^{\frac14} 2^{ \frac{N}{2} } \sqrt{N!} N^{-\frac{1}{12}  }
	 e^{ \frac12 (\sqrt{2N}+\frac{z}{\sqrt{2}N^{1/6} })^2 } (\Ai z) \cdot (1+o(1)).
	\end{split}
\end{align*}
Applying \eqref{zeta edge}, we now obtain
\begin{align} \label{RN pax 2}
	\begin{split}
		&\quad \re\{  H_{N-1}( \sqrt{ \tfrac{Nab}{b-a}  } \zeta )  H_N ( \sqrt{ \tfrac{Nab}{b-a}  } \bar{\zeta} ) \}
		\\
		&=\re\{  H_{N-1}(\sqrt{2N}+\tfrac{c^4+2cz}{\sqrt{2}  N^{1/6} } )  H_N ( \sqrt{2N}+\tfrac{c^4+2c\bar{z} }{\sqrt{2}  N^{1/6} } ) \}\cdot (1+o(1))
		\\
		& =   \sqrt{\pi} \, \tfrac{2^{N} N!}{ N^{\frac23} } e^{  2 N + 2( c^4 + 2 c \, x) N^{\frac 13} }  |\Ai(2cz+c^4)|^2 \cdot (1+o(1)).
	\end{split}
\end{align}
For the last equality we used that
$$
\re \{( \sqrt{2N}+\tfrac{2c z+c^4}{\sqrt{2}N^{1/6}} )^2\} = 2 N + 2( c^4 + 2 c \, x) N^{\frac 13}+O(N^{-\frac 16}).
$$
Combining \eqref{RN pax}, \eqref{RN pax 1} and \eqref{RN pax 2}, we obtain that
\begin{equation} \label{RN pax cr}
	\tfrac {\d R_N}{\d x}(z)= -\sqrt{2\pi}\,4c^2 e^{ \frac{4}{3}c^6-2 y^2+4c^3 x } |\textup{Ai}(2cz+c^4)|^2\cdot (1+o(1)),
\end{equation}
where $o(1)$-term holds uniformly on compact subsets of $\C$.

Let $R=\lim R_{N_k}$ be a limiting $1$-point function. We shall prove that for all $y\in\R$,
\begin{equation}\label{limita}\lim\limits_{x\to +\infty} R(x+iy)= 0.\end{equation}
To verify this, we note that the modified potential $Q_N$ satisfies $\Delta Q_N\asymp N^{\frac 1 3}$. The
modified counterpart to Lemma \ref{spike3} is thus that $\bfR_N(\zeta)\le CN^{\frac 4 3}e^{-N(Q_N-\check{Q}_N)(\zeta)}$.
Combining this with the estimate \eqref{beckom} with $M_N\asymp N^{\frac 1 3}$ in the present case, we obtain \eqref{limita}.

Using the asymptotic formula \eqref{Airy asym}, we obtain
\begin{align*}
e^{  4c^3x } |\Ai (2cz+c^4)|^2 &=- \int_0^\infty \tfrac \d {\d u} \{e^{  4c^3(u+x) }  |\Ai ( 2c(z+u)+c^4 )|^2\}   \, du
\\
&=-\tfrac \d {\d x} \int_0^\infty e^{  4c^3(u+x) }  |\Ai ( 2c(z+u)+c^4 )|^2 \, du.
\end{align*}
Therefore, using \eqref{RN pax cr} and \eqref{limita}, an integration with respect to $x$ gives that
$$
R(x+iy)=\sqrt{2\pi} \, 4c^2 \,  e^{ \frac{4}{3}c^6-2y^2 }
\int_0^\infty e^{  4c^3(u+x) }  |\textup{Ai}( 2c(x+iy+u)+c^4 )|^2   \, du.
$$

Our proof of Theorem \ref{Thm_Bender} is complete. $\qed$

\subsection{Edge scaling limit for ALUE} \label{EALUE} We now prove Theorem \ref{Thm_Osborn} and
recall that
$$Q_N(\zeta)=-\tfrac 1 N\log(K_\nu(\tfrac {N^2|\zeta|}{c^2})|\zeta|^\nu)-(\tfrac N {c^2}-1)\re\zeta.$$
Also let $V=\lim Q_N$ be the Marchenko-Pastur potential in \eqref{mppot}.

We now rescale about the origin.
It follows from the formula \eqref{obt} that
the rescaled $1$-point function
$$R_N(z)=(\tfrac c N)^4\bfR_N(\zeta),\qquad z=(\tfrac N c)^2\zeta$$ is given by
\begin{equation} \label{RO c0}
R_N(z)=e^{-NQ_N(\zeta)}\,(\tfrac c N)^2  ( \tfrac{a^2-b^2}{2a}N )^{\nu+1}\cdot \sum_{j=0}^{N-1} \tfrac{\tau^{2j}j! }{(j+\nu)!} |L_j^\nu(\tfrac{a^2-b^2}{2b}N \zeta )|^2 ,
\end{equation}
where we remind that
$$
a=\tfrac 1 {c^2}N, \quad  b=\tfrac 1 {c^2}N-1, \quad \tau=\tfrac ba=1-\tfrac {c^2}N.
$$

Note that since
$
e^{-NQ_N(\zeta)}=K_\nu(|z|)\cdot (\tfrac cN)^{2\nu}\,|z|^\nu e^{ (1-\frac{c^2}{N})\,\re z },
$
we have
\begin{equation}\label{RO c1}
e^{-NQ_N(\zeta)}\,(\tfrac c N)^2  \left( \tfrac{a^2-b^2}{2a}N \right)^{\nu+1}=K_\nu(|z|)\,|z|^\nu\, e^{\, \re z }\cdot (\tfrac {c^2}N)^{\nu+1}  \cdot (1+o(1)).
\end{equation}

For a fixed parameter $t \in (0,1)$, we will now estimate the contribution of terms
$$|L_j^\nu(\tfrac{a^2-b^2}{2b}N \zeta )|^2=|L_j^\nu(\tfrac{c^2}{N} z )|^2 (1+O(N^{-2}))$$
as $j,N\to\infty$ with $\tfrac jN=t$.
For this we note that
$$
\log \tau^{2j}=2j \log ( 1-\tfrac{c^2}{N} ) = -2c^2t \cdot (1+O(N^{-1})), \qquad \tfrac{j!}{(j+\nu)!}=j^{-\nu}(1+O(N^{-1})).
$$
Therefore we obtain that, on replacing $\tfrac jN$ by $t$,
\begin{equation} \label{RO c2}
\sum_{j=0}^{N-1} \tfrac{\tau^{2j}j! }{(j+\nu)!} |L_j^\nu(\tfrac{a^2-b^2}{2b}N \zeta )|^2=   (1+o(1))  \sum_{j=0}^{N-1} \tfrac{e^{-2c^2t}}{j^\nu}   |L_j^\nu( \tfrac{c^2 t}{j} z)|^2.
\end{equation}
Inserting \eqref{RO c1},\eqref{RO c2} into \eqref{RO c0}, we have
\begin{equation}
R_N(z)=  (1+o(1))  K_\nu(|z|)\,|z|^\nu\, e^{ \re z }\cdot \tfrac{c^2}{N} \sum_{j=0}^{N-1} (c^2t)^\nu \tfrac{e^{-2c^2t}}{j^{2\nu}}   |L_j^\nu( \tfrac{c^2 t}{j} z)|^2.
\end{equation}

We now apply the formula
\begin{equation*}
	\lim_{j\to\infty} \tfrac 1 {j^\nu} L_j^\nu(\tfrac z j)
	=z^{-\tfrac \nu2}J_\nu (2\sqrt{z}),
\end{equation*}
found in \cite[Theorem 8.1.3]{S}.
This leads to
\begin{equation}
	R_N(z)=   (1+o(1))  K_\nu(|z|)\, e^{\, \re z }\cdot \tfrac{c^2}{N} \sum_{j=0}^{N-1} e^{-2c^2t}  |J_\nu(2c\sqrt{t} \, z^{1/2})|^2.
\end{equation}
Recognizing the right hand side a Riemann sum, we conclude that
\begin{align*}
	R_N(z)&=   (1+o(1)) \, K_\nu(|z|)\, e^{\, \re z }\cdot c^2 \int_0^1  e^{-2c^2t}  |J_\nu(2c\sqrt{t} \, z^{\frac 12})|^2\,dt
	\\
	&=   (1+o(1)) \, \tfrac12 K_\nu(|z|)\, e^{\, \re z } \int_0^{2c}  s\, e^{-s^2/2}  |J_\nu(s \, z^{\frac 12})|^2\,ds,
\end{align*}
which completes our proof of Theorem \ref{Thm_Osborn}. $\qed$

\begin{rmk*} Consider the $1$-point function $R$ in Theorem \ref{Thm_Osborn}, namely
\begin{equation}
R(z)=\tfrac12 K_\nu(|z|)\, e^{\, \re z } \int_0^{\,2c}  s\, e^{-\frac 1 2 s^2}  |J_\nu(s \, z^{\frac 12})|^2\,ds.
\end{equation}

Arguing as in \cite[Section 3]{AKS}, one can prove that $R$ satisfies a Ward equation of the form
\begin{equation}\label{Ward Bessel}
	\dbar C= R-\Lap Q_0-\Lap \log R,\qquad Q_0(z)=-\log(K_\nu(|z|)\cdot|z|^\nu).
\end{equation}

Following the method in \cite{AKS}, it can be shown that in the limit as $c\to\infty$, i.e., the function
$R_{(c=\infty)}(z)=\tfrac 1 2 K_\nu(|z|)I_\nu(|z|)$
is the unique radially symmetric solution to \eqref{Ward Bessel}, which also satisfies the mass-one condition $\int_\C B(z,w)\,dA(w)=1$. We skip giving a
detailed proof, but we remark that in the special cases when $\nu \in \{ \pm \tfrac{1}{2}\}$, this result is shown in \cite{AKS}, and that the general case can be treated in a similar way.
\end{rmk*}

\section{Chiral ensembles} \label{QCD}

In this section, we compare our results for the ALUE to related ensembles which arise for example in connection with quantum chromodynamics  with a chemical potential (QCD).

\subsection{Ensembles with $d$-interaction}
Let us fix an integer $d\ge 1$ which we call the interaction  parameter.
We associate to a configuration $\{\zeta_j\}_1^N$ the \textit{$d$-energy} with respect to an external potential $Q_N$ as
\begin{equation}\label{d-Ham}\Ham_{N,d}=\sum_{j\ne k}\log \tfrac 1 {|\zeta_j^d-\zeta_k^d|}+N\sum_{j=1}^NQ_N(\zeta_j).\end{equation}

Given an inverse temperature $\beta>0$, we define a corresponding Boltzmann-Gibbs measure by
\begin{equation}\label{d-BG}d\bfP_{N,d}^\beta=\tfrac 1 {Z_{N,d}^\beta}\, e^{-\beta H_{N,d}}\, dA_N.\end{equation}

Let $\{\zeta_j\}_1^N$ denote a random sample with respect to \eqref{d-BG} where $\beta=1$. It is easy to see that
the process is determinantal, and that a correlation kernel $\bfK_N$ may be taken as the reproducing kernel for the subspace $\calW_{N,d}$ of $L^2=L^2(dA)$
spanned by the weighted monomials
$$\zeta^{jd}\cdot e^{-\frac 1 2 NQ_N(\zeta)},\qquad j=0,\ldots,N-1.$$
Below we write
$\bfR_N(\zeta)=\bfK_N(\zeta,\zeta)$
for the 1-point function of $\{\zeta_j\}_1^N$.

Note that due to the form of the logarithmic interaction in \eqref{d-Ham} a particle $\zeta_j$ will simultaneously repel $d$ symmetrically distributed positions
$\zeta_j \cdot e^{\frac {2\pi ik}d}$ for $k=0,1,\ldots,d-1.$

\begin{figure}[h!]
	\begin{subfigure}[h]{0.32\textwidth}
		\begin{center}
			\includegraphics[width=1.69in,height=1.36in]{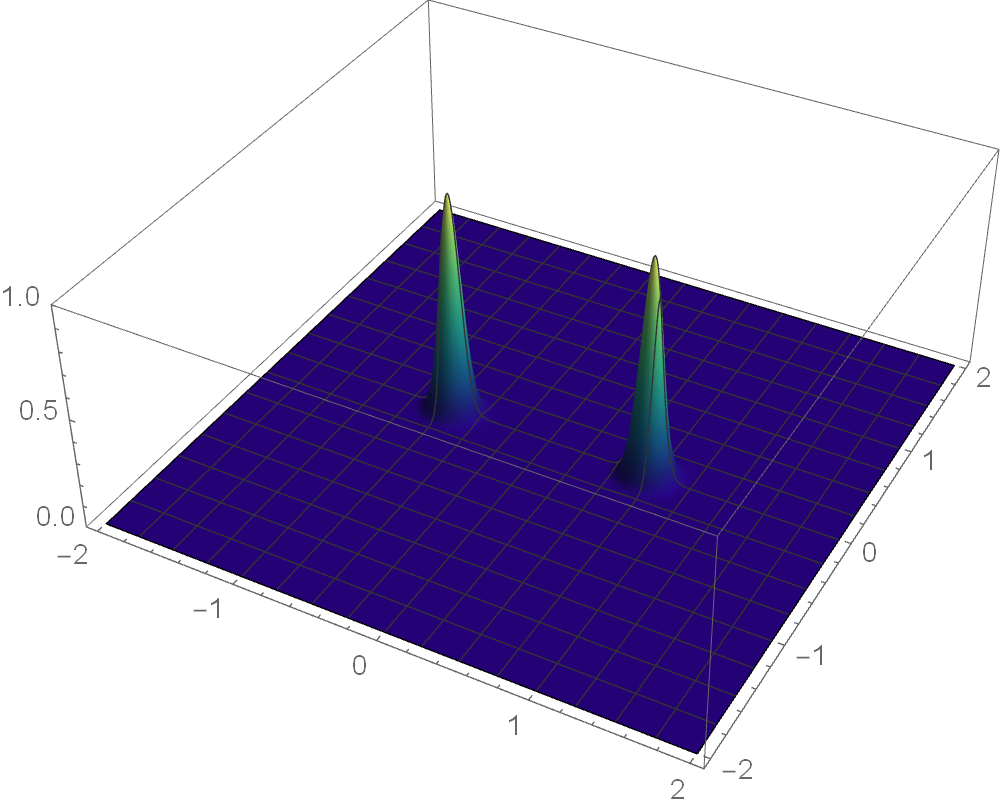}	
		\end{center}
		\caption{$\zeta=1/\sqrt{2} $}
	\end{subfigure}
	\begin{subfigure}[h]{0.32\textwidth}
		\begin{center}
			\includegraphics[width=1.69in,height=1.36in]{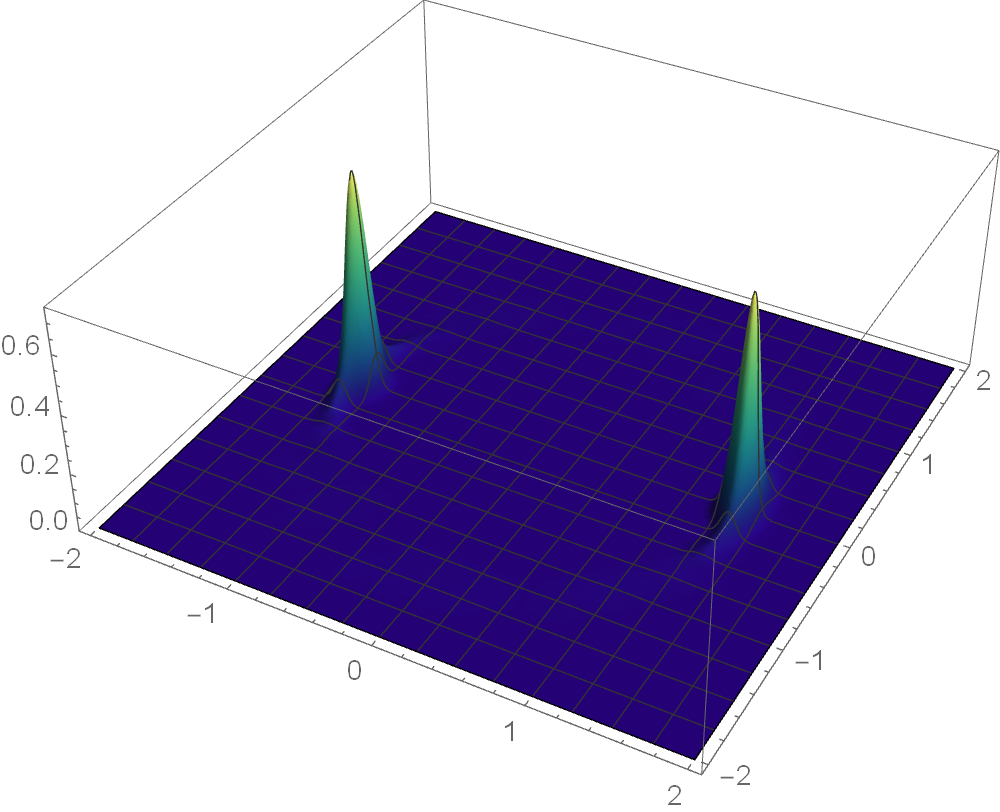}
		\end{center}
		\caption{$\zeta=\sqrt{2}$}
	\end{subfigure}	
	\begin{subfigure}[h]{0.32\textwidth}
		\begin{center}
			\includegraphics[width=1.69in,height=1.36in]{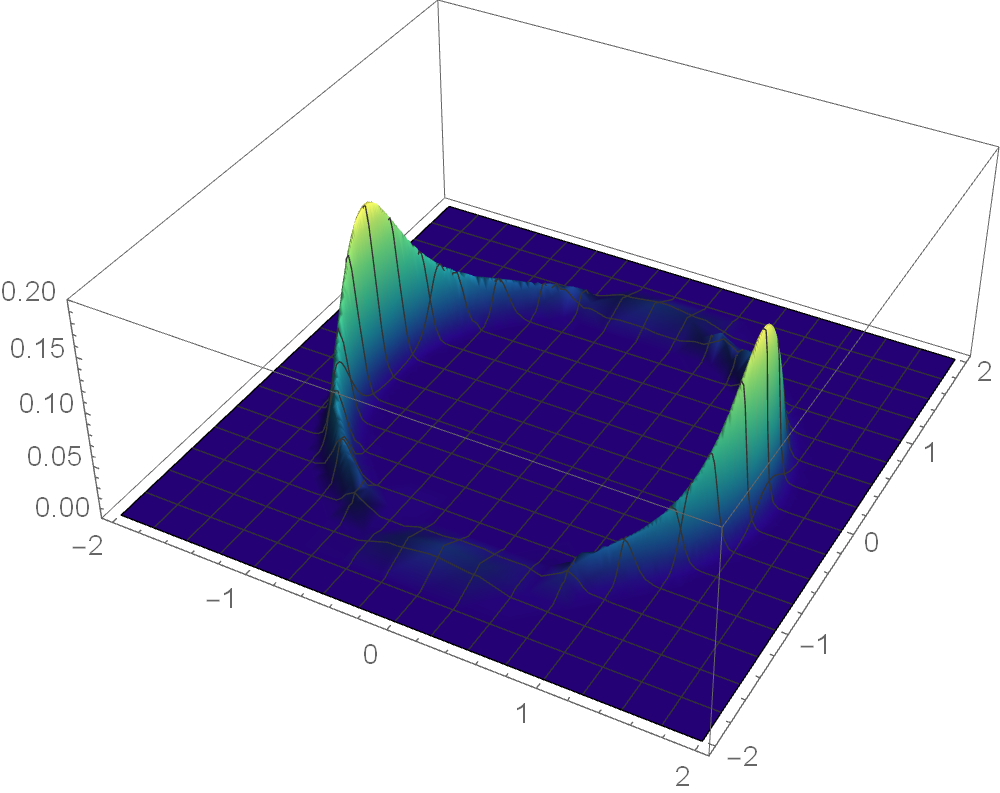}
		\end{center}
		\caption{$\zeta=2$}
	\end{subfigure}	
	\caption{The Berezin density $
\eta\mapsto \bfB_N(\zeta,\eta)$ where $d=2$, $Q(\zeta)=|\zeta|^2$,
$N=100.$ }
	\label{Fig_Berezin}
\end{figure}

\begin{eg*} (``$d$-Ginibre ensemble'') Suppose that $Q_N(\zeta)=|\zeta|^2$ for all $N$. It is easy to see that the process $\{\zeta_j\}_1^N$ is determinantal with a correlation kernel of the form
$$\bfK_N(\zeta,\eta)=N\sum_{j=0}^{N-1}\tfrac {(N\zeta\bar{\eta})^{dj}}{(dj)!} e^{-\frac 1 2 N|\zeta|^2-\frac 1 2 N|\eta|^2}.$$

To illustrate the $\mathbb{Z}_d$-symmetry of this ensemble, we consider the Berezin kernel
$\bfB_N(\zeta,\eta)=\tfrac {|\bfK_N(\zeta,\eta)|^2}{\bfK_N(\zeta,\zeta)}.$

This kernel measures the repelling effect of insertion of a point charge at $\zeta$, more precisely we have
$\bfB_N(\zeta,\eta)=\bfR_{N}(\eta)-\bfR_{N-1}^{(\zeta)}(\eta),$
where $\bfR_{N-1}^{(\zeta)}$ is the $1$-point function for the $(N-1)$-point process which is ``$\{\zeta_j\}_1^N$ conditioned on the event $\zeta\in \{\zeta_j\}_1^N$'', see \cite[Subsection 7.6]{AHM}.

Figure \ref{Fig_Berezin} illustrates the effect of insertion of a charge at various points $\zeta$ in the chiral case $d=2$.

\end{eg*}

\subsection{Chiral almost-Hermitian LUE-type ensembles}
Now fix an integer $d\ge 1$ and consider the $d$-interacting ensemble $\{\zeta_j\}_1^N$ associated with the potential
\begin{equation} \label{Q NWishart d}
	Q_{N}(\zeta)=Q_{N,d}(\zeta)=\tfrac{1}{N} \log \tfrac{1}{  K_\nu( N^2c^{-2}|\zeta|^d) |\zeta|^{d (\nu+2)-2 }   }-(\tfrac N {c^{2}}-1)\re \zeta^d.
\end{equation}

I.e., we let $\{\zeta_j\}_1^N$ be random sample from the measure \eqref{d-BG}. Note that when $d=1$, we recover the eigenvalue statistics of the almost-Hermitian LUE discussed earlier. For $d=2$ we recover the eigenvalue statistics of Dirac matrices, which is relevant for QCD theory. See \cite{O} as well as \cite[Section 15.11]{F} with references.

In general,
using the evident relation between $Q_{N,d}(\zeta)$ and $Q_{N,1}(\zeta^d)$,
it is not hard to see that the system $\{\zeta_j\}_1^N$ will tend to occupy the $d$-droplet $S_d$ where $S_d=\{\zeta\, ;\, \zeta^d\in S\}$, $S$ being the droplet associated with $d=1$, given in \eqref{wishdrop}. Figure \ref{Fig_d123} shows such droplets for $d\in\{1,2,3\}$.

\begin{figure}[h!]
	\begin{subfigure}[h]{0.32\textwidth}
		\begin{center}
			\includegraphics[width=1.69in,height=1.12in]{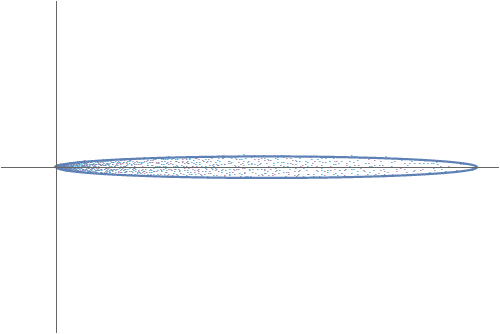}	
		\end{center}
		\caption{$S_1$}
	\end{subfigure}
	\begin{subfigure}[h]{0.32\textwidth}
		\begin{center}
			\includegraphics[width=1.69in,height=1.12in]{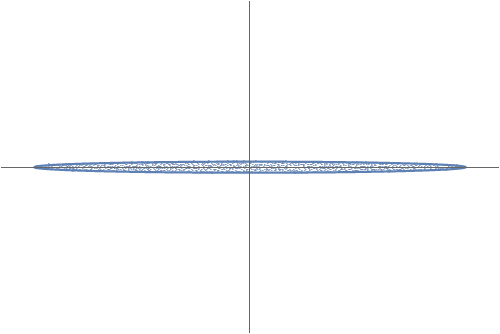}	
		\end{center}
		\caption{$S_2$}
	\end{subfigure}
	\begin{subfigure}[h]{0.32\textwidth}
		\begin{center}
			\includegraphics[width=1.69in,height=1.12in]{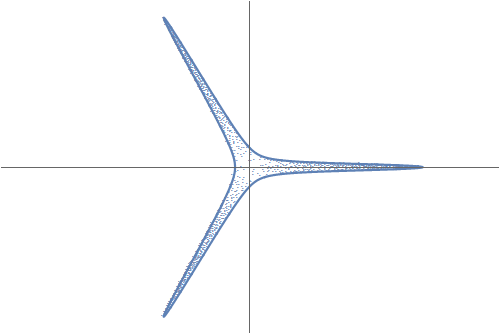}	
		\end{center}
		\caption{$S_3$}
	\end{subfigure}
	\caption{$d$-droplets for a few values of $d$.} \label{Fig_d123}
\end{figure}

We rescale the process $\{\zeta_j\}_1^N$ about  the origin via
$$z_j=(\tfrac N c)^{\frac 2 d}\zeta_j,\qquad j=1,\ldots,N$$
and denote by $R_N(\zeta)=(\tfrac c N)^{\frac 4 d}\bfR_N(\zeta)$ the $1$-point function of the system $\{z_j\}_1^N$, where $z=(\tfrac N c)^{\frac 2 d}\zeta$.

(The rescaling order $N^{\frac 2d}$ is chosen according to the mean sample spacing, which follows the density $d^{-1}\Delta Q_N$ near the origin.)

With a slight modification of our proof of Theorem~\ref{Thm_Osborn} (i.e., by a similar argument as used in \cite{O}) one can obtain the following result.

\begin{thm}
	
	We have $R_N \to R$ uniformly on compact subsets of $\C$, where
\begin{equation} \label{RO d}
	R(z)=\tfrac{d}{2} \, |z|^{2d-2}\, K_\nu( |z|^d)   e^{\,\re z^d}\int_{0}^{2c} s\, e^{\,-s^2/2}  |J_\nu(s \, z^{d/2})|^2 \, ds.
\end{equation}
\end{thm}

 We now show that for $d=2$ and $\nu\in\{\pm\tfrac 1 2\}$, the process in \eqref{RO d} interpolates between known point fields,
namely between (1-dimensional) anti-symmetric sine-processes and (2-dimensional) Mittag-Leffler fields.

To this end, we first note that passing to the limit as $c \to \infty$ in \eqref{RO d}, we obtain the limiting 1-point function
\begin{equation} \label{sjoman}
	R_{(c=\infty)} (z)=\tfrac{d}{2}\,|z|^{2d-2} \,K_\nu( |z|^d)  I_\nu( |z|^d).
\end{equation}	

To obtain a 1-dimensional perspective, we rescale once more. Given the limiting point-field $\{z_j\}_1^\infty$, we pass
to the system $\{\tilde{z}_j\}_{j=1}^\infty$ where $\tilde{z}_j=c^2 z_j$. The corresponding 1-point functions are related via
 $\tilde{R}(\tilde{z})=c^{-4} R(c^{-2} \tilde{z})$.

In the limit as $c\to 0$, the functions $\tilde{R}_{(c)}$ converges to a limit $\tilde{R}_{(c=0)}$ which vanishes on $\C \setminus ( \cup_{k=0}^{d-1}  \,e^{\frac {2\pi i k}d} \cdot \R_+)$ and is for any integer $k$ given on the ray $e^{\frac {2\pi i k}d} \cdot \R_+$ by
\begin{align}\label{vilsjo}
	\begin{split}
		\tilde{R}_{(c=0)}(e^{\frac {2\pi ik}d}\cdot x)
		&=\tfrac{\pi d}{4}\,x^{d-1} (J_\nu(x^{\frac d2})^2 -J_{\nu+1}(x^{\frac d2})J_{\nu-1}(x^{\frac d2}) )
		,\quad (x>0).
	\end{split}
\end{align}

In the special cases when when $d=2$ and $\nu \in \{ \pm 1/2\}$, the 1-point function in \eqref{sjoman} becomes
\begin{equation}\label{vilgot}
	R_{(c=\infty)} (z)=
	\begin{cases}
		e^{-|z|^2}\sinh(|z|^2) &\text{if } \nu=+\tfrac 1 2,
		\\
	  	e^{-|z|^2}  \cosh(|z|^2) &\text{if } \nu=-\tfrac 1 2.
	\end{cases}
\end{equation}

The $1$-point functions in \eqref{vilgot} correspond precisely to the $1$-point functions of the degenerate elliptic determinantal point processes with root system C or D found by Katori in \cite[Theorem 3.5]{K}, and are special cases of Mittag-Leffler fields in \cite{AKS}.

On the other hand, for $d=2$ and $\nu\in\{\pm \tfrac 1 2\}$, the expression \eqref{vilsjo} reduces to
\begin{equation}
\tilde{R}_{(c=0)}(x)=1 \mp \tfrac{\sin(2x)}{2x}, \quad \text{if } \nu=\pm \tfrac 1 2,\qquad (x\in\R)
\end{equation}
which corresponds to the anti-symmetric sine point processes in \cite[Section 13.1]{M}.

\section{Further results and concluding remarks} \label{crem}

In this section, we give further results and provide some concluding remarks. We shall consider variants of our main model ensembles
(generalized ALUE, induced AGUE, hard edge ensembles) and compare with other related works.

\subsection{Almost-Hermitian LUE with rectangular parameter} \label{agen}
For a non-negative integer $\nu$ (and $\beta=1$), a random sample $\{\zeta_j\}_1^N$ picked with respect to the ALUE-potential $Q_N$ in \eqref{wish} can be realized as the eigenvalues of the product matrix $X=X_1X_2^*$
where
\begin{equation}\label{ppp}
X_1=\sqrt{1+\tau}\, P+\sqrt{1-\tau} \, Q, \quad X_2=\sqrt{1+\tau} \, P-\sqrt{1-\tau} \, Q, \quad \tau=1-\tfrac{c^2}{N}.
\end{equation}
Here $P$ and $Q$ are rectangular matrices of size $N \times (N + \nu)$, with independent, centered, complex Gaussian entries of variance $\tfrac 1{4N}$. (See e.g., \cite{ABK,KS})

We now fix $\alpha\ge 0$ and consider the case when $\nu=\nu_N$ (not necessarily an integer) varies with $N$ as
$$\nu_N=\alpha N\cdot (1+o(1))$$ where $o(1)\to 0$ as $N\to\infty$.
Thus we consider the potential
\begin{equation}\label{nun}
	Q_N(\zeta)=\tfrac 1 N\log\left[K_{\nu_N}(\tfrac {N^2|\zeta|}{c^2})\cdot|\zeta|^{\nu_N}\right]-(\tfrac N {c^2}-1)\cdot \re\zeta.\end{equation}

 We shall use the code-notation ALUE$(\alpha)$ to denote a corresponding random sample $\{\zeta_j\}_1^N$ (and various limits as $N\to\infty$).

By \cite[Theorem 1]{ABK}, the droplet $S_{Q_N}$ is approximately given by the equation
$$(\xi-\alpha-2)^2+(N\eta/c^2)^2 \le 4(\alpha+1). $$
Note in particular that for $\alpha>0$, the origin is outside of the droplet for large $N$ since $(\alpha+2)^2 > 4(\alpha+1).$ This makes the analysis somewhat easier, with more tools being available.

Also note that $Q_N$ converges pointwise to the limit
\begin{equation}\label{apo}
V(\xi)=\xi-\alpha \log \xi,\qquad \xi>0,\end{equation}
and $V=+\infty$ off of $[0,\infty)$. Cf.~\cite[Subsection 3.1]{ABK}.

It is well known (again \cite{ABK} and references) that the equilibrium density in potential \eqref{apo} is the Marchenko-Pastur law $\MPa$ given by
\begin{equation}
\MPa(\xi)=\tfrac{1}{2\pi\,\xi} \sqrt{(\lambda_+-\xi)(\xi-\lambda_-)}\cdot \1_{[\lambda_-,\lambda_+]}(\xi), \qquad \lambda_\pm=(\sqrt{\alpha+1}\pm 1)^2.
\end{equation}

Now write $c_N(\xi)=\tfrac 1 N\int_{\R}\bfR_N(\xi+i\eta)\,d\eta$ for the usual cross-section. We have the following result, which generalizes Theorem~\ref{mthMP}.

\begin{thm} \label{ThMP2} (``\emph{Cross-section convergence for ALUE$(\alpha)$}'') We have the convergence $\tfrac 1 \pi c_N \to\MPa$ in the weak sense of measures, as $N\to\infty$.
\end{thm}

\begin{proof}[Remark on the proof]
Our proof in the case $\alpha=0$ in Subsection \ref{WEAK} generalizes without difficulty to the case when $\alpha > 0$, and in fact we could have treated the general case $\alpha\ge 0$ at once. Indeed, the $j:$th orthonormal polynomial $q_j$ in weight $e^{-NQ_N/2}$ can be taken as
\begin{equation}\label{labb}
q_j(\zeta)=\tfrac{N}{c}   ( N- \tfrac{c^2}{2} )^{\frac {\nu_N+1}2}  (1-\tfrac {c^2}N)^{j}\sqrt{  \tfrac{j!}{(j+\nu_N)!}  }L_j^{\nu_N}(  \tfrac{N(2N-c^2)}{2(N-c^2)}\, \zeta ).
\end{equation}
With minor modifications, these polynomials can be analyzed using similar techniques as in the case $\alpha=0$. Details are omitted.
\end{proof}

Now fix a bulk point $p_*$ (i.e., $\lambda_-<p_*<\lambda_+$). We rescale the process $\{\zeta_j\}_1^N$ about $p_*$ as usual (see \eqref{rescaled}) where $Q_N$ is given by \eqref{nun}, and
denote by $R_N$ the 1-point function of the rescaled system $\{z_j\}_1^N$.
It is easy to see that
the limit \eqref{rhop} is given by $\rho(p_*)=2c\sqrt{p_*}$ for all $\alpha$
(cf.~\cite[Subsection 3.1]{ABK} for details).

We have the following generalization of Theorem~\ref{mthWish}.

\begin{thm} (``\emph{Bulk scaling limit for ALUE$(\alpha)$}'') For any $\alpha \ge 0$, if $\lambda_-<p_*<\lambda_+$ and if $\nu>-1$ is an integer then $R_N\to R$ locally uniformly where $R(z)=F(2\im z)$ and $F(z)=\gamma*\1_{(-2a,2a)}(z)$, $a=a(p_*)=\frac{\pi}{2}\cdot \rho(p_*)\cdot \MPa(p_*)$.
\end{thm}

\begin{proof}[Remark on the proof]
 Our proof in the case $\alpha=0$ in Subsection \ref{BALUE} generalizes in a straightforward way to $\alpha>0$, using again the form of orthogonal polynomials \eqref{labb}. We omit details.
\end{proof}

\subsection{Induced Almost-Hermitian GUE}
 We now briefly consider the \textit{induced AGUE}. This is a natural generalization of AGUE introduced in the paper \cite{ADDV}, obtained by inserting a point charge (of strength $2\nu$) at the origin.

For the definition, we fix $c>0$ and put
$$q_N(\zeta)=\tfrac12 \xi^2+\tfrac12 \tfrac{N}{c^2}\eta^2.$$
Ne next fix
a real number $\nu>-1/2$ and consider the potential
\begin{equation}\label{cone}
	Q_N(\zeta)=q_N(\zeta)+\tfrac{2\nu}{N}\log\tfrac 1 {|\zeta|}.
\end{equation}

Let $\{\zeta_j\}_1^N$ be a random sample with respect to the corresponding Boltzmann-Gibbs measure \eqref{prob} (with $\beta=1$) and rescale about the origin via $z_j={\scriptstyle \sqrt{N\Delta q_N}}\cdot \zeta_j$.
Writing $R_N^{(\nu,c)}(z)$ for the 1-point function of $\{z_j\}_1^N$, it is natural to try to characterize limiting 1-point functions $R^{(\nu,c)}=\lim_{k\to\infty}R_{N_k}^{(\nu,c)}$. It was observed
in \cite{ADDV} that when $\nu$ is an integer, $R^{(\nu,c)}$ as well as a correlation kernel $K^{(\nu,c)}$,
can be obtained from the known case $\nu=0$ by an inductive procedure. For example, since $K^{(\nu=0)}$ is nothing but the kernel $K$ from \eqref{ulg2}, we obtain after a brief computation
\begin{equation}\label{bm}
	\begin{split}
		R^{(1,c)}(z)&=R^{(0,c)}(z)-B^{(0,c)}(0,z)
		\\
		&=R^{(0,c)}(z)
		-\tfrac{1}{4\erf(\sqrt{2}c)}e^{-|z|^2}  | \erf( \tfrac{2c+iz}{\sqrt{2}} )+ \erf(\tfrac{2c-iz}{\sqrt{2}} ) |^2,
	\end{split}
\end{equation}
where we used the interpretation of the Berezin kernel $B^{(0,c)}(0,z)=\tfrac{ |K^{(0,c)}(0,z)|^2 }{R^{(0,c)}(0)}$ as the
difference $R^{(0,c)}(z)-R^{(1,c)}(z)$, see e.g. \cite[Lemma 7.6.2]{AHM}. A similar reasoning is used in \cite{ADDV} to go from $\nu=1$ to $\nu=2$, and this procedure
can in principle be repeated to obtain correlation kernels for all positive integers $\nu$. (We are not aware a result for general real $\nu$ at this time.)

We now explain how the above kernels $K^{(\nu,c)}$ (for integers $\nu$) interpolate between known insertion ensembles in dimension one in \cite{KV,Dan}, and in dimension
two from the papers \cite{FBKS,AKS}. It will suffice to treat the case $\nu=1$ in detail.

More precisely, let $V(\xi)$ be the Gaussian potential, $V(\xi)=\tfrac 1 2\xi^2$ for $\xi\in\R$ and $+\infty$ otherwise, and consider the corresponding insertion potential
$V_N(\zeta)=V(\zeta)+\tfrac{2\nu}{N}\log\tfrac 1 {|\zeta|}$.

It is shown in \cite[Theorem 1.1]{KV} that in the appropriate scaling limit, a microscopic correlation kernel at the origin takes the following form, for $x,y\in\R$ with $xy>0$
\begin{equation}\label{K gen sine}
K_\R^{(\nu)}(x,y)=\tfrac{ \pi }{2}  \tfrac{ \sqrt{x y} } { (x-y) } (  J_{ \nu+\frac12 } (\pi x) J_{\nu-\frac12} (\pi y) -  J_{ \nu+\frac12} (\pi y) J_{ \nu-\frac12 } (\pi x)   ).
\end{equation}

On the other hand, taking $Q_N(\zeta)=|\zeta|^2+\tfrac{2\nu}{N}\log\tfrac 1 {|\zeta|}$ and rescaling in a natural way about the origin, one obtains
a point field with correlation kernel
\begin{equation} \label{K ML1c}
K^{(\nu)}_{\C}(z,w)=G(z,w) P(\nu,z\bar{w}),\qquad P(a,z)=\tfrac{\gamma(a,z)}{\Gamma(a)},
\end{equation}
where $G$ is the Ginibre kernel \eqref{gin0} and $\gamma(a,z)$ is the lower incomplete Gamma function. This kernel appears in \cite{FBKS} and in \cite{AKS}, for example.

\begin{figure}[h!]
	\begin{subfigure}[h]{0.32\textwidth}
		\begin{center}
			\includegraphics[width=1.69in,height=1.36in]{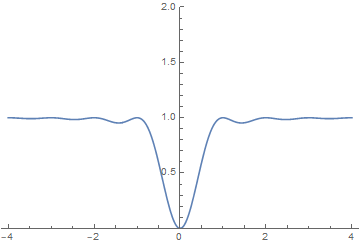}	
		\end{center}
		\caption{$c=0$, $\nu=1$}
	\end{subfigure}
	\begin{subfigure}[h]{0.32\textwidth}
		\begin{center}
			\includegraphics[width=1.69in,height=1.36in]{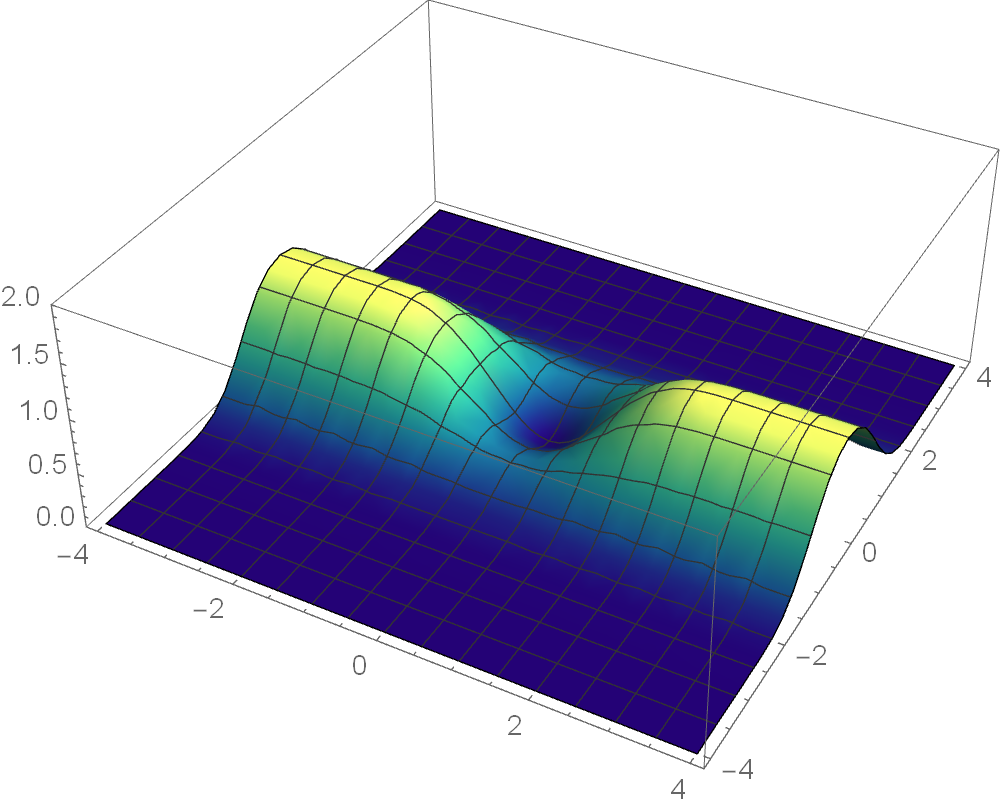}
		\end{center}
		\caption{$c=1$, $\nu=1$}
	\end{subfigure}	
	\begin{subfigure}[h]{0.32\textwidth}
		\begin{center}
			\includegraphics[width=1.69in,height=1.36in]{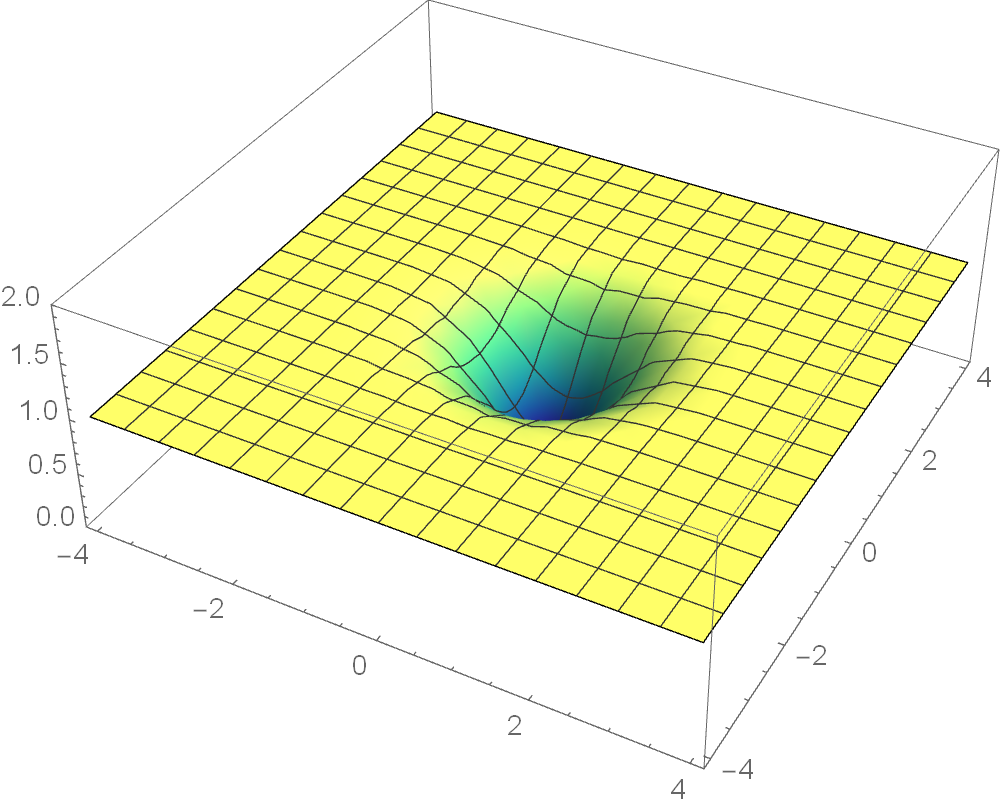}
		\end{center}
		\caption{$c=\infty$, $\nu=1$}
	\end{subfigure}	
	\caption{Graphs of $\tilde{R}^{(1)}_\R(x)$ and of $R^{(1,1)}(z)$, $R^{(1)}_\C(z)$. }
	\label{Fig_ADDP}
\end{figure}

Sending $c \to \infty$ in \eqref{bm} we can recover \eqref{K ML1c}, i.e.,
$$
\lim_{c\to\infty} R^{(1,c)}(z)=1-e^{-|z|^2}=R^{(1)}_\C(z).
$$
Here, we have used $R^{(0,c)}(z) \to 1$ as $c \to \infty$ and $\erf(x) \to 1$ as $x \to +\infty$.

To pass to the corresponding one-dimensional limit, after an appropriate rescaling, we are led to consider the 1-point function
$$\tilde{R}^{(\nu,c)}(z)=\tfrac{1}{\alpha^2}R^{(\nu,c)}(\tfrac{z}{\alpha}),\qquad \alpha=\tfrac{2}{\pi}c.$$
Write $\tilde{K}^{(0,c)}$ for a corresponding rescaled correlation kernel, $\tilde{R}^{(\nu,c)}(z)=\tilde{K}^{(\nu,c)}(z,z)$.

 By the discussion after the formulation of Theorem \ref{mainth1}, we can arrange that $\lim_{c\to0}\tilde{K}^{(0,c)}= \pi K^{\mathrm{sin}}$. Hence $\tilde{R}^{(1,c)}(z)$ converges to $\tilde{R}_\R^{(1)}(x)=\pi K_\R^{(1)}(x,x)$, as $c\to 0$ if $z=x$ is real (and $\tilde{R}^{(1,c)}(z)\to 0$ otherwise)
where
$$
K_\R^{(1)}(x,y)=\tfrac{1}{\pi}\tfrac{ \sin (\pi x-\pi y)}{x-y}-\tfrac{1}{\pi^2}\tfrac{\sin (\pi x) \sin (\pi y)}{xy}
$$
which is the same as \eqref{K gen sine} when $\nu=1$.

We have shown that the 1-point function $R^{(\nu,c)}$ interpolates between $R^{(\nu)}_\R$ and $R^{(\nu)}_\C$ when $\nu=1$, see Figure \ref{Fig_ADDP}. As we indicated,
a similar proof works in the case when $\nu$ is an arbitrary integer.

\subsection{Almost-Hermitian hard-edge ensembles}
Let $(Q_N)_1^\infty$ be a suitable sequence of potentials, for instance the AGUE-potentials in \eqref{wish}, or a more general admissible sequence in the sense of Section \ref{BCG}. In particular, the potential $Q_N$ is real-analytic in a neighbourhood of the boundary $\d S_{Q_N}$.

In the hard edge setting  we completely confine the gas to the droplet by redefining the potential $Q_N$ by setting
$Q_N(\zeta)=+\infty$ when $\zeta\not\in S_{Q_N}.$
Similarly, we redefine the limiting potential $V=\lim Q_N$ by setting $V=+\infty$ outside of $S_V$. This kind of hard edge ensembles are studied in the papers \cite{AKM,AKMW,AKS1,CK,Seo}.

Fix a point $p_*$ in the bulk, i.e., a point
 $p_* \in \R$ satisfying $\sigma_V(p_*)>0$. We also assume that the limit $\rho(p_*)$ in \eqref{rhop} exists and is strictly positive. We recall that the asymptotic height of the rescaled droplet is then given by
 the expression $a(p_*)$ in \eqref{ap}, i.e.,
 $$a(p_*)=\tfrac \pi 2\cdot \rho(p_*)\cdot \sigma_V(p_*).$$

Under these hypothesis we look at the cross-sections
$$c_N(p_*)=\tfrac 1 N\int_\R\bfR_N(p_*+i\eta)\, d\eta,$$
where $\bfR_N$ is the $1$-point function for the hard edge ensemble $\{\zeta_j\}_1^N$ associated with the redefined potential. (So in particular $\bfR_N=0$ on $\C\setminus S_{Q_N}$.)

As in the case of ``free-boundary ensembles'', we expect that the cross-sections should enjoy the convergence $\tfrac 1 \pi c_N(p_*)\to\sigma_V(p_*)$, but our above proofs in the cases
of AGUE/ALUE, which depend on computations with particular orthogonal polynomials, do not immediately carry over to a hard edge setting. In order not to complicate matters, let us
\textit{assume} for the sake of argument that this limit holds, i.e., that
\begin{equation}\label{assa}\lim_{N\to\infty} c_N(p_*)=\pi\cdot \sigma_V(p_*).\end{equation}

Given these assumptions, we now rescale about $p_*$ precisely as in the free-boundary case (see \eqref{blowup}) and we write $R_N$ for the $1$-point function of the rescaled system.
The structure theorem for limiting kernels (analogue of Lemma \ref{LPF}) then takes the following form.

\begin{lem} (``\emph{Structure of limiting hard edge kernels}'') Under the above hypotheses, there exists a sequence $c_N(z,w)$ of cocycles such that each subsequence of the kernels
$c_NK_N$ has a further subsequence which converges locally uniformly on the set
$$\Sigma(p_*)=\{(z,w)\in \C^2\, ;\, |\im z|<a(p_*),\,|\im w|<a(p_*)\}$$ to a Hermitian kernel of the form
$K(z,w)=G(z,w)\cdot L(z,w)$ where $G$ is the Ginibre kernel and $L$ is a Hermitian-analytic function on $\Sigma(p_*)$.
\end{lem}

\begin{proof}[Remark on the proof] This follows by a simple adaptation of the normal-families argument used in the proof of Lemma \ref{LPF}. See \cite[Subection 1.9]{AKMW} for related comments.
\end{proof}

By the lemma, we can form limiting $1$-point functions $R(z)=\lim R_{N_k}(z)$ which are smooth in the strip
$|\im z|<a(p_*)$ and satisfy $R(z)=0$ when $|\im z|>a(p_*)$. Since
$R$ is not identically zero by the assumption \eqref{assa}, we can assert that $R>0$ throughout the strip, and that Ward's equation
\begin{equation}\label{whard}\dbar C(z)=R(z)-1-\Delta \log R(z),\qquad |\im z|<a(p_*)\end{equation}
holds if we understand $C(z)$ as the function
$$B(z,w)=\tfrac {|K(z,w)|^2}{R(z)}\1_{\Sigma(p_*)}(z,w),\qquad C(z)=\int\tfrac {B(z,w)}{z-w}\, dA(w).$$

Note that \eqref{whard} gives a hard edge analogue of Corollary \ref{ecor}. The proof in the hard edge case works basically the same way as in the free boundary case,
again see \cite[Subection 1.9]{AKMW}.

Now recall the function $F=\gamma*\1_{(-2a,2a)}$ appearing in the free boundary case (Theorem \ref{mainth1}), i.e.,
$$
F(z)=\gamma * \1_{(-2a,2a)}(z)=\tfrac 1 {\sqrt{2\pi}}\int_{-2a}^{\, 2a}e^{-\frac 1 2 (z-t)^2}\, dt.
$$

We can now state the following theorem.

\begin{thm}\label{vom} Under the above assumptions (in particular we assume that the cross-section convergence \eqref{assa} holds), each translation invariant
limiting 1-point function $R$ at $p_*$ is of the form
	\begin{align}
		\label{rett_h}
		R(z)&=\1_{ \{ |\im z|<a \} } \cdot \tfrac{1  }{ \sqrt{2\pi} } \int_{-2a}^{2a} \tfrac{ e^{ -\frac12 ( 2\im z-t )^2 } }{  F(t) }\,dt,
	\end{align}
where $a=a(p_*)$ is given in \eqref{ap}.
\end{thm}

\begin{proof} We appeal to the characterization of translation invariant solutions to the hard edge Ward equation \eqref{whard} found in \cite[Theorem 6]{AKMW}, which gives that each translation
invariant $R$ must be given by the formula \eqref{rett_h} for some $a>0$. Then the rescaled version of the cross-section convergence \eqref{assa} fixes the value of $a$ as $a(p_*)$.
\end{proof}

The assumptions (cross-section convergence and translation invariance) may in general be subtle to check, even in the simplest model cases such as for hard edge AGUE. We shall not dwell on this
matter here (it will be the topic of a forthcoming study), but we remark that Theorem \ref{vom} yields further support for the conjecture in \cite{AKMW} that the parameter-value $a=a(p_*)$ should
yield the correct ``physical'' solution to Ward's equation in a strip.

See Figure~\ref{Fig_RH} for the $1$-point density $R$ in \eqref{rett_h} for a few values of $c$.

\begin{figure}[h!]
	\begin{subfigure}[h]{0.32\textwidth}
		\begin{center}
			\includegraphics[width=1.69in,height=1.36in]{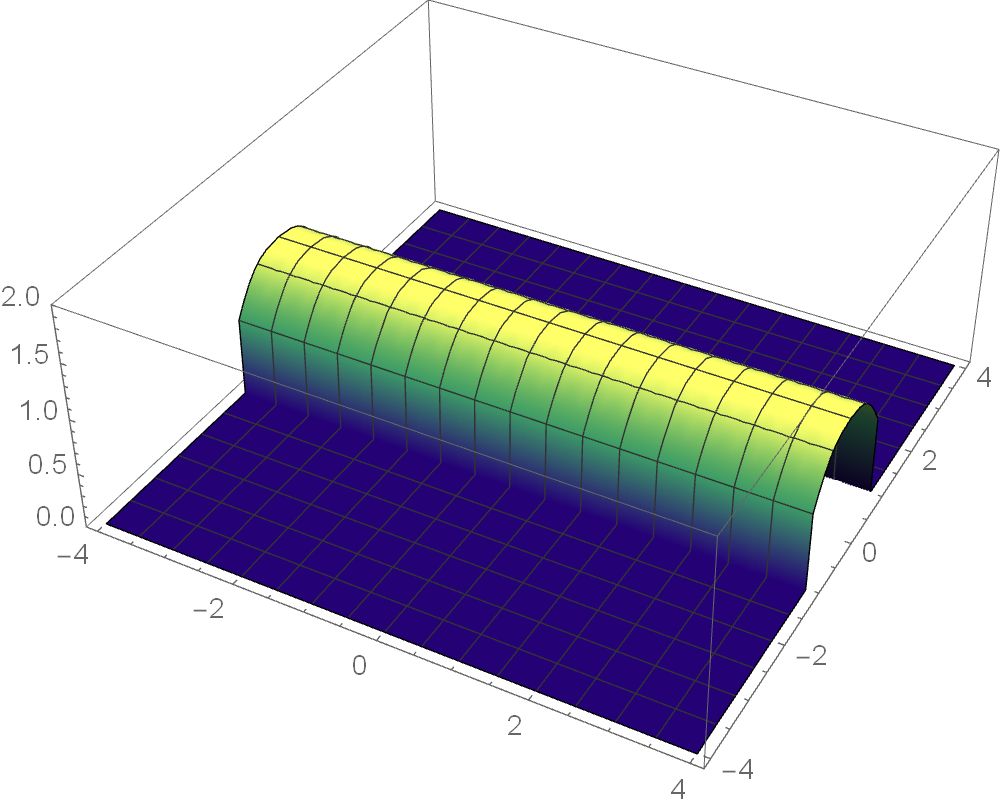}	
		\end{center}
		\caption{$c=1$}
	\end{subfigure}
	\begin{subfigure}[h]{0.32\textwidth}
		\begin{center}
			\includegraphics[width=1.69in,height=1.36in]{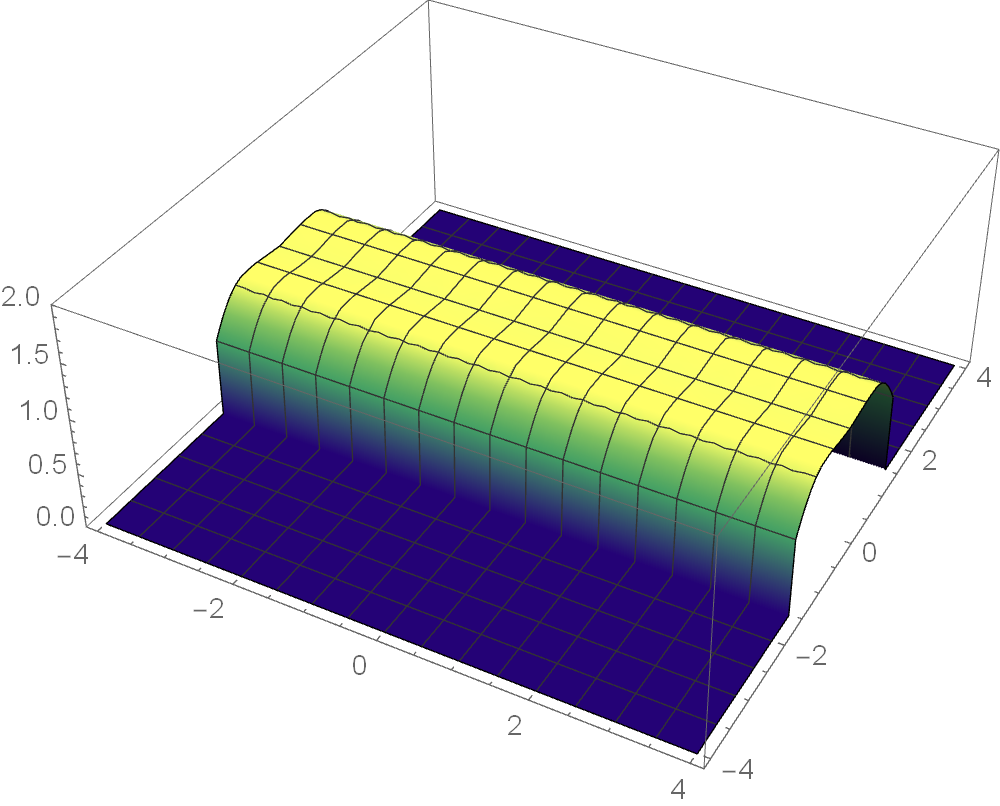}
		\end{center}
		\caption{$c=1.5$}
	\end{subfigure}	
	\begin{subfigure}[h]{0.32\textwidth}
		\begin{center}
			\includegraphics[width=1.69in,height=1.36in]{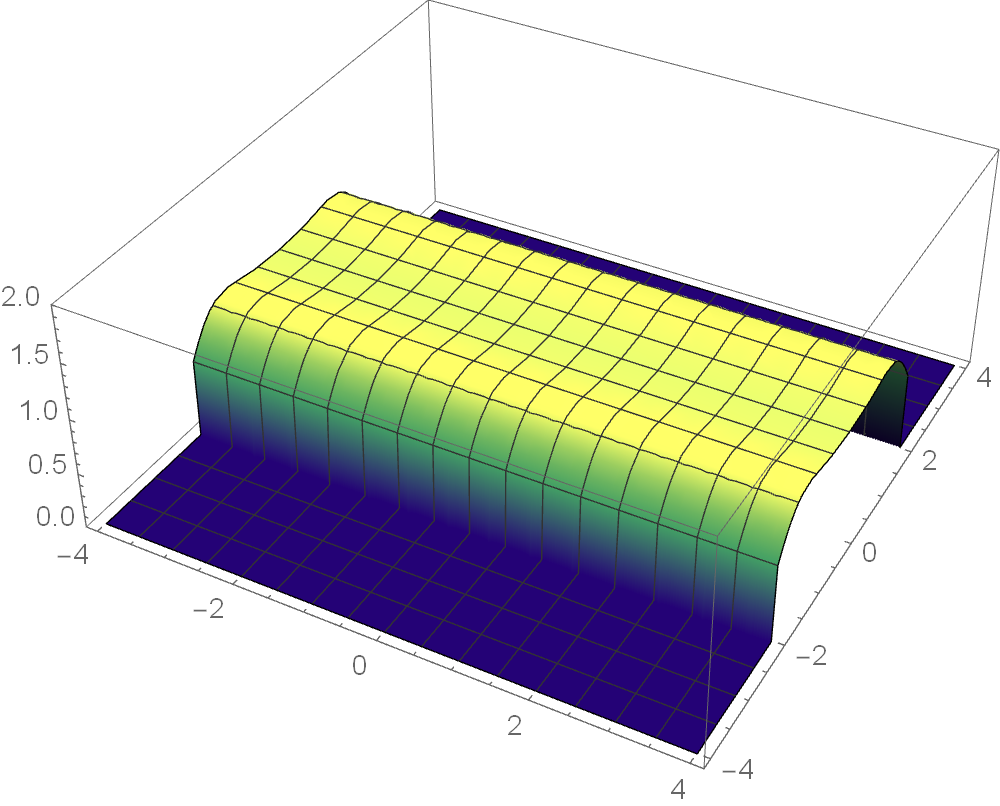}
		\end{center}
		\caption{$c=2$}
	\end{subfigure}	
	\caption{Graphs of natural candidates for limiting $1$-point density $R(z)$ about $p_*=0$.
	}
	\label{Fig_RH}
\end{figure}

\subsection{Further related works} Beyond the topics we study in this work, there have been several investigations on almost-Hermitian regime in various contexts. In \cite{ACV}, a fixed-trace version of
AGUE is studied, which brings about some new challenges since this ensemble is not determinantal.
The paper \cite{J2} studies
the interpolation problem between the Tracy-Widom and
Gumbel distributions.
An almost-Hermitian analogue of the Jacobi unitary ensemble
was recently discovered in \cite{NAKP}.
The bulk statistics for almost-Hermitian Gaussian ensembles in orthogonal and symplectic symmetry classes are found in \cite{FN} and \cite{Kan} respectively, and edge statistics is studied in \cite{AP}. Analogous problems for Laguerre ensembles are investigated in \cite{Ake05,APS}. These works, which construct $2\times 2$ matrix-valued kernels of Pfaffian point processes, use explicit computations with skew-orthogonal polynomials.

We also want to mention the paper \cite{FK}, where a different kind of almost-Hermitian model is introduced, motivated by applications to resonance statistics in quantum chaotic scattering. This model is non-determinantal, but in \cite{FK} it is argued that one can still define determinantal universality limits given by certain explicit kernels. Some rigorous results in this direction were obtained by Kozhan in \cite{Koz}. (We thank Yan Fyodorov for this remark.)

We finally want to mention that in dimensions 1 and 2, various microscopic limits in the determinantal case $\beta=2$ have been used to extract information about other $\beta$-ensembles, and specifically, they were used to study ``freezing transitions'', which occur for large $\beta$, in \cite{AMR,MR} and the references there. As far as we are aware, a rigorous analysis for almost Hermitian $\beta$-ensembles, relying on properties of various explicit kernels found above, has not yet been carried out, and could potentially be quite interesting.

\appendix

\section{Plancherel-Rotach asymptotics for Hermite polynomials} \label{AppHermite}

In this appendix we collect certain asymptotic formulas for Hermite polynomials which are used in our proofs of cross-section convergence in Subsection \ref{SCAGUE} and of translation invariance in Subsection \ref{BAGUE}. In both cases, we fix a point $p\in (-2,2)$, and we look at asymptotics for the scaled Hermite polynomial $H_n(\sqrt{\tfrac N 2}\zeta)$
where $n=N-1$ or $n=N$ and $\zeta$ is in a microscopic neighbourhood of $p$ (of radius proportional to $N^{-1/2}$) or on the vertical line
$p+i\R$. Recall that here $H_n(z)=(-1)^ne^{z^2}\tfrac {d^n}{d z^n}e^{-z^2}$ is the $n$:th ``physicists' Hermite polynomial'', which has leading coefficient $2^n$.

It turns out that our desired estimates can be deduced using strong asymptotic formulas of Plancherel-Rotach type, as given in the paper \cite{Wa}. These formulas look a little different depending on whether $-2<p<0$, $p=0$, or $0<p<2$.

The following partially overlapping asymptotic formulas are known.
By \cite[Corollary 4.1]{Wa}, we have \emph{for $z$ in a neighborhood of $p$, where $0<p<2$,}
\begin{align} \label{herna1}
	\begin{split}
	H_n(\sqrt{\tfrac{n}{2}}z)&=( \tfrac{2n}{e} )^{\frac{n}{2}} (4-z^2)^{-\frac{1}{4} } e^{ \frac{nz^2}{4} }
	\\
	&\times 2 \cos [ (n+\tfrac{1}{2}) \arccos \tfrac{z}{2}-\tfrac{nz}{4}\sqrt{4-z^2}-\tfrac{\pi}{4} ]\cdot(1+o(1)).
		\end{split}
\end{align}
For comparison, note that \cite{Wa} uses monic orthogonal polynomials $\pi_n$ rather than $H_n$, so $H_n=2^n\pi_n$. (There is a similar formula for $-2<p<0$ in \cite{Wa}, which we skip stating here.)

In a microscopic neighbourhood of the origin, we have instead the following Mehler-Heine formula (see \cite[Section 18.11]{OLBC})
\begin{align}\label{herna15}
	\lim_{n \to \infty} \tfrac{(-1)^n \sqrt{n}}{2^{2n} n!} H_{2n}(\tfrac{z}{2\sqrt{n}})=\tfrac{1}{\sqrt{\pi}} \cos z, \quad
	\lim_{n \to \infty} \tfrac{(-1)^n }{2^{2n} n!} H_{2n+1}(\tfrac{z}{2\sqrt{n}})=\tfrac{2}{\sqrt{\pi}} \sin z,
\end{align}
which holds uniformly for $z$ in any compact subset of $\C$.

Finally, for $z$ \emph{in the complement of a neighbourhood of the interval $[-2,2]$}, we have by \cite[Lemma 2.5]{LR} or \cite[Corollary 4.1]{Wa}
\begin{align} \label{herna2}
	\begin{split}
		H_n(\sqrt{\tfrac{n}{2}}z)&=( \tfrac{n}{2e} )^{ \frac{n}{2} } (z+\sqrt{z^2-4})^n (  \tfrac{z+\sqrt{z^2-4}}{2\sqrt{z^2-4}} )^{\frac12} e^{ \frac{n}{4} z(z-\sqrt{z^2-4}) } \cdot(1+o(1)).
	\end{split}
\end{align}

Using \eqref{herna1}, \eqref{herna15} and Stirling's formula, one easily obtains the following lemma.

\begin{lem} \label{herman} Suppose that $|p|\le \alpha$ where $\alpha<2$ and
fix $M>0$. Then for all $z$ with $|z|\le M$, we have
\begin{equation} \label{PR Hermite}
	H_N(\sqrt{\tfrac N 2}\cdot p+\tfrac z {\sqrt{N}}) =
e^{\frac N 4 p^2} 2^{\frac N 2} \sqrt{N!} \cdot N^{-\frac 1 4}\cdot O(1),\qquad (N\to\infty),
\end{equation}
where the $O(1)$-constant may depend on $M$ and $\alpha$.
\end{lem}

\subsection{Computation for cross-section} Again fix $p$ with $0<p<2$.

Recall from \eqref{juice} that the $\tfrac \d {\d \xi}$-derivative of the 1-point function with respect to potential \eqref{ellipse} obeys
\begin{align}
	\nonumber \tfrac 1 {N^2}\tfrac {\d \bfR_N} {\d \xi}(p+i\tfrac y N)=&-(1+o(1))\tfrac{1}{c\sqrt{2}}  e^{-\frac N 2 p^2-\frac 1 {2 c^{2}} y^2-2 c^2} \tfrac{ 1 }{2^N(N-1)!}
	\\
	&\times \Re \{  H_{N-1}(  \sqrt{\tfrac N 2}(p+i\tfrac{y }{N})  ) H_{N}(  \sqrt{\tfrac N 2}(p-i\tfrac{y }{N})  )  \}.
\end{align}

We use \eqref{herna1} to conclude that there is $\delta=\delta(p)>0$ and $k=k(\delta)>0$ such that when $|y|\le\delta N$, then
\begin{align*}
 H_{N}(  \sqrt{\tfrac N 2}(p-i\tfrac{y }{N})  )  &= ( \tfrac{2N}{e} )^{\frac{N}{2}} e^{\frac{N}{4}p^2 }
 \cdot \cosh (ky)\cdot O(1)
\end{align*}
and
\begin{align*}
	H_{N-1}(  \sqrt{\tfrac N 2}(p+i\tfrac{y }{N})  )&=( \tfrac{2N}{e} )^{\frac{N-1}{2}} e^{\frac{N}{4}p^2 }
\cdot \cosh (ky)\cdot O(1).
\end{align*}

Therefore there is a constant $C$ such that for all $|y|\le N\delta$,
$$
\tfrac 1 {N^2}\tfrac {\pa \bfR_N} {\pa \xi}(\xi+i\tfrac y N) \le C_1e^{ -\frac 1 {2c^2} y^2} \cosh^2(ky),
$$
where $C_1=C_1(C,k)$ is a new constant.

On the other hand, for $y$ such that $|y|/N > \delta$, letting  $z=p-i \tfrac{y}{N}$
\begin{align*}
\quad H_{N}(  \sqrt{\tfrac N 2}(p-i\tfrac{y }{N})  )H_{N-1}(  \sqrt{\tfrac N 2}(p+i\tfrac{y }{N})  )&\le C_2( \tfrac{2N}{e} )^{N-\frac12} |\tfrac{z+\sqrt{z^2-4}}{2}|^{2N-1} e^{ \frac{N}{2} \Re\,[z(z-\sqrt{z^2-4})] }\\
&\le C_3( \tfrac{2N}{e} )^{N-\frac12}(\tfrac {C_4 y} N)^{2N}e^{-\frac {y^2}{2N}}.
\end{align*}

Combining all of the above, we obtain that
\begin{equation}\label{f1p}\tfrac 1 {N^2}\tfrac {\d \bfR_N} {\pa \xi}(p+i\tfrac y N)\le Ce^{ -\frac 1 {2c^2} y^2 }\max\{1,y^{2N}N^{-N}\} \cosh^2(ky).
\end{equation}

This latter estimate \eqref{f1p} is used in Subsection \ref{SCAGUE} to deduce convergence of cross-sections of the AGUE.

\section{Plancherel-Rotach asymptotics for Laguerre polynomials} \label{AppLag}
In this appendix, we explain how to deduce certain asymptotic formulas for Laguerre polynomials, which were used in our proof of translation invariance of bulk scaling limits for the ALUE, in Subsection \ref{BALUE}. More precisely, we shall adapt to our present situation some Plancherel-Rotach type asymptotic formulas from Vanlessen's paper \cite{V}.

Fix a small $\delta>0$ and define a ``bulk region'' by
$$B_\delta=\{z\in\C\, ;\, \delta<\Re\,z< 1-\delta,\, -\delta <\Im\,z<\delta\}.$$
The following lemma is a special case of \cite[Theorem 2.4, (b)]{V}.

\begin{lem} \label{vanl} For $z\in B_\delta$ we have the following form of Plancherel-Rotach asymptotic for Laguerre polynomials, as $n\to\infty$.
\begin{align} \label{PR Lag z}
	\begin{split}
	L_n^\alpha(4n z)&=(4n z)^{-\frac \alpha 2} e^{2nz} (2\pi \sqrt{z(1-z)})^{-\frac 12} n^{-\frac 12} ( \tfrac{(n+\alpha)!}{n!} )^{\frac 12}
	\\
	&\times [ \cos( 2n \sqrt{z (1-z)}  -(2n+\alpha+1)\arccos\sqrt{z} +\tfrac{\pi}{4} )(1+O(1/n))
	\\
	&\quad +\cos( 2n \sqrt{z (1-z)}  -(2n+\alpha-1)\arccos\sqrt{z} +\tfrac{\pi}{4} ) O(1/n) ].
		\end{split}
\end{align}
\end{lem}

\begin{proof}
We explain how to see this formula from \cite[Theorem 2.4, (b)]{V}.
Consider the Laguerre-type weight
$$
w(x)=x^\alpha e^{-x}, \qquad x\in[0,\infty).
$$
We denote that the associated orthonormal polynomials (on $\R$) are given in terms of the Laguerre polynomials $L_n^\alpha$ by
\begin{equation}\label{p Lag}
p_n(x)=( \tfrac{n!}{(n+\alpha)!} )^{\frac 12}\, L_n^\alpha(x).
\end{equation}

Then by \cite[Theorem 2.4, (b)]{V} (in the special case $m=1$, $Q(x)=x$, $A_1=\tfrac 12$, $\beta^{(0)}=4$, $\beta_n=4n$) we have that for $z$ as above,
\begin{align*}
p_n(\beta_n z)&=(\beta_n z)^{-\alpha/2 } e^{ Q(\beta_n z)/2 } \sqrt{\tfrac{2}{\pi \beta_n}} \tfrac{1}{z^{1/4} (1-z)^{1/4}}
\\
&\times [ \cos( \tfrac12 (\alpha+1)\arccos(2z-1)-\pi n \int_{1}^{z} \psi_n (s)\,ds-\tfrac{\pi}{4} )(1+O(1/n))
\\
&\quad +\cos( \tfrac12 (\alpha-1)\arccos(2z-1)-\pi n \int_{1}^{z} \psi_n (s)\,ds-\tfrac{\pi}{4} ) O(1/n) ].
\end{align*}
Here the function $\int_1^z \psi_n(s)\,ds$ is explicitly computable (see \cite[Remark 2.5, 3.14, 3.15]{V}) and the result is
\begin{align*}
\int_1^z \psi_n(s)\,ds= \tfrac{1}{2\pi} H_n(z) \sqrt{z(1-z) } -\tfrac{2}{\pi} \arccos \sqrt{z},
\end{align*}
where $H_n$ is certain polynomial with real coefficients of degree $m-1$, i.e., it is a constant. (Indeed, $H_n$ is the polynomial appearing in the density
$
\tfrac{m}{2\pi} \sqrt{\tfrac{1-x}{x}}\,H_n(x)
$
of the associated equilibrium measure.)

In the special case when $Q(x)=x$, we have $H_n(x)=4$, (again see \cite[Remark 2.3]{V}) thus
$$
\int_1^z \psi_n(s)\,ds= \tfrac{2}{\pi} ( \sqrt{z(1-z)} - \arccos \sqrt{z} ) .
$$
Combining all of the above we obtain that when $Q(x)=x$,
\begin{align*}
	p_n(4n z)&=(4n z)^{-\frac{\alpha}{2}} e^{2nz} (2\pi \sqrt{z(1-z)} )^{-1/2} n^{-1/2}
	\\
	&\times [ \cos( \tfrac12 (\alpha+1)\arccos(2z-1)-2n ( \sqrt{z (1-z)} - \arccos \sqrt{z} ) -\tfrac{\pi}{4} )\cdot (1+O(1/n))
	\\
	&\quad +\cos( \tfrac12 (\alpha-1)\arccos(2z-1)-2n ( \sqrt{z (1-z)} - \arccos \sqrt{z} ) -\tfrac{\pi}{4} ) O(1/n) ].
\end{align*}
Note that since $\arccos(2z-1)=2\arccos \sqrt{z},$ we can simplify above expression as
\begin{align*}
	p_n(4n z)&=(4n z)^{-\alpha/2} e^{2nz} (2\pi \sqrt{z(1-z)} )^{-1/2} n^{-1/2}
	\\
	&\times [ \cos( 2n \sqrt{z (1-z)}  -(2n+\alpha+1)\arccos\sqrt{z} +\tfrac{\pi}{4} )\cdot (1+O(1/n))
	\\
	&\quad +\cos( 2n \sqrt{z (1-z)}  -(2n+\alpha-1)\arccos\sqrt{z} +\tfrac{\pi}{4} ) O(1/n) ].
\end{align*}
Now \eqref{PR Lag z} follows from this and the relation \eqref{p Lag}.
\end{proof}

\subsection{Computation for translation invariance} We now use Lemma \ref{vanl} to verify the following estimate, which is used
in our proof of translation invariance in Subsection \ref{BALUE}. We will use the notation in Subsection \ref{BALUE},
i.e., we fix a point $p_*\in I_\alpha=[\alpha,4-\alpha]$ where $\alpha>0$ is small and we write
$$\zeta(z)=p_*+\tfrac {2c\sqrt{p_*}}{N}z+o(\tfrac 1 N).$$
Moreover, we put $a=\tfrac N {c^2}$ and $b=\tfrac N {c^2}-1$.

\begin{lem}\label{BAAL} In the above setting, we have as $N\to\infty$,
\begin{align*}
	L_{ N+\nu-1 }^{ 1-\nu }(\tfrac{a^2-b^2}{2b}N\,\zeta(z) ) L_{N-1}^{\nu+1}(\tfrac{a^2-b^2}{2b}N\,\bar{\zeta}(z) ) = \tfrac{1}{N}\, e^{\frac{a^2-b^2}{2b}N \frac{\zeta(z)+ \bar{\zeta}(z)}{2}  }\cdot O(1).
\end{align*}
\end{lem}

\begin{proof}
Recall that we have for $p_*\in I_\alpha$ we have
$
\tfrac{a^2-b^2}{2b}N \zeta =Np_*+\tfrac{c^2}{2}p_*+2c\sqrt{p_*}\,z+O(N^{-1}).
$
Hence by \eqref{PR Lag z} together with the asymptotic
$( \tfrac{(n+\alpha)!}{n!} )^{\frac 12}=n^{\frac \alpha 2} \cdot (1+o(1)),$
we obtain
\begin{align}
	\begin{split}
		L_n^\alpha(4n z)&=(4n z)^{-\alpha/2} e^{2nz} (2\pi \sqrt{z(1-z)})^{-1/2} n^{\alpha/2-1/2} \cdot O(1)
		\\
		&= n^{-1/2} e^{2nz} \cdot O(1).
	\end{split}
\end{align}
Setting $n=N+\nu-1$ and then $n=N-1$ we obtain
\begin{align*}
L_{ N+\nu-1 }^{ 1-\nu }(\tfrac{a^2-b^2}{2b}N\,\zeta(z) ) &=N^{-1/2} e^{\frac{a^2-b^2}{2b}N \frac{\zeta(z)}{2}  } \cdot O(1)
\\
L_{N-1}^{\nu+1}(\tfrac{a^2-b^2}{2b}N\,\bar{\zeta}(z) ) &=N^{-1/2} e^{\frac{a^2-b^2}{2b}N \frac{\bar{\zeta}(z)}{2}  } \cdot O(1),
\end{align*}
which completes the computation.
\end{proof}

\subsection*{Data availability} Data sharing not applicable to this article as no datasets were generated or analysed during the current study.

\subsection*{Conflict of interest statement} We have no conflicting interests.

\end{document}